\documentclass{birkjour}
 \newtheorem{thm}{Theorem}[section]
 
 \newtheorem{cor}[thm]{Corollary}
 \newtheorem{lem}[thm]{Lemma}
 \newtheorem{prop}[thm]{Proposition}
 \theoremstyle{definition}
 \newtheorem{defn}[thm]{Definition}
 \theoremstyle{remark}
 \newtheorem{rem}[thm]{Remark}
 \newtheorem{examp}{Example}
 \numberwithin{equation}{section}

\usepackage{amsmath}
\usepackage{amssymb}
\usepackage{amsthm}
\usepackage{verbatim}
\usepackage{color}
\usepackage{geometry}
\usepackage{setspace}

\begin{document}

%
%
%
%
%
%
%
%
%

\title[]
 {Sub-diffusion processes in Hilbert space and their associated stochastic differential equations and Fokker-Planck-Kolmogorov equations}
\author{Lise Chlebak}
\address{
Tufts University, Department of Mathematics,\\ 
503 Boston Avenue, Medford, MA 02155, USA.}

\author{Patricia Garmirian}
\address{
Tufts University, Department of Mathematics,\\ 
	503 Boston Avenue, Medford, MA 02155, USA.}
\email{Patricia.Garmirian@tufts.edu}

\author{Qiong Wu}
\address{
Tufts University, Department of Mathematics,\\ 
	503 Boston Avenue, Medford, MA 02155, USA.}
\email{wuqionghit@gmail.com}




\begin{abstract}
	This paper focuses on the time-changed Q-Wiener process, a Hilbert space-valued sub-diffusion. It is a martingale with respect to an appropriate filtration, hence a stochastic integral with respect to it is definable. For the resulting integral, two change of variables formulas are derived.  Via a duality theorem for integrals, existence and uniqueness theorems for stochastic differential equations (SDEs) driven by the time-changed Q-Wiener process are discussed.  Associated fractional Fokker-Planck-Kolmogorov equations are derived using either a time-changed It\^o formula or duality.   Connections are established between three integrals driven by time-changed versions of the Q-Wiener process, cylindrical Wiener process, and martingale measure. 
\end{abstract}
\keywords{inverse stable subordinator, time-changed $Q$-Wiener processes, Hilbert space-valued sub-diffusion processes, time-changed stochastic differential equations (SDEs) in Hilbert space, fractional Fokker-Planck-Kolmogorov (FPK) equations}
\maketitle
\section{Introduction}
The $Q$-Wiener process is a stochastic process with values in a separable Hilbert space which is usually infinite dimensional. Analogous to a classic Brownian motion in finite dimensional space, the $Q$-Wiener process is a Gaussian diffusion process which has independent and stationary increments and a covariance operator $Q$~\cite{Prevot2007,Koyacs2008,Curtain1971}. Stochastic integrals in Hilbert space with respect to the $Q$-Wiener process are constructed as an infinite sum of real-valued stochastic integrals with respect to infinitely many independent standard Brownian motions \cite{Prato2014,Mandrecar2010,Pardoux2007,Karczewska2005}. A Hilbert space of functions is a natural setting for the semigroup approach to deal with stochastic partial differential equations (SPDEs) since the solutions to these equations are typically elements of such a space~\cite{Ichikawa1982, Ichikawa1984,Prato1996}. For motivation and details on the semigroup approach, see~\cite{Pazy2012,Butzer2013}. From the monographs~\cite{Prevot2007,Prato2014,Mandrecar2010}, stochastic differential equations (SDEs) driven by the $Q$-Wiener process in Hilbert space are expressed as
\begin{eqnarray}\label{SDEinHilbert}
\left\{
\begin{array}{ll}
dX(t) = [AX(t) + F(t, X(t))]\mathrm{d}t + C\mathrm{d}W_{t} \\
\\
X(0) = x_{0}\in{H},\\
\end{array} 
\right.
\end{eqnarray}
where the operator, $A$, is typically an elliptic differential operator and $W_{t}$ is the $Q$-Wiener process. This type of SDE usually corresponds to stochastic partial differential equations (SPDEs) and the solution, $X(t)$, of the SDE~\eqref{SDEinHilbert} is also a diffusion process in Hilbert space. Semigroups and distributions at a fixed time associated with the diffusion processes in Hilbert spaces are studied in \cite{Manita2015} and \cite{Leha1984}. Furthermore, deterministic Fokker-Planck-Kolmogorov (FPK) equations  corresponding to the diffusion process, $X(t)$, in Hilbert space together with the existence and uniqueness of their solutions have been investigated in \cite{Bogachev2009,Bogachev2010,Bogachev2011}. Specifically, Bogachev and others have started the study of FPK equations on Hilbert space in~\cite{Bogachev2009,Bogachev2011} and~\cite{Bogachev2010} continues this study by proving existence and uniqueness results for irregular, even non-continuous drift coefficients.

Sub-diffusion processes in a finite dimensional space are investigated by Meerschaert and Scheffler \cite{Meerschaert2004,Meerschaert2008} and Magdziarz \cite{Magdziarz2010}. In particular, the sub-diffusion processes arising as the scaling limit of continuous-time random walks are considered in \cite{Meerschaert2004,Meerschaert2008}. The sample path properties of this type of process are investigated by the martingale approach in \cite{Magdziarz2010}. It is known that a Brownian motion with an embedded time-change which is the first hitting time process of a stable subordinator of index between $0$ and $1$ is a sub-diffusion process~\cite{Meerschaert2013}. Since the stable subordinator process is an increasing L\'evy process with jumps, there are intervals for which the inverse process is constant. Therefore, a time-changed Brownian motion process often models phenomena for which there are intervals where the process does not change~\cite{Hurd2009,Kawai2010,Umarov2014}. The transition probabilities, i.e., the densities of the time-changed Brownian motion, satisfy a time-fractional FPK equation appearing in~\cite{Hahn2012}. Furthermore, the stochastic calculus for a time-changed semimartingale and the associated SDEs driven by the time-changed semimartingale are investigated in \cite{Kei2011}. See references~\cite{Hahn201156, Hahn2011150, Hahn2011139} as well as the forthcoming book in~\cite{Hahn2016} for further discussion and generalizations of the topics mentioned earlier in this paragraph. 

Corresponding sub-diffusion processes in Hilbert space, specifically, the time-changed Q -Wiener processes have not been investigated. Also the SDEs driven by a time-changed Q -Wiener process and their associated fractional FPK equations have not yet been discussed. In this paper, we focus on the time-changed $Q$-Wiener process and its associated stochastic calculus. Also, SDEs driven by the time-changed $Q$-Wiener process and fractional FPK equations of the solutions to the time-changed SDEs are investigated. Specifically, section~\ref{timechangedQwienerprocess} introduces the concept of a time-changed $Q$-Wiener process which is a sub-diffusion process. Similar to the time-changed Brownian motion, the time-changed $Q$-Wiener process is proven to be a square-integrable martingale in Hilbert space with respect to an appropriate filtration. Furthermore, the increasing process and the quadratic variation process of the time-changed $Q$-Wiener process are explicitly derived.

Section~\ref{stointegralwithtimechangedQwiener} develops the stochastic integral with respect to the time-changed $Q$-Wiener process. First and second change of variable formulas for the time-changed stochastic integral are provided. Also the time-changed It\^o formula for an It\^o process driven by the time-changed $Q$-Wiener process is developed. Moreover, SDEs driven by the time-changed $Q$-Wiener process are introduced.  Based on results from Grecksch and Tudor~\cite{Grecksch1995} for general martingales in Hilbert spaces, the existence and uniqueness of a mild solution to a type of time-changed SDE is provided. From the duality developed between time-changed SDEs and their corresponding non-time-changed SDEs, the existence and uniqueness of strong solutions to the associated time-changed SDEs are established. 

Section~\ref{connectionbetweenWalshandHilbert} discusses a connection between three stochastic integrals, the one driven by a time-changed $Q$-Wiener process in Hilbert space, the one driven by a time-changed cylindrical Wiener process in Hilbert space and the one driven by the martingale measure, known as Walsh's integral~\cite{Walsh1986,Karczewska2005}. This connection allows one to potentially interpret mild solutions to SDEs driven by time-changed $Q$-Wiener processes in Hilbert spaces as random field solutions from the Walsh approach to SPDEs by applying the results of Dalang and Quer-Sardanyons \cite{Dalang2011}.

Section~\ref{fpeinhilbertspace} derives time-fractional FPK equations associated with the solutions to SDEs driven by time-changed $Q$-Wiener processes in two ways. The first method uses the time-changed It\^o formula developed in Section~\ref{stointegralwithtimechangedQwiener}. The same time-fractional FPK equations are next derived by applying the duality theorem also developed in Section~\ref{stointegralwithtimechangedQwiener} and the FPK equations associated with the solution of SDEs driven by the classic $Q$-Wiener process.

\section{The time-changed $Q$-Wiener process}\label{timechangedQwienerprocess} 
After recalling the definition of the $Q$-Wiener process in Hilbert space, we introduce the time-changed $Q$-Wiener process in Hilbert space. The related martingale properties of the time-changed $Q$-Wiener process will also be investigated.
\begin{defn} Following \cite{Mandrecar2010}\label{Qwiener},
Let $Q$ be a nonnegative definite, symmetric, trace-class operator on a separable Hilbert space $K$, let  $\{f_{j}\}_{j=1}^{\infty}$  be an orthonormal basis in $K$ diagonalizing $Q$, and let the corresponding eigenvalues be $\{\lambda_{j}\}_{j=1}^{\infty}$.
Let $\{w_{j}(t)\}_{t\geq 0}, j=1, 2, \cdots,$ be a sequence of independent Brownian motions defined on a filtered probability space $(\Omega, \mathcal{F}, \{\mathcal{F}_{t}\}_{t \geq 0}, \mathbb{P})$. Then the process
	\begin{eqnarray*}
		W_{t} := \sum_{j=1}^{\infty}\lambda_{j}^{1/2}w_{j}(t)f_{j}
	\end{eqnarray*}
	is called a $Q$-Wiener process in $K$.
\end{defn}
\noindent For more details of properties of $Q$-Wiener processes, see \cite{Prato2014, Mandrecar2010}.

Before introducing the time-changed $Q$-Wiener processes, we need to introduce the time-change process which is used throughout this paper. The time-change applied in this paper is the first hitting time process of a $\beta$-stable subordinator defined as 
\begin{equation}\label{inversesubordinator}
E_{t} := E_{\beta}(t) = \inf\{\tau >0: U_{\beta}(\tau) > t\},
\end{equation}
where $U_{\beta}(t)$ is the $\beta$-stable subordinator which has index $\beta\in(0,1)$ and Laplace transform
\begin{eqnarray}\label{laplacetransformofsubordinator}
\mathbb{E}(e^{-uU_{\beta}(\tau)}) = e^{-\tau u^{\beta}}.
\end{eqnarray}
Note that $E_{t}$ is also called the inverse $\beta$-stable subordinator. Let $B(t)$ denote a one-dimensional standard Brownian motion. Consider $Z_{\beta}(t):= B(E_{t})$, a subordinated Brownian motion which has been time-changed by $E_{t}$. The following result of Magdziarz \cite{Magdziarz2010} shows that $Z_{\beta}(t)$ is a square integrable martingale with respect to the appropriate right-continuous filtration,
\begin{equation}\label{filtrationbytimechange}
\bar{\mathcal{F}}_{t} = \bigcap_{u>t}\{\sigma[B(s): 0\leq s \leq u]\vee\sigma[E_{s}: s\geq 0]\},
\end{equation}
where $\bar{\mathcal{F}}_0$ is assumed to be complete.
\begin{thm} Magdziarz, \cite{Magdziarz2010}\label{timechangedBrownianmotion}
The time-changed Brownian motion $Z_{\beta}(t)$ is a mean zero and square integrable martingale with respect to the filtration $\{\bar{\mathcal{F}}_{E_t}\}_{t\geq 0}$. The quadratic variation process of $Z_{\beta}(t)$ is $\langle Z_{\beta}(t), Z_{\beta}(t)\rangle = E_{t}$.
\end{thm}
The time-changed Brownian motion, $Z_{\beta}(t)$, is a sub-diffusion process. By incorporating the time-change, $E_{t}$, into the independent Brownian motions in Definition~\ref{Qwiener}, we define a Hilbert space-valued time-changed $Q$-Wiener process as follows: 
\begin{defn}\label{timechangedqwiener}
	With $Q$, $\{f_{j}\}_{j=1}^{\infty}$, and $\{\lambda_{j}\}_{j=1}^{\infty}$ as defined in \ref{Qwiener}, let $\{w_{j}(t)\}_{t\geq 0}, j=1, 2, \cdots,$ be a sequence of independent Brownian motions defined on $(\Omega, \mathcal{F}, \{\mathcal{F}_{t}\}_{t \geq 0}, \mathbb{P})$ which are independent of $E_t$. Then the process
\begin{eqnarray}\label{defoftimechangedQwiener}
	W_{E_{t}} := \sum_{j=1}^{\infty}\lambda_{j}^{1/2}w_{j}(E_{t})f_{j}
\end{eqnarray}
is called a time-changed $Q$-Wiener process in $K$.
\end{defn}
This time-changed $Q$-Wiener process can be considered as an infinite-dimensional Hilbert space analog of a one-dimensional time-changed Brownian motion. It is a sub-diffusion process in Hilbert space. Let $\mu_{t}$ be the Borel probability measure induced by the time-changed $Q$-Wiener process, $W_{E_{t}}$ on $K$, i.e., $\mathbb{E}(W_{E_{t}}) = \int_{K}x\mu_{t}(\mathrm{d}x)$. Then the time-fractional FPK equation corresponding to the time-changed $Q$-Wiener process in the following theorem can be considered as a special case of Theorem~\ref{difffracfpe} in Section~\ref{fpeinhilbertspace}.
\begin{thm}\label{sub-diffusionoftimechangedQwienerprocess}
Suppose $\mu_{t}$ is the probability measure induced by the time-changed $Q$-Wiener process, $W_{E_{}}$, on $K$. Then $\mu_{t}$ satisfies the following time-fractional PDE
\begin{eqnarray*}
	D_{t}^{\beta}\mu_{t} = D_{x}^{2}\mu_{t},
\end{eqnarray*}
where $D_{x}^{2}$ denotes the second-order Fr\'echet derivative in space and $D_{t}^{\beta}$ is the Caputo time fractional derivative operator defined as
\begin{eqnarray*}
	D^{\beta}_{t}f(t) = \frac{1}{\Gamma(1 - \beta)}\int_{0}^{t}\frac{f^{\prime}(\tau)}{(t - \tau)^{\beta}}\mathrm{d}\tau,
\end{eqnarray*}
where $\beta$ is the index associated with the $\beta$-stable subordinator $U_{\beta}(t)$ and $\Gamma(\beta)$ is the gamma function.
\end{thm}
The time-changed $Q$-Wiener process is our main object of study. In order to define integrals with respect to this process, and, ultimately, to consider SDEs driven by this process, it is advantageous to view the time-changed $Q$-Wiener process as a martingale with respect to an appropriate filtration. To prove this, we begin with the definition of a martingale in a Hilbert space.

\begin{defn}\label{H-valuedmartingale}Following \cite{Mandrecar2010},
let $K$ be a separable Hilbert space endowed with its Borel $\sigma$-field $\mathcal{B}(K)$. Fix $T>0$ and let $(\Omega, \mathcal{F}, \{\mathcal{F}_{t}\}_{t\leq T}, \mathbb{P})$ be a filtered probability space and $\{M_{t}\}_{t\leq T}$ be a $K$-valued process adapted to the filtration $\{\mathcal{F}_{t}\}_{t\leq T}$. Assume that $M_{t}$ is integrable, i.e., $\mathbb E\|M_{t}\| < \infty$. 
\begin{itemize}
	\item[a)] If for any $0\leq s \leq t$, $\mathbb E(M_{t}|\mathcal{F}_{s}) = M_{s},~\mathbb{P}$-a.s., then $M_{t}$ is called an $\mathcal{F}_t$-martingale.
	\item[b)] If $\mathbb{E}\|M_{T}\|^{2} < \infty$, the martingale $M_{t}$ is called square integrable on~$0 \leq t \leq T$.
\end{itemize}
\end{defn}

The condition for $\{M_{t}\}_{t < T}$ to be a martingale in Hilbert space is equivalent to the following: for all $s < t$ and $h\in K$,
\begin{eqnarray*}
	\mathbb{E}(\langle M_{t}, h\rangle|\mathcal{F}_{s}) = \langle M_{s}, h\rangle,~\mathbb{P}-a.s.
\end{eqnarray*}
Let $\mathcal{M}^2_{T}(K)$ denote the collection of all continuously square integrable martingales in the Hilbert space $K$.

\begin{examp}
	The $Q$-Wiener process $W_{t} := \sum_{j=1}^{\infty}\lambda_{j}^{1/2}w_{j}(t)f_{j}$ introduced in Definition~\ref{Qwiener} is a square integrable martingale with respect to the filtration $\mathcal{F}_{t}$ generated by independent Brownian motions, i.e., $\sigma(w_{j}(s): 0 \leq s \leq t, j = 1, 2, \cdots)$.
\end{examp}

Define 
 \begin{equation}\label{generalizationoffiltration}
 \mathcal{\tilde{F}}_{t} = \bigcap_{u>t}\{\sigma[w_j(s): 0\leq s \leq u,j=1,2,...]\vee\sigma[E_{s}: s\geq 0]\},
 \end{equation}
where $\tilde{\mathcal{F}}_0$ is assumed to be complete, and $\{w_{j}(t)\}_{t\geq 0}, j=1, 2, \cdots,$ is a sequence of independent Brownian motions which are independent of $E_{t}$.
The following theorem establishes that the time-changed $Q$-Wiener process is also a square-integrable martingale, i.e., $W_{E_t}\in\mathcal{M}^2_{T}(K)$ with respect to the filtration $\mathcal{\tilde{F}}_{E_t}.$

\begin{thm}\label{martingaleoftimechagedQwiener}
	The time-changed $Q$-Wiener process defined by Definition \ref{timechangedqwiener} is a $K$-valued square integrable martingale with respect to the filtration $\mathcal{G}_{t}:=\tilde{\mathcal{F}} _{E_t}$.
\end{thm}
\begin{proof}
From \cite{Magdziarz2010}, all moments of the time change, $E_{t}$, are finite, i.e.,
\begin{eqnarray*}
\mathbb{E}(E_{t}^n)=\frac{t^{n\beta}n!}{\Gamma(n\beta+1)},
\end{eqnarray*}
for $n=1,2, \cdots$. Let $f_{E_{t}}$ be the density function for $E_t$. Then, the second moment for the time-changed Brownian motion $w_{j}(E_t)$ is given by
\begin{eqnarray*}
	\begin{aligned}
	  \mathbb{E}(w_{j}^2(E_t))=\int_{0}^{\infty}\mathbb{E}(w_{j}^2(\tau))f_{E_t}(\tau)\mathrm{d}\tau=\int_{0}^{\infty}\tau f_{E_t}(\tau)\mathrm{d}\tau =\frac{t^{\beta}}{\Gamma(\beta+1)}.
	\end{aligned}
\end{eqnarray*}
Thus, 
\begin{eqnarray*}
	\begin{aligned}
	  \mathbb{E}\|W_{E_{t}}\|^{2}_{K} &= \mathbb{E}\bigg\langle \sum_{j=1}^{\infty}\lambda^{1/2}_{j}w_{j}(E_{t})f_{j}, \sum_{i=1}^{\infty}\lambda^{1/2}_{i}w_{i}(E_{t})f_{i}\bigg\rangle_{K}=\mathbb E\sum_{j=1}^{\infty}\lambda_{j}w^{2}_{j}(E_{t})\\
	  &=\sum_{j=1}^{\infty}\lambda_{j}\mathbb E(w_{j}^{2}(E_{t}))=\frac{t^{\beta}}{\Gamma(\beta + 1)}\sum_{j=1}^{\infty}\lambda_{j}<\infty.
	\end{aligned}
\end{eqnarray*}
The sum is finite since $Q$ is a trace-class operator. Also, since the $Q$-Wiener process, $W_{t}$, is a square integrable martingale in the Hilbert space $K$, it follows that for any $h\in K$, the process $X_{t}$ defined by 
\begin{eqnarray*}
	X_{t} :=\langle W_{t},h\rangle_K
\end{eqnarray*}
is a real-valued square integrable martingale with respect to the filtration $\mathcal{\tilde{F}}_{t}$. This means that in order to prove the time-changed $Q$-Wiener process, $W_{E_{t}}$, is a square integrable martingale, it suffices to verify that the time-changed real-valued process, $X_{E_{t}}$, defined by
\begin{eqnarray*}
  X_{E_t}:=\langle W_{E_t},h\rangle_K,
\end{eqnarray*}
is a square integrable martingale with respect to the filtration  $\mathcal G_t=\mathcal{\tilde{F}}_{E_t}$. 
Define the sequence of $\{\mathcal{\tilde{F}}_\tau\}$-stopping times, $T_n$, by
\begin{eqnarray*}
	T_n=\inf\{\tau>0:|X(\tau)|\geq n\}.
\end{eqnarray*}
It is known that the stopped process $X(T_n\wedge \tau)$ is a bounded martingale with respect to $\tilde{\mathcal{F}}_{\tau}$. Thus, by Doob's Optional Sampling Theorem, for $s<t$,
\begin{eqnarray}\label{doobsampling}
	\mathbb{E}(X(T_n\wedge E_t)\ |\ \mathcal G_s)=X(T_n\wedge E_s)
\end{eqnarray}
The right hand side of~\eqref{doobsampling} converges to $X(E_s)$ as $n\rightarrow \infty$.
For the left hand side,
\begin{eqnarray*}
	|X(T_n\wedge E_t)|\leq\sup_{0\leq s\leq t}|X(E_s)|.
\end{eqnarray*}
Then, applying Doob's Maximal Inequality yields
\begin{eqnarray*}
	\begin{aligned}
		\mathbb{E}(\sup_{0\leq s \leq t}&|X(E_{s})|^2) = \mathbb{E}(\sup_{0\leq s \leq E_{t}}|X(s)|^2) = \int_{0}^{\infty}\mathbb{E}(\sup_{0\leq s \leq \tau}|X(s)|^2\mid E_{t} = \tau)f_{E_{t}}(\tau)\mathrm{d}\tau\\
		&= \int_{0}^{\infty}\mathbb{E}(\sup_{0\leq s \leq \tau}|X(s)|^2)f_{E_{t}}(\tau)\mathrm{d}\tau\leq 4\int_{0}^{\infty}\mathbb{E}(|X(\tau)|^2)f_{E_{t}}(\tau)\mathrm{d}\tau\\
		&= 4\int_{0}^{\infty}\mathbb{E}(|\langle W_{\tau},h\rangle_K|^2)f_{E_{t}}(\tau)\mathrm{d}\tau \leq  4\int_{0}^{\infty}\mathbb{E}(\|W_{\tau}\|_K^2\|h\|_K^2)f_{E_{t}}(\tau)\mathrm{d}\tau\\
		& =  4\|h\|_K^2\int_{0}^{\infty}\mathbb E \|W_\tau\|^2f_{E_t}(\tau)d\tau  =  4\|h\|_K^2\mathbb E\|W_{E_t}\|^2_K<\infty.\\
	\end{aligned}
\end{eqnarray*}
Also by Holder's inequality,
\begin{eqnarray*}
  \mathbb{E}(\sup_{0\leq s\leq t}|X_{E_{s}}|) \leq (\mathbb{E}(\sup_{0\leq s \leq t}|X_{E_{s}}|)^2)^{1/2}.
\end{eqnarray*} 
Thus,
\begin{eqnarray*}
   \mathbb{E}(\sup_{0\leq s\leq t}|X_{E_{s}}|) < \infty.
\end{eqnarray*}
By the dominated convergence theorem,
\begin{eqnarray*}
	\mathbb{E}(X(T_{n}\wedge E_{t})\mid \mathcal{G}_{s}) \longrightarrow \mathbb{E}(X(E_{t})\mid\mathcal{G}_{s}),~\textrm{as}~n\to\infty.
\end{eqnarray*}
Therefore, from~\eqref{doobsampling},
\begin{eqnarray*}
	\mathbb{E}(X(E_{t})\mid\mathcal{G}_{s}) = X(E_{s}),
\end{eqnarray*}
		which implies that $X(E_{t})=\langle W_{E_t},h \rangle_K$ is a martingale with respect to the filtration $\mathcal{\tilde{F}}_{E_t}$. Therefore, $W_{E_t}$ is a square integrable martingale in the Hilbert space $K$.
\end{proof}
The following definition and lemma concern the increasing processes and the quadratic variation processes of square integrable martingales in a Hilbert space. These will be applied in Proposition~\ref{increasingandquadracticprocess} to the time-changed $Q$-Wiener process $W_{E_t}$.  

\begin{defn} Following \cite{Mandrecar2010}\label{defincreasingandquadraticprocesses},
	let $M_{t}\in\mathcal{M}^2_{T}(K)$. Denote by $\langle M\rangle_{t}$ the unique adapted continuous increasing process starting from 0 such that $\|M_{t}\|^2_{K} - \langle M\rangle_{t}$ is a continuous martingale. The quadratic variation process $\langle\langle M\rangle\rangle_{t}$ of $M_{t}$ is an adapted continuous process starting from 0, with values in the space of nonnegative definite trace-class operators on $K$, such that for all $h, g\in K$,
	$$\langle M_{t}, h\rangle_{K}\langle M_{t}, g\rangle_{K} - \big\langle\langle\langle M\rangle\rangle_{t}(h), g\big\rangle_K$$
	is a martingale.
\end{defn}

\begin{lem} Following \cite{Mandrecar2010}\label{relationbetweenincreasingandquadraticprocess},
	the quadratic variation process of a martingale $M_{t}\in\mathcal{M}_{T}^{2}(K)$ exists and is unique. Moreover, 
	\begin{eqnarray*}
		\langle M\rangle_{t} = tr(\langle\langle M\rangle\rangle_{t}).
	\end{eqnarray*}
\end{lem}

\begin{prop}\label{increasingandquadracticprocess}
	The increasing process and quadratic variation process of the time-changed $Q$-Wiener process in Definition~\ref{timechangedqwiener} are respectively
	\begin{eqnarray*}
		\langle W_{E}\rangle_{t} = tr(Q)E_{t}~\textrm{and}~\langle\langle W_{E}\rangle\rangle_{t} = QE_{t}.
	\end{eqnarray*}
\end{prop}
\begin{proof}
	Let $Q$ be a nonnegative definite, symmetric, trace-class operator on a separable Hilbert space $K$ and let $\{f_{j}\}_{j=1}^{\infty}$  be an orthonormal basis in $K$ diagonalizing $Q$ with corresponding eigenvalues $\{\lambda_{j}\}_{j=1}^{\infty}$. Then, 
\begin{eqnarray*}
	\|W_{E_{t}}\|^2_{K} &=\bigg\langle \sum_{j=1}^{\infty}\lambda_{j}^{1/2}w_{j}(E_{t})f_{j}, \sum_{i=1}^{\infty}\lambda_{i}^{1/2}w_{i}(E_{t})f_{i}\bigg\rangle_{K} = \sum_{j=1}^{\infty}\lambda_{j}w^2_{j}(E_{t}).
\end{eqnarray*}
	On the other hand,
\begin{eqnarray*}
    tr(Q)E_{t} =  E_{t}\sum^{\infty}_{j=1}\lambda_{j}.
\end{eqnarray*}
So define the process, $N_{E_{t}}$, as
\begin{eqnarray*}
	N_{E_{t}} = \|W_{E_{t}}\|^2_{K} - tr(Q)E_{t} = \sum_{j=1}^{\infty}\lambda_{j}(w_{j}^2(E_{t}) - E_{t}),
\end{eqnarray*}
which can be considered as a time-change of $N_{t}$ where
\begin{eqnarray*}
	N_{t} = \|W_{t}\|^2_{K} - tr(Q)t = \sum_{j=1}^{\infty}\lambda_{j}(w_{j}^2(t) - t).
\end{eqnarray*}
From Definition~\ref{defincreasingandquadraticprocesses}, $N_{t}$ is a real-valued martingale since $tr(Q)t$ is the unique increasing process of the $Q$-Wiener process. By an argument similar to that of Theorem~\ref{martingaleoftimechagedQwiener}, the time-changed process, $N_{E_{t}}$, is a martingale. Further, $W_{E_{t}}$ is a martingale and there is an increasing process, $\langle W_{E}\rangle_{t}$, such that $\|W_{E_{t}}\|_{K}^{2} - \langle W_{E}\rangle_{t}$ is a martingale. Finally, since $\|W_{E_{t}}\|^{2}_{K}$ is a real-valued submartingale, by the uniqueness of the Doob-Meyer decomposition~\cite{Protter2013}, 
\begin{eqnarray}\label{incrasingprocessproof}
\langle W_{E}\rangle_{t} = tr(Q)E_{t}.
\end{eqnarray}
Again, by Theorem~\ref{martingaleoftimechagedQwiener} and Lemma~\ref{relationbetweenincreasingandquadraticprocess}, the quadratic process, $\langle\langle W_{E}\rangle\rangle_{t}$, of the time-changed $Q$-Wiener process, $W_{E_{t}}$, exists and is unique, and satisfies
\begin{eqnarray}\label{relationbetweenincrasingandquadraticprocessproof}
	tr\big(\langle\langle W_{E}\rangle\rangle_{t}\big) = \langle W_{E}\rangle_{t}.
\end{eqnarray}
Therefore, from~\eqref{incrasingprocessproof} and~\eqref{relationbetweenincrasingandquadraticprocessproof}, 
\begin{eqnarray*}
	\langle\langle W_{E}\rangle\rangle_{t} = QE_{t}.
\end{eqnarray*}
\end{proof}
\section{SDEs driven by the time-changed Q-Wiener process}\label{stointegralwithtimechangedQwiener}
In this section, we begin by developing the It\^o stochastic integral with respect to the time-changed $Q$-Wiener process in Hilbert space. Also a time-changed It\^o formula for an It\^o process driven by a time-changed $Q$-Wiener process is developed. Finally, the existence and uniqueness of solutions to the time-changed Hilbert space-valued SDEs are investigated.

\subsection{Stochastic integral with respect to the time-changed $Q$-Wiener process}
In order to construct an It\^o stochastic integral with respect to the time-changed $Q$-Wiener process, we briefly recall It\^o stochastic integrals with respect to a Q-Wiener process without a time change as in~\cite{Prato2014,Mandrecar2010}.

As in Section \ref{timechangedQwienerprocess}, let $K$ and $H$ be two separable Hilbert spaces, and $Q$ be a symmetric, nonnegative definite trace-class operator on $K$. Let $\{f_{j}\}_{j=1}^{\infty}$ be an orthonormal basis (ONB) in $K$ such that $Qf_{j} = \lambda_{j}f_{j}$, where these eigenvalues $\lambda_{j} > 0$ for all $j = 1, 2, \cdots$. Then the separable Hilbert space $K_{Q} = Q^{1/2}K$ with an ONB $\{\lambda_{j}^{1/2}f_{j}\}_{j=1}^{\infty}$ is endowed with the following scalar product 
\begin{eqnarray*}
	\langle u, v\rangle_{K_{Q}} = \sum_{j=1}^{\infty}\frac{1}{\lambda_{j}}\langle u, f_{j}\rangle_{K}\langle v, f_{j}\rangle_{K}.
\end{eqnarray*}
Let $\mathcal{L}_{2}(K_{Q}, H)$ be the space of Hilbert-Schmidt operators from $K_{Q}$ to $H$. The Hilbert-Schmidt norm of an operator $L\in\mathcal{L}_{2}(K_{Q}, H)$ is given by 
\begin{eqnarray*}
	\|L\|_{\mathcal{L}_{2}(K_{Q}, H)}^{2} = \|LQ^{1/2}\|_{\mathcal{L}_{2}(K, H)}^{2} = tr((LQ^{1/2})(LQ^{1/2})^{*}).
\end{eqnarray*}
The scalar product between two operators $L, M\in\mathcal{L}_{2}(K_{Q}, H)$ is defined by 
\begin{eqnarray*}
	\langle L, M\rangle_{\mathcal{L}_{2}(K_{Q}, H)} = tr((LQ^{1/2})(MQ^{1/2})^{*}).
\end{eqnarray*}
Define $\Lambda_{2}(K_{Q},H)$ as the class of $\mathcal L_{2}(K_{Q},H)$-valued processes which are measurable mappings from 
\[
([0,T]\times \Omega,\mathcal B([0,T])\times \mathcal F)
\]
to 
\[
(\mathcal L_{2}(K_{Q},H),\mathcal B(\mathcal L_{2}(K_{Q},H))),
\]
adapted to the filtration $\{\mathcal F_{t}\}_{t\leq T}$,
and satisfying the condition
\begin{eqnarray*}
	\mathbb{E}\int_{0}^{T}\|\Phi(t)\|^{2}_{\mathcal L_{2}(K_{Q},H)}dt<\infty.
\end{eqnarray*}
Note that $\Lambda_{2}(K_{Q},H)$ is a Hilbert space if it is equipped with the norm
\begin{eqnarray*}
	\|\Phi\|_{\Lambda_{2}(K_{Q},H)}:=\left(\mathbb E\int_{0}^{T}\|\Phi(t)\|^{2}_{\mathcal L_{2}(K_{Q},H)}dt\right)^{1/2}.
\end{eqnarray*}
The following lemma from~\cite{Mandrecar2010} can be considered as a definition of a stochastic integral with respect to the $Q$-Wiener process:
\begin{lem} Mandrekar, \cite{Mandrecar2010}\label{stochintwithqwiener}
	Let $W_{t}$ be a Q-Wiener process in a separable Hilbert space $K$, $\Phi\in\Lambda_{2}(K_{Q}, H)$, and $\{f_{j}\}_{j=1}^{\infty}$ be an ONB in $K$ consisting of eigenvectors of $Q$. Then, 
	\begin{eqnarray*}
		\int_{0}^{t}\Phi(s)\mathrm{d}W_{s} = \sum_{j=1}^{\infty}\int_{0}^{t}(\Phi(s)\lambda_{j}^{1/2}f_{j})\mathrm{d}\langle W_{s},  \lambda_{j}^{1/2}f_{j}\rangle_{K_{Q}}.
	\end{eqnarray*}
\end{lem}
In order to incorporate the time-change, $E_{t}$, into the It\^o stochastic integral, the generalized $\tilde{\Lambda}_{2}(K_{Q}, H)$ is also considered as the class of $\mathcal L_{2}(K_{Q},H)$-valued processes which are measurable mappings from 
\[
([0,T]\times \Omega,\mathcal B([0,T])\times \tilde{\mathcal{F}}_{t})
\]
 to 
\[
(\mathcal L_{2}(K_{Q},H),\mathcal B(\mathcal L_{2}(K_{Q},H))), 
\]
adapted to the filtration $\{\tilde{\mathcal{F}}_{E_{t}}\}_{t\leq T}$, and satisfying the condition
\begin{eqnarray*}
	\mathbb{E}\int_{0}^{T}\|\Phi(t)\|^{2}_{\mathcal L_{2}(K_{Q},H)}dE_{t}<\infty.
\end{eqnarray*}
Similarly,  $\tilde{\Lambda}_{2}(K_{Q},H)$ is a separable Hilbert space if it is equipped with the norm
\[
\|\Phi\|_{\tilde{\Lambda}_{2}(K_{Q},H)}:=\left(\mathbb E\int_{0}^{T}\|\Phi(t)\|^{2}_{\mathcal L_{2}(K_{Q},H)}dE_{t}\right)^{1/2}.
\]
Thus, It\^o stochastic integrals with respect to the time-changed $Q$-Wiener process can be introduced.
\begin{defn}\label{timechangstochinte}
	Let $W_{E_{t}}$ be a time-changed $Q$-Wiener process in a separable Hilbert space $K$, $\Phi\in\tilde{\Lambda}_{2}(K_{Q}, H)$, and let $\{f_{j}\}_{j=1}^{\infty}$ be an ONB in $K$ consisting of eigenvectors of $Q$. Then, 
\begin{eqnarray*}
		\int_{0}^{t}\Phi(s)\mathrm{d}W_{E_{s}} = \sum_{j=1}^{\infty}\int_{0}^{t}(\Phi(s)\lambda_{j}^{1/2}f_{j})\mathrm{d}\langle W_{E_{s}},  \lambda_{j}^{1/2}f_{j}\rangle_{K_{Q}}.
\end{eqnarray*}
\end{defn}
Now that the It\^o integral with respect to the time-changed $Q$-Wiener process has been established, the next step is to derive the It\^o isometry first for elementary processes, and then by extension, for arbitrary processes in $\tilde{\Lambda}_{2}(K_{Q}, H)$.
Consider the class of $\{\mathcal{G}_{t}\}$-adapted elementary processes of the form
\begin{eqnarray}\label{elementaryprocess}
	\Phi(t,\omega)=\phi(\omega)1_{\{0\}}(t)+\sum_{j=0}^{n-1}\phi_j(\omega)1_{(t_j,t_{j+1}]}(t),
\end{eqnarray}
where $0\leq t_0 \leq t_1 \leq \cdots \leq t_n=T$ and $\phi,\phi_j$, $j=0,1,\cdots,n-1$ are respectively $\mathcal{G}_{0}$-measurable and $\mathcal{G}_{t_j}$-measurable $\mathcal{L}_2(K_Q, H)$-valued random variables such that $\phi(\omega),\phi_j(\omega)$ are linear, bounded operators from $K$ to $H$.  Let $\mathcal{E}(\mathcal{L}(K,H))$ denote this class of elementary processes. Proceeding to the It\^o isometry for an elementary process $\Phi(t, \omega)$, we need the following useful lemma.
\begin{lem}\label{crossproducttermofmartingale}
Let $\{f_{j}\}_{j=1}^{\infty}$ be an ONB in $K$ consisting of eigenvectors of $Q$ and $\mathcal{G}_{s} = \tilde{\mathcal{F}}_{E_{s}}$ be the filtration.	Then, for $l\neq l^{\prime}$ and $t > s > 0$,
	\begin{eqnarray*}
	  \mathbb{E}\bigg(\mathbb{E}\bigg(\bigg\langle W_{E_t}-W_{E_s},f_{\l}\bigg\rangle_K \bigg\langle W_{E_t}-W_{E_s},f_{\l'}\bigg\rangle_K \bigg| \mathcal{G}_s\bigg)\bigg)=0.
	\end{eqnarray*}
\end{lem}
\begin{proof}
	From Definition~\ref{timechangedqwiener} in Section~\ref{timechangedQwienerprocess},
\begin{eqnarray*}
	W_{E_{t}} = \sum_{j=1}^{\infty}\lambda_{j}^{1/2}w_{j}(E_{t})f_{j}.
\end{eqnarray*}
Using the definition of $W_{E_{t}}$,
\begin{eqnarray*}
   \begin{aligned}
		&\mathbb{E}\bigg(\mathbb{E}\bigg(\bigg\langle W_{E_t}-W_{E_s},f_{\l}\bigg\rangle_K \bigg\langle W_{E_t}-W_{E_s},f_{\l'}\bigg\rangle_K \bigg| \mathcal{G}_s\bigg)\bigg)\\
		&= \mathbb{E}\bigg\{\mathbb{E}\bigg(\bigg\langle W_{E_t},f_{\l}\bigg\rangle_K \bigg\langle W_{E_t},f_{\l'}\bigg\rangle_K - \bigg\langle W_{E_s},f_{\l}\bigg\rangle_K \bigg\langle W_{E_t},f_{\l'}\bigg\rangle_K\\ 
		&\quad- \bigg\langle W_{E_s},f_{\l'}\bigg\rangle_K \bigg\langle W_{E_t},f_{\l}\bigg\rangle_K
		+ \bigg\langle W_{E_s},f_{\l}\bigg\rangle_K \bigg\langle W_{E_s},f_{\l'}\bigg\rangle_K \bigg| \mathcal{G}_s\bigg)\bigg\}\\
		&= \mathbb{E}\bigg(\mathbb{E}\bigg(\lambda_{l}^{1/2} w_{l}(E_t) \lambda_{\l'}^{1/2} w_{\l'}(E_t)|\mathcal{G}_{s}\bigg)\bigg) - \mathbb{E}\bigg(\mathbb{E}\bigg(\lambda_{\l}^{1/2} w_{\l}(E_s) \lambda_{\l'}^{1/2} w_{\l'}(E_t) | \mathcal{G}_s\bigg)\bigg)\\
		&\quad- \mathbb{E}\bigg(\mathbb{E}\bigg(\lambda_{\l'}^{1/2} w_{\l'}(E_s) \lambda_{\l}^{1/2} w_{\l}(E_t)| \mathcal{G}_s\bigg)\bigg)
		+ \mathbb{E}\bigg(\lambda_{\l}^{1/2} w_{\l}(E_s) \lambda_{\l'}^{1/2} w_{\l'}(E_s)\bigg)\\
		&:= I_{1} - I_{2} - I_{3} + I_{4}.
   \end{aligned}
\end{eqnarray*}
Since $w_{l}$ is independent of $w_{l^{\prime}}$, conditioning on $E_{t}$ to compute the first term yields
\begin{eqnarray*}
	\begin{aligned}
		I_{1} &= \mathbb{E}\bigg(\mathbb{E}\bigg(\lambda_{l}^{1/2} w_{l}(E_t) \lambda_{\l'}^{1/2} w_{\l'}(E_t)|\mathcal{G}_{s}\bigg)\bigg) =\mathbb{E}(\lambda_{\l}^{1/2} w_{\l}(E_t) \lambda_{\l'}^{1/2} w_{\l'}(E_t))\\ 
		&= \lambda_{\l}^{1/2}\lambda_{\l'}^{1/2} \int_0^{\infty} \mathbb{E}(w_{\l}(\tau)w_{\l'}(\tau)) f_{E_t}(\tau) \mathrm{d}\tau = \lambda_{\l}^{1/2}\lambda_{\l'}^{1/2} \int_0^{\infty} 0 \cdot f_{E_t}(\tau) \mathrm{d}\tau=0.
	\end{aligned}
\end{eqnarray*}
On the other hand,
\begin{eqnarray*}
\begin{aligned}
I_{2} &= \mathbb{E}\bigg(\mathbb{E}\bigg(\lambda_{\l}^{1/2} w_{\l}(E_s) \lambda_{\l'}^{1/2} w_{\l'}(E_t) \bigg| \mathcal{G}_s\bigg)\bigg) = \lambda_{\l}^{1/2} \lambda_{\l'}^{1/2} \mathbb{E}\bigg(\mathbb{E}\bigg(w_{\l}(E_s)w_{\l'}(E_t) \bigg| \mathcal{G}_s\bigg)\bigg)\\
&= \lambda_{\l}^{1/2} \lambda_{\l'}^{1/2} \mathbb{E}\bigg(\mathbb{E}\bigg(w_{\l}(E_s)\bigg(w_{\l'}(E_t)-w_{\l'}(E_s)\bigg)+w_{\l}(E_s)w_{\l'}(E_s) \bigg| \mathcal{G}_s\bigg)\bigg)\\
&= \lambda_{\l}^{1/2} \lambda_{\l'}^{1/2} \mathbb{E}\bigg(w_{\l}(E_s)\mathbb{E}\bigg(w_{\l'}(E_t)-w_{\l'}(E_s) \bigg| \mathcal{G}_s\bigg)+w_{\l'}(E_s)w_{\l}(E_s)\bigg)\\
&=0,
\end{aligned}
\end{eqnarray*}
since $\mathbb{E}(w_{\l'}(E_t)-w_{\l'}(E_s) | \mathcal{G}_s)=0$ by the martingale property of $W_{E_t}$ and $\mathbb{E}(w_{\l'}(E_s)w_{\l}(E_s))=0$ by the same conditioning argument previously used in computing term $I_{1}$. Similarly, the third term, $I_{3}$, and the fourth term, $I_{4}$, are also equal to 0. 
\end{proof}
\begin{thm}\label{itoisometryofelementaryprocesses}
Let $\Phi\in\mathcal{E}(\mathcal{L}(K, H))$ be a bounded elementary process. Then, for $t\in[0, T]$, 
\begin{eqnarray*}
	\mathbb{E}\bigg\|\int_0^t \Phi(s)dW_{E_s}\bigg\|_H^2= \mathbb{E}\int_0^t \big\|\Phi(s)\big\|_{\mathcal{L}_2(K_Q, H)}^2 dE_s < \infty.
\end{eqnarray*}
\end{thm}
\begin{proof}
   First, without loss of generality, assume that $t=T$.  Then, for the bounded elementary process,~$\Phi$, defined in~\eqref{elementaryprocess},
 \[
 \mathbb{E} \bigg\|\int_0^T \Phi(s)dW_{E_s}\bigg\|_H^2 = \mathbb{E}\bigg\|\sum_{j=0}^{n-1}\phi_j(W_{E_{t_{j+1}}}-W_{E_{t_j}})\bigg\|_H^2\]
\[
= \sum_{j=0}^{n-1}\mathbb{E}\bigg\|\phi_j(W_{E_{t_{j+1}}}-W_{E_{t_j}})\bigg\|_H^2
\]	
\[
+ \sum_{i\neq j=0}^{n-1} \mathbb{E}\bigg\langle \phi_j(W_{E_{t_{j+1}}}-W_{E_{t_j}}), \phi_i(W_{E_{t_{i+1}}}-W_{E_{t_i}})\bigg\rangle_H := I + II.
\]
Let $\{e_m\}_{m=1}^{\infty}$ in $H$ and $\{f_{\l}\}_{\l=1}^{\infty}$ in $K$ be ONBs. For fixed $j$, $I_{j}$ is denoted by
\begin{eqnarray*}
	\begin{aligned}
		&I_{j} = \mathbb{E}\bigg\|\phi_j(W_{E_{t_{j+1}}}-W_{E_{t_j}})\bigg\|_H^2= \mathbb{E}\sum_{m=1}^{\infty} \bigg\langle \phi_j(W_{E_{t_{j+1}}}-W_{E_{t_j}}),e_m \bigg\rangle_H^2\\
		&= \sum_{m=1}^{\infty} \mathbb{E}\bigg(\mathbb{E}\big(\langle \phi_j(W_{E_{t_{j+1}}}-W_{E_{t_j}}),e_m \rangle_H^2 | \mathcal{G}_{t_j}\big)\bigg)\\
		&= \sum_{m=1}^{\infty} \mathbb{E}\bigg(\mathbb{E}\bigg(\bigg\langle W_{E_{t_{j+1}}}-W_{E_{t_j}},\phi_j^*e_m \bigg\rangle_K^2\bigg | \mathcal{G}_{t_j}\bigg)\bigg)\\
		&= \sum_{m=1}^{\infty} \mathbb{E}\bigg(\mathbb{E}\bigg(\bigg(\sum_{\l=1}^{\infty}\bigg\langle W_{E_{t_{j+1}}}-W_{E_{t_j}},f_{\l}\bigg\rangle_K \bigg\langle \phi_j^*e_m, f_{\l} \bigg\rangle_K\bigg)^2 \bigg| \mathcal{G}_{t_j}\bigg)\bigg)\\
		&= \sum_{m=1}^{\infty} \mathbb{E}\bigg(\mathbb{E}\bigg(\bigg(\sum_{\l=1}^{\infty}\bigg\langle W_{E_{t_{j+1}}}-W_{E_{t_j}},f_{\l}\bigg\rangle_K^2 \bigg\langle \phi_j^*e_m, f_{\l} \bigg\rangle_K^2\bigg) \bigg| \mathcal{G}_{t_j}\bigg)\bigg)\\
		&\quad + \sum_{m=1}^{\infty} \mathbb{E}\bigg(\mathbb{E}\bigg(\bigg(\sum_{\l\neq \l'=1}^{\infty}\langle W_{E_{t_{j+1}}}-W_{E_{t_j}},f_{\l}\rangle_K \langle \phi_j^*e_m, f_{\l} \rangle_K \\
		&\quad\times \langle W_{E_{t_{j+1}}}-W_{E_{t_j}},f_{\l'}\rangle_K \langle \phi_j^*e_m, f_{\l'} \rangle_K\bigg)\bigg| \mathcal{G}_{t_j}\bigg)\bigg)\\
		&:= J_{1} + J_{2}.
	\end{aligned}
\end{eqnarray*}
Since $\phi_j^*$ is $\mathcal{G}_{t_j}$-measurable and $W_{E_{t_{j+1}}}$ is a discrete martingale with respect to $\mathcal{G}_{t_{j}}$, the first term,~$J_{1}$, becomes
\begin{eqnarray*}
\begin{aligned}
   J_{1} &= \sum_{m=1}^{\infty} \mathbb{E}\bigg(\sum_{\l=1}^{\infty}\bigg\langle \phi_j^*e_m, f_{\l} \bigg\rangle_K^2 \mathbb{E}\bigg(\bigg\langle W_{E_{t_{j+1}}}-W_{E_{t_j}},f_{\l}\bigg\rangle_K^2 \bigg| \mathcal{G}_{t_j}\bigg)\bigg)\\
   &=\sum_{m=1}^{\infty} \mathbb{E}\bigg\{\sum_{\l=1}^{\infty}\bigg\langle \phi_j^*e_m, f_{\l} \bigg\rangle_K^2 \bigg(\mathbb{E}\bigg(\bigg\langle W_{E_{t_{j+1}}}, f_{\l} \bigg\rangle_K^2 \bigg| \mathcal{G}_{t_j}\bigg) - \bigg\langle  W_{E_{t_j}},f_{\l}\bigg\rangle_K^2\bigg)\bigg\}\\
   &=\sum_{m=1}^{\infty} \mathbb{E}\bigg\{\sum_{\l=1}^{\infty}\bigg\langle \phi_j^*e_m, f_{\l} \bigg\rangle_K^2 \bigg(\mathbb{E}\bigg(\bigg\langle W_{E_{t_{j+1}}}, f_{\l} \bigg\rangle_K^2 \bigg| \mathcal{G}_{t_j}\bigg)\bigg\}\\ 
   &\quad - \sum_{m=1}^{\infty} \mathbb{E}\bigg\{\sum_{\l=1}^{\infty}\bigg\langle \phi_j^*e_m, f_{\l} \bigg\rangle_K^2\bigg\langle  W_{E_{t_j}},f_{\l}\bigg\rangle_K^2\bigg)\bigg\}\\
   &=\mathbb{E}\bigg\{\sum_{m=1}^{\infty} \bigg(E_{t_{j+1}}-E_{t_j}\bigg) \sum_{\l=1}^{\infty} \lambda_{\l} \bigg\langle \phi_j^*e_m, f_{\l} \bigg\rangle_K^2\bigg\}\\
   &=\mathbb{E}\bigg\{\bigg(E_{t_{j+1}}-E_{t_j}\bigg) \sum_{m,\l=1}^{\infty} \bigg\langle \phi_j(\lambda_{\l}^{1/2}f_{\l}), e_m \bigg\rangle_H^2\bigg\}\\
   &=\mathbb{E}\bigg\{\bigg(E_{t_{j+1}}-E_{t_j}\bigg) \|\phi_j\|^2_{\mathcal{L}_2(K_Q,H)}\bigg\}.
		\end{aligned}
\end{eqnarray*}	
Also using the $\mathcal{G}_{t_j}$-measurability of $\phi_j^*$ and Lemma~\ref{crossproducttermofmartingale}, the second term, $J_{2}$, becomes
\begin{eqnarray*}
	\begin{aligned}
		J_{2} &= \sum_{m=1}^{\infty} \mathbb{E}\bigg\{\sum_{\l\neq \l'=1}^{\infty} \bigg\langle \phi_j^*e_m, f_{\l} \bigg\rangle_K \bigg\langle \phi_j^*e_m, f_{\l'} \bigg\rangle_K \\
		&\quad\times\mathbb{E}\bigg(\bigg\langle W_{E_{t_{j+1}}}-W_{E_{t_j}},f_{\l}\bigg\rangle_K  \bigg\langle W_{E_{t_{j+1}}}-W_{E_{t_j}},f_{\l'}\bigg\rangle_K  \bigg| \mathcal{G}_{t_j}\bigg)\bigg\}=0.
	\end{aligned}
\end{eqnarray*}
Thus, 
\[
I = \sum_{j = 0}^{n-1}I_{j} = \sum_{j = 0}^{n - 1} \mathbb{E}\bigg\{\bigg(E_{t_{j+1}}-E_{t_j}\bigg) \|\phi_j\|^2_{\mathcal{L}_2(K_Q,H)}\bigg\} 
\]
\[
= \mathbb{E}\int_0^T \big\|\Phi(s)\big\|_{\mathcal{L}_2(K_Q, H)}^2 dE_s < \infty.
\]
On the other hand, without loss of generality, assume that $i< j$. From Lemma~\ref{crossproducttermofmartingale},
\begin{eqnarray*}
	\begin{aligned}
	II &=\mathbb{E} \sum_{m=1}^{\infty} \mathbb{E}\bigg( \sum_{\l,\l'=1}^{\infty} \bigg\langle W_{E_{t_{j+1}}}- W_{E_{t_{j}}}, f_{\l} \bigg\rangle_K \bigg\langle \phi_j^*e_m, f_{\l} \bigg\rangle_K \\
	&\quad\times\bigg\langle W_{E_{t_{i+1}}}- W_{E_{t_{i}}}, f_{\l'} \bigg\rangle_K \bigg\langle \phi_i^*e_m, f_{\l'} \bigg\rangle_K \bigg| \mathcal{G}_{t_j}\bigg)=0.
	\end{aligned}
\end{eqnarray*}
\end{proof}
\begin{thm}(Time-changed It\^o Isometry)
	For $t\in[0, T]$, the stochastic integral $\Phi\mapsto\int_{0}^{t}\Phi(s)\mathrm{d}W_{E_{s}}$ with respect to a $K$-valued time-changed $Q$-Wiener process $W_{E_{t}}$ is an isometry between $\tilde{\Lambda}_{2}(K_{Q}, H)$ and the space of continuous square-integrable martingales $\mathcal{M}_{T}^{2}(H)$, i.e.,
	\begin{eqnarray}\label{generalitoisometryinsection3}
		\mathbb{E}\bigg\|\int_{0}^{t}\Phi(s)\mathrm{d}W_{E_{s}}\bigg\|_{H}^{2} = \mathbb{E}\int_{0}^{t}\big\|\Phi(s)\big\|^{2}_{\mathcal{L}_{2}(K_{Q}, H)}\mathrm{d}E_{s} < \infty.
	\end{eqnarray} 
\end{thm}
\begin{proof}
For elementary processes $\Phi\in\mathcal{E}(\mathcal{L}(K,H)$, Theorem~\ref{itoisometryofelementaryprocesses} establishes the desired equality~\eqref{generalitoisometryinsection3} and consequently the square-integrability of the integral $\displaystyle\int_{0}^{t}\Phi(s)\mathrm{d}W_{E_{s}}$. Furthermore, since the time-changed Q-Wiener process, $W_{E_{t}}$, is a K-valued martingale, for any $h\in H$ and $s < t$,
\[
	\mathbb{E}\bigg(\bigg\langle \int_0^t \Phi(r) dW_{E_r}, h \bigg\rangle_H \bigg| \mathcal{G}_s\bigg)
	= \mathbb{E}\bigg(\bigg\langle \sum_{j=0}^{n-1} \phi_j(W_{E_{t_{j+1}\wedge t}}-W_{E_{t_{j}\wedge t}}), h \bigg\rangle_H \bigg| \mathcal{G}_s\bigg)
\]
\[
= \sum_{j=0}^{n-1} \mathbb{E}\bigg(\bigg\langle W_{E_{t_{j+1}\wedge t}}-W_{E_{t_{j}\wedge t}}, \phi^*(h) \bigg\rangle_K \bigg| \mathcal{G}_s\bigg) = \sum_{j=0}^{n-1} \bigg\langle W_{E_{t_{j+1}\wedge s}}-W_{E_{t_{j}\wedge s}}, \phi^*(h) \bigg\rangle_K
\]
\[
	= \bigg\langle \sum_{j=0}^{n-1} \phi(W_{E_{t_{j+1}\wedge s}}-W_{E_{t_{j}\wedge s}}), h \bigg\rangle_H = \bigg\langle \int_0^s \Phi(s) dW_{E_s}, h \bigg\rangle_H,
\]
which implies that the stochastic integral $\int_{0}^{t}\Phi(s)\mathrm{d}W_{E_{s}}$ is a square-integrable martingale. Therefore, the desired result holds when $\Phi(s)$ is an elementary process.

Now, let $\{\Phi_n\}_{n=1}^{\infty}$ be a sequence of elementary processes approximating $\Phi \in \tilde{\Lambda}_2(K_Q, H)$.  Assume that $\Phi_1=0$ and $||\Phi_{n+1}-\Phi_n||_{\tilde{\Lambda}_2(K_Q, H)} < \frac{1}{2^n}$. Then, by Doob's Maximal Inequality,
\[
		\sum_{n=1}^{\infty} \mathbb{P} \left(\sup_{t\leq T} \left\lVert \int_0^t \Phi_{n+1}(s) dW_{E_s} - \int_0^t \Phi_{n}(s) dW_{E_s} \right\rVert_H > \frac{1}{n^2} \right)
\]
\[
		\leq \sum_{n=1}^{\infty} n^4 \mathbb{E} \left\lVert \int_0^T (\Phi_{n+1}(s)- \Phi_{n}(s)) dW_{E_s} \right\rVert^2_H 
\]
\[
= \sum_{n=1}^{\infty} n^4 \mathbb{E} \int_0^T ||\Phi_{n+1}(s)- \Phi_{n}(s)||^2_{\mathcal{L}(K_Q,H)} dE_s \leq \frac{T^{\beta}}{\Gamma(\beta+1)}\sum_{n=1}^{\infty} \frac{n^4}{2^n} < \infty. 
\]
By Borel-Cantelli, it follows that for some $k(\omega) > 0$,
\begin{eqnarray*}
	\sup_{t\leq T} \left\lVert \int_0^t \Phi_{n+1}(s) dW_{E_s} - \int_0^t \Phi_{n}(s) dW_{E_s} \right\rVert_H \leq \frac{1}{n^2},~~~ n>k(\omega),
\end{eqnarray*}
holds $\mathbb{P}$-almost surely. Therefore, for every $t\leq T$,
\begin{eqnarray*}
	\sum_{n=1}^{\infty} \left( \int_0^t \Phi_{n+1}(s) dW_{E_s} - \int_0^t \Phi_{n}(s) dW_{E_s} \right)\to \int_0^t \Phi(s) dW_{E_s}~\textrm{in}~L^2(\Omega, H),
\end{eqnarray*}
which also converges $\mathbb{P}$-a.s. to a continuous version of the integral.

Thus, the map $\displaystyle\Phi \mapsto \int_0^t \Phi(s) dW_{E_s}$, viewed as an isometry from elementary processes  to the space of continuous square-integrable martingales, has an extension to $\Phi\in\tilde{\Lambda}_2(K_Q, H)$ by the completeness property of $H$.  
\end{proof}

The following two change of variable formulas concern the It\^o stochastic integral related to the time-change $E_{t}$. They are needed later and can be considered as the Hilbert space extensions of formulas in~\cite{Kei2011}.
\begin{thm}(\textbf{1st change of variable formula in Hilbert space})\label{1stchangeformula}
Let $W_{t}$ be a $Q$-Wiener process in a separable Hilbert space $K$, $\Phi\in\tilde{\Lambda}_{2}(K_{Q}, H)$, and $E_{t}$ be the inverse of a $\beta$-stable subordinator. Then, with probability one, for all $t \geq 0$,  
\begin{eqnarray*}
	\int_{0}^{E_{t}}\Phi(s)\mathrm{d}W_{s} = \int_{0}^{t}\Phi(E_{s})\mathrm{d}W_{E_{s}}.
\end{eqnarray*}
\end{thm}
\begin{proof}
	Let $\{f_{j}\}_{j=1}^{\infty}$ be an ONB in the separable Hilbert space $K$ consisting of eigenvectors of $Q$. Then, it follows from Lemma~\ref{stochintwithqwiener} that
	\begin{eqnarray*}
		\int_{0}^{E_{t}}\Phi(s)\mathrm{d}W_{s} = \sum_{j=1}^{\infty}\int_{0}^{E_{t}}(\Phi(s)\lambda_{j}^{1/2}f_{j})\mathrm{d}\langle W_{s},  \lambda_{j}^{1/2}f_{j}\rangle_{K_{Q}}.
	\end{eqnarray*}
	For any $h\in H$,
	\begin{eqnarray*}
		\begin{aligned}
			\bigg\langle \int_{0}^{E_{t}}\Phi(s)\mathrm{d}W_{s}, h\bigg\rangle_{H} &= \bigg\langle \sum_{j=1}^{\infty}\int_{0}^{E_{t}}(\Phi(s)\lambda_{j}^{1/2}f_{j})\mathrm{d}\langle W_{s},  \lambda_{j}^{1/2}f_{j}\rangle_{K_{Q}}, h\bigg\rangle_{H}\\
			&=  \sum_{j=1}^{\infty}\int_{0}^{E_{t}}\langle (\Phi(s)\lambda_{j}^{1/2}f_{j}), h\rangle_{H}\mathrm{d}\langle W_{s},  \lambda_{j}^{1/2}f_{j}\rangle_{K_{Q}}\\
			&= \sum_{j=1}^{\infty}\int_{0}^{t}\langle (\Phi(E_{s})\lambda_{j}^{1/2}f_{j}), h\rangle_{H}\mathrm{d}\langle W_{E_{s}},  \lambda_{j}^{1/2}f_{j}\rangle_{K_{Q}}\\
			&= \bigg\langle\sum_{j=1}^{\infty}\int_{0}^{t} (\Phi(E_{s})\lambda_{j}^{1/2}f_{j}) \mathrm{d}\langle W_{E_{s}},  \lambda_{j}^{1/2}f_{j}\rangle_{K_{Q}}, h\bigg\rangle_{H}\\
			&= \bigg\langle \int_{0}^{t}\Phi(E_{s})\mathrm{d}W_{E_{s}}, h\bigg\rangle_{H}.
		\end{aligned}
	\end{eqnarray*}
The third equality follows from the first change of variable formula of the real-valued stochastic integral from~\cite{Kei2011}.
\end{proof}

\begin{thm}(\textbf{2nd change of variable formula in Hilbert space})\label{2ndcnhageformula}
	Let $W_{t}$ be a Q-Wiener process in a separable Hilbert space $K$ and $\Phi\in\tilde{\Lambda}_{2}(K_{Q}, H)$. Let $U_{t}$ be a $\beta$-stable subordinator with $\beta\in(0, 1)$ and $E_{t}$ be its inverse stable subordinator. Then, with probability one, for all $t \geq 0$,  
	\begin{eqnarray*}
		\int_{0}^{t}\Phi(s)\mathrm{d}W_{E_{s}} = \int_{0}^{E_{t}}\Phi(U_{s-})\mathrm{d}W_{s}.
	\end{eqnarray*}
\end{thm}
\begin{proof}
	Let $\{f_{j}\}_{j=1}^{\infty}$ be an ONB in the separable Hilbert space $K$ consisting of eigenvectors of $Q$. Applying Definition~\ref{timechangstochinte} yields,
	\begin{eqnarray*}
		\int_{0}^{t}\Phi(s)\mathrm{d}W_{E_{s}} = \sum_{j=1}^{\infty}\int_{0}^{t}(\Phi(s)\lambda_{j}^{1/2}f_{j})\mathrm{d}\langle W_{E_{s}},  \lambda_{j}^{1/2}f_{j}\rangle_{K_{Q}}.
	\end{eqnarray*}
	For any $h\in H$, 
	\begin{eqnarray*}
		\begin{aligned}
			\bigg\langle \int_{0}^{t}\Phi(s)\mathrm{d}W_{E_{s}}, h\bigg\rangle_{H} &= \bigg\langle \sum_{j=1}^{\infty}\int_{0}^{t}(\Phi(s)\lambda_{j}^{1/2}f_{j})\mathrm{d}\langle W_{E_{s}},  \lambda_{j}^{1/2}f_{j}\rangle_{K_{Q}}, h\bigg\rangle_{H}\\
			&=  \sum_{j=1}^{\infty}\int_{0}^{t}\langle (\Phi(s)\lambda_{j}^{1/2}f_{j}), h\rangle_{H}\mathrm{d}\langle W_{E_{s}},  \lambda_{j}^{1/2}f_{j}\rangle_{K_{Q}}\\
			&= \sum_{j=1}^{\infty}\int_{0}^{E_{t}}\langle (\Phi(U_{s-})\lambda_{j}^{1/2}f_{j}), h\rangle_{H}\mathrm{d}\langle W_{s},  \lambda_{j}^{1/2}f_{j}\rangle_{K_{Q}}\\
			&= \bigg\langle\sum_{j=1}^{\infty}\int_{0}^{E_{t}} (\Phi(U_{s-})\lambda_{j}^{1/2}f_{j}) \mathrm{d}\langle W_{s},  \lambda_{j}^{1/2}f_{j}\rangle_{K_{Q}}, h\bigg\rangle_{H}\\
			&= \bigg\langle \int_{0}^{E_{t}}\Phi(U_{s-})\mathrm{d}W_{s}, h\bigg\rangle_{H}.
		\end{aligned}
	\end{eqnarray*}
The first equality follows from Definition~\ref{timechangstochinte} and the third equality follows from the second change of variable formula for real-valued stochastic integrals from~\cite{Kei2011}. 
\end{proof}

\subsection{Time-changed It\^o formula in Hilbert space}
The technique used to develop the time-changed It\^o formula in this section is inspired by the proof of the standard It\^o formula of Theorem 2.9 in \cite{Mandrecar2010}.
\begin{thm}(Time-changed It\^o formula)
Let $Q$ be a symmetric, nonnegative definite trace-class operator on a separable Hilbert space $K$, and let $\{W_{E_t}\}_{0\leq t\leq T}$ be a time-changed Q-Wiener process on a filtered probability space $(\Omega, \mathcal{G}, \{\mathcal{G}_{t}\}_{0\leq t\leq T}, \mathbb{P})$.  Assume that a stochastic process $X(t)$ is given by
\begin{eqnarray*}
	X(t)=X(0)+\int_0^t \psi(s)\mathrm{d}s + \int_0^t \gamma(s)\mathrm{d}E_s + \int_0^t \phi(s) \mathrm{d}W_{E_s},
\end{eqnarray*}
where $X(0)$ is a $\mathcal{G}_0$-measurable $H$-valued random variable, $\psi(s)$ and $\gamma(s)$ are H-valued $\mathcal{G}_s$-measurable $\mathbb{P}$-a.s. integrable processes on $[0,T]$ such that
\begin{eqnarray*}
	\int_0^T\|\psi(s)\|_H\mathrm{d}s < \infty ~~~~\text{and}~~~~\int_0^T\|\gamma(s)\|_H
	\mathrm{d}E_s< \infty,
\end{eqnarray*}
and $\phi \in \tilde{\Lambda}_2(K_Q, H)$.
Also assume that $F:H \rightarrow \mathbb{R}$ is continuous and its Fr\'echet derivatives $F_x: H \rightarrow \mathcal{L}(H, \mathbb{R})$ and $F_{xx}: H \rightarrow \mathcal{L}(H, \mathcal{L}(H,\mathbb{R}))$ are continuous and bounded on bounded subsets of $H$. Then,
\begin{eqnarray*}
	\begin{aligned}
		F(X(t)) &- F(X(0))=\int_0^t \big\langle F_x(X(s)), \psi(s)\big\rangle_H\mathrm{d}s\\
		&+ \int_0^{E_t}\big\langle F_x(X(U(s-))), \gamma(U(s-))\big\rangle_H\mathrm{d}s \\
		&+ \int_0^{E_t}\big\langle F_x(X(U(s-))), \phi(U(s-))\mathrm{d}W_{s}\big\rangle_H\\
		&+\frac{1}{2} \int_0^{E_t} tr(F_{xx}(X(U(s-)))(\phi(U(s-))Q^{1/2})(\phi(U(s-))Q^{1/2})^*)\mathrm{d}s,
	\end{aligned}
\end{eqnarray*}
$\mathbb{P}$-a.s. for all $t \in [0,T]$.
\end{thm}
\begin{proof}
First, the desired It\^o formula is reduced to the case where $\psi(s)=\psi$, $\gamma(s)=\gamma$, and $\phi(s)=\phi$ are constant processes for $s \in [0,T]$.  Let $C>0$ be a constant, and define the stopping time
\begin{eqnarray*}
	\begin{aligned}
		\tau_C=\inf\bigg\{ t\in[0,T]: \max &\bigg(||X(t)||_H,\int_0^t ||\phi(s)||_H\mathrm{d}s, \int_0^t ||\gamma(s)||_H \mathrm{d}E_s, \\
		&\int_0^t ||\phi(s)||^2_{\mathcal{L}_2(K_Q,H)}\mathrm{d}E_s\bigg) \geq C \bigg\}.
	\end{aligned}
\end{eqnarray*}
Then, define $X_{C}(t)$ as
\begin{eqnarray*}
	X_C(t)=X_C(0)+\int_0^t\psi_C(s)\mathrm{d}s+\int_0^t\gamma_C(s)\mathrm{d}E_s+\int_0^t\phi_C(s) \mathrm{d}W_{E_s}, ~t\in[0,T],
\end{eqnarray*}
where $X_C(t)=X(t \wedge \tau_C), \psi_C(t)=\psi(t)1_{[0,\tau_C]}(t), \gamma_C(t)=\gamma(t)1_{[0,\tau_C]}(t)$, and $\phi_C(t)=\phi(t)1_{[0,\tau_C]}(t)$. It is enough to prove the It\^o formula for the processes stopped at $\tau_C$. Since
\begin{eqnarray*}
	\begin{aligned}
		\mathbb{P}\bigg(\int_0^T \|\psi_C(s)\|_H&\mathrm{d}s< \infty\bigg)=1,~~~\mathbb{P}\bigg(\int_0^T ||\gamma_C(s)||_H\mathrm{d}E_s< \infty\bigg)=1,~~\text{and}\\
		&\mathbb{E} \int_0^T ||\phi_C(s)||^2_{\mathcal{L}_{2}(K_Q,H)}\mathrm{d}E_s < \infty,
	\end{aligned}
\end{eqnarray*}
it follows that $\psi_C$, $\gamma_C$, and $\phi_C$ can be approximated respectively by sequences of bounded elementary processes $\psi_{C,n}$, $\gamma_{C,n}$, and $\phi_{C,n}$ such that as $n\rightarrow\infty$
\begin{eqnarray}\label{conditionsofprocesses}
	\begin{aligned}
		\int_0^t\|\psi_{C,n}(s) - &\psi_C(s)\|_H\mathrm{d}s \rightarrow 0,~~\int_0^t \|\gamma_{C,n}(s)-\gamma_C(s)\|_H\mathrm{d}E_s\rightarrow 0,~~\text{and}\\
		&\left\|\int_0^t \phi_{C,n}(s) dW_{E_s}- \int_0^t \phi_C(s)\mathrm{d}W_{E_s}\right\|_H \rightarrow 0,~~\mathbb{P}-\textrm{a.s.}
	\end{aligned}
\end{eqnarray}
uniformly in $t\in[0, T]$. Let
\begin{eqnarray*}
	X_{C,n}(t)=X(0)+\int_0^t \psi_{C,n}(s)\mathrm{d}s +\int_0^t \gamma_{C,n}(s)\mathrm{d}E_s +  \int_0^t \phi_{C,n}(s)\mathrm{d}W_{E_s}.
\end{eqnarray*}
Then, as $n\to\infty$
\begin{eqnarray}\label{almostsureestimation}
	\sup_{t\leq T}\|X_{C,n}(t)-X_C(t)\|_H \rightarrow 0,~\mathbb{P}-\textrm{a.s.}.
\end{eqnarray}
Assume the It\^o formula for $X_{C,n}(t)$ holds $\mathbb{P}$-a.s. for all $t\in[0, T]$, i.e.,
\begin{eqnarray}\label{itoformulaforelementaryprocess}
  \begin{aligned}
    F(&X_{C, n}(t)) - F(X(0)) = \int_0^t \langle F_x(X_{C, n}(s)),\phi_{C, n}(s)\mathrm{d}W_{E_s}\rangle_H\\
    &+\int_0^t \langle F_x(X_{C, n}(s)), \psi_{C, n}(s)\rangle_H\mathrm{d}s
    + \int_0^t \langle F_x(X_{C, n}(s)), \gamma_{C, n}(s)\rangle_H\mathrm{d}E_s \\
    &+ \int_{0}^{t}\frac{1}{2} tr[F_{xx}(X_{C, n}(s))(\phi_{C, n}(s)Q^{1/2})(\phi_{C, n}(s)Q^{1/2})^*]\mathrm{d}E_s\\
    &:= I_{C, n}^{1} + I_{C, n}^{2} + I_{C, n}^{3} + I_{C, n}^{4}.
  \end{aligned}
\end{eqnarray}
By using the continuity of $F$ and the continuity and local boundedness of $F_x$ and $F_{xx}$, it will suffice to show that the following holds $\mathbb{P}$-a.s. for all $t\in[0, T]$:
\begin{eqnarray}\label{itoformulaofstoppedprocess}
  \begin{aligned}
  F(&X_C(t)) - F(X(0)) = \int_0^t \langle F_x(X_C(s)),\phi_C(s)\mathrm{d}W_{E_s}\rangle_H \\
  &+\int_0^t \langle F_x(X_C(s)), \psi_C(s) \rangle_H\mathrm{d}s
  + \int_0^t \langle F_x(X_C(s)), \gamma_C(s) \rangle_H\mathrm{d}E_s\\
  &+ \frac{1}{2} \int_{0}^{t}tr[F_{xx}(X_C(s))(\phi_C(s)Q^{1/2})(\phi_C(s)Q^{1/2})^*]\mathrm{d}E_s\\
  &:= I_{C}^{1} + I_{C}^{2} + I_{C}^{3} + I_{C}^{4}.
  \end{aligned}
\end{eqnarray}
Consider, term by term, the difference between both sides of~\eqref{itoformulaforelementaryprocess} and~\eqref{itoformulaofstoppedprocess}.
Due to the continuity of $F$ and almost sure convergence in~\eqref{almostsureestimation}, the left hand side of~\eqref{itoformulaforelementaryprocess} converges to the left hand side of~\eqref{itoformulaofstoppedprocess} $\mathbb{P}$-a.s for all $t\leq T$, i.e.,
\begin{eqnarray}\label{almostconvergenceofLHS}
	F(X_{C,n}(t)) \rightarrow F(X_C(t)),~~\mathbb{P}-\textrm{a.s. as}~n\to\infty.
\end{eqnarray}
Turn to the first terms in both right hand sides of~\eqref{itoformulaforelementaryprocess} and~\eqref{itoformulaofstoppedprocess}, 
\begin{eqnarray*}
	\begin{aligned}
		\mathbb{E}|I_{C,n}^{1} - &I_{C}^{1}|^{2}
		= \mathbb{E}\bigg|\int_0^t\bigg( \phi_{C,n}^{*}(s)F_x(X_{C,n}(s)) -  \phi_{C}^{*}(s)F_x(X_{C}(s))\bigg)\mathrm{d}W_{E_s}\bigg|^2 \\
        &\leq 2 \mathbb{E} \int_0^t \|(\phi^*_{C,n}(s) - \phi_C^*(s))F_x(X_{C,n}(s))\|^2_{\mathcal{L}_2(K_Q, R)}\mathrm{d}E_s\\
		&\quad+ 2\mathbb{E} \int_0^t \|\phi_C^*(s)(F_x(X_{C,n}(s))-F_x(X_C(s)))\|^2_{\mathcal{L}_{2}(K_Q, R)}\mathrm{d}E_s\\
		&\leq 2 \mathbb{E} \int_0^t \bigg(\|\phi_C^*(s)-\phi^*_{C,n}\|^2_{\mathcal{L}_{2}(K_Q,H)} \|F_x(X_{C,n}(s)\|^2_H\bigg)\mathrm{d}E_s\\
		&\quad+ 2 \mathbb{E} \int_0^t \bigg(\|\phi_C^*(s)\|^2_{\mathcal{L}_{2}(K_Q,H)} \|F_x(X_{C,n}(s))-F_x(X_C(s))\|^2_{H}\bigg) \mathrm{d}E_s\\
		&:= J_{1} + J_{2},
	\end{aligned}
\end{eqnarray*}
where $\int_0^t \beta^*(s)\alpha(s)\mathrm{d}W_{E_s} := \int_0^t \langle \alpha(s),\beta(s)\mathrm{d}W_{E_s} \rangle_H $ and $\beta^{*}(s)$ is the adjoint operator of $\beta(s)$. Since $F_{x}$ is bounded on bounded subsets of $H$, there exists an $M > 0$ such that   
\begin{eqnarray*}
	J_{1} \leq M\mathbb{E}\int\limits_{0}^{t}\|\phi_C^*(s) - \phi^*_{C,n}\|^2_{\mathcal{L}^2(K_Q,H)}\mathrm{d}E_{s} \rightarrow 0, ~\text{as}~ n\to\infty.
\end{eqnarray*}  
Since $\phi_{C}(s)$ is square integrable in the space $\tilde{\Lambda}_{2}(K_{Q}, H)$ and $F_{x}$ is bounded in $H$, $J_{2}\rightarrow 0$ by applying the Lebesgue dominated convergence theorem. So, $I_{C, n}^{1}$ converges to $I_{C}^{1}$ in mean square, i.e.,
\begin{eqnarray}\label{meansquareconvergenceoffirstterm}
	\mathbb{E}|I_{C, n}^{1} - I_{C}^{1}|^{2} \rightarrow 0,
\end{eqnarray}
and thus converges in probability. For the second terms, $I_{C, n}^{2}$ and $I^{2}_{C}$, the RHSs of~\eqref{itoformulaforelementaryprocess} and~\eqref{itoformulaofstoppedprocess}, applying the conditions of~\eqref{conditionsofprocesses} and~\eqref{almostsureestimation} leads to
\begin{eqnarray}\label{almostconvergenceof2ndterm}
	\begin{aligned}
		I_{C, n}^{2} - I_{C}^{2} 
		&= \int_0^t\bigg(\bigg\langle F_x(X_{C,n}(s))- F_x(X_C(s)),\psi_{C,n}(s)\bigg\rangle_H\\
		&~~~ + \bigg\langle F_x(X_C(s)),\psi_{C,n}(s)-\psi_C(s)\bigg\rangle_H\bigg)\mathrm{d}s \rightarrow 0,~~\mathbb{P}-\textrm{a.s.}.	
	\end{aligned}
\end{eqnarray}
Similarly, for the third terms, $I_{C, n}^{3}$ and $I^{3}_{C}$, the RHSs of~\eqref{itoformulaforelementaryprocess} and~\eqref{itoformulaofstoppedprocess}, 
\begin{eqnarray}\label{almostconvergenceof3rdterm}
     \begin{aligned}
        I_{C, n}^{3} - I_{C}^{3} 
        &= \int_0^t\bigg(\bigg\langle F_x(X_{C,n}(s))- F_x(X_C(s)),\gamma_{C,n}(s)\bigg\rangle_H\\
        &~~~ + \bigg\langle F_x(X_C(s)), \gamma_{C,n}(s) - \gamma_C(s)\bigg\rangle_H\bigg)\mathrm{d}E_{s}\rightarrow 0,~~\mathbb{P}-\textrm{a.s.}.	
    \end{aligned}
\end{eqnarray}
Before proceeding to the fourth terms, $I_{C, n}^{4}$ and $I^{4}_{C}$, of RHSs of~\eqref{itoformulaforelementaryprocess} and~\eqref{itoformulaofstoppedprocess}, note that 
\begin{eqnarray*}
	\|\phi_{C,n}(s)-\phi_{C}(s)\|_{\tilde{\Lambda}_2(K_Q,H)} \rightarrow 0,
\end{eqnarray*}
which means there exists a subsequence $n_k$ such that for all $s\leq T$
\begin{eqnarray*}
	\|\phi_{C,n_k}(s)-\phi_{C}(s)||_{\mathcal{L}_{2}(K_Q,H)} \rightarrow 0,~~\mathbb{P}-\textrm{a.s.}.
\end{eqnarray*}
Thus, for the eigenvectors $\{f_{j}\}_{j=1}^{\infty}$ of $Q$ and all $t\leq T$,
\begin{eqnarray}\label{phicon}
	||\phi_{C,n_k}(s)f_j-\phi_{C}(s)f_j||_H \rightarrow 0,~\mathbb{P}-\text{a.s.}
\end{eqnarray}
On the other hand, for the ONB, $\{f_{j}\}_{j=1}^{\infty}$, in the  Hilbert space $K$, 
\begin{eqnarray*}
	\begin{aligned}
		tr(F_{xx}&(X_{C,n_k}(s))\phi_{C, n_k}(s)Q\phi^*_{C,n_k}(s)) = tr(\phi^*_{C,n_k}(s)F_{xx}(X_{C,n_k}(s))\phi_{C,n_k}(s)Q)\\
		&= \sum_{j=1}^{\infty} \lambda_j \langle F_{xx}(X_{C,n_k}(s))\phi_{C,n_k}(s)f_j,\phi_{C,n_k}(s)f_j \rangle_H,
	\end{aligned}
\end{eqnarray*}
where $\lambda_{j}$ is the eigenvalue associated with eigenvector $f_{j}$ of $Q$.
Since $X_{C,n_k}(s)$ is bounded and $F_{xx}$ is continuous,~\eqref{phicon} implies that for $s\leq T$
\begin{eqnarray*}
	\begin{aligned}
	  \langle F_{xx}(X_{C,n_k}(s))&\phi_{C,n_k}(s)f_j,\phi_{C,n_k}(s)f_j\rangle_H\\
	  &\rightarrow \langle F_{xx}(X_C(s))\phi_C(s)f_j, \phi_C(s)f_j \rangle_H,~\mathbb{P}-a.s..
	\end{aligned}
\end{eqnarray*}
By the Lebesgue Dominated Convergence Theorem (DCT) (with respect to the counting measure), it holds a.e. on $[0, T]\times\Omega$ that
\begin{eqnarray}\label{convergenceoftraceterm}
    \begin{aligned}
    	tr( F_{xx}(X_{C,n_k}(s))&\phi_{C,n_k}(s)Q\phi^*_{C,n_k}(s))\\
    	&\rightarrow tr(F_{xx}(X_C(s))\phi_C(s)Q\phi^{*}_{C}(s)).
    \end{aligned}
\end{eqnarray}
Moreover, the left hand side of~\eqref{convergenceoftraceterm} is bounded above by 
\begin{eqnarray*}
 \eta_n(s) := \|F_{xx}(X_{C,n_k}(s))\|_{\mathcal{L}(H)}\|\phi_{C,n_k}\|^2_{\tilde{\Lambda}_2(K_Q,H)},
\end{eqnarray*}
and a.e. on $[0, T]\times\Omega$
\begin{eqnarray*}
	\eta_n(s)\rightarrow\eta(s)=\|F_{xx}(X_{C}(s))\|_{\mathcal{L}(H)}\|\phi_{C}\|^2_{\tilde{\Lambda}_2(K_Q,H)}.
\end{eqnarray*}
So, by the boundedness of $F_{xx}$, $\int_0^t \eta_n(s) dE_s \rightarrow \int_0^t \eta(s) dE_s$ and by the general Lebesgue DCT, it holds $\mathbb{P}$-a.s. that for $t\leq T$
\begin{eqnarray}\label{almostconvergenceoftraceterm}
	\begin{aligned}
	I_{C, n_{k}}^{4} - I_{C}^{4} &= \int_{0}^{t}\frac{1}{2} tr[F_{xx}(X_{C,n_k}(s))(\phi_{C,n_k}(s)Q^{1/2})(\phi_{C,n_k}(s)Q^{1/2})^*]\mathrm{d}E_s \\
		&- \int_{0}^{t}\frac{1}{2} tr[F_{xx}(X_C(s))(\phi_C(s)Q^{1/2})(\phi_C(s)Q^{1/2})^*]\mathrm{d}E_s\rightarrow 0.	
	\end{aligned}
\end{eqnarray}
Therefore, from~\eqref{almostconvergenceofLHS},~\eqref{meansquareconvergenceoffirstterm},~\eqref{almostconvergenceof2ndterm},~\eqref{almostconvergenceof3rdterm} and~\eqref{almostconvergenceoftraceterm}, the It\^o formula for the process $X_{C,n}(t)$,~\eqref{itoformulaforelementaryprocess}, converges in probability to the It\^o formula for the process $X_{C}(t)$,~\eqref{itoformulaofstoppedprocess}, and possibly for a subsequence, $n_{k}$, converges $\mathbb{P}$-a.s.

Second, the proof can be reduced to the case where 
\begin{eqnarray*}
	X(t)=X(0)+\psi t +\gamma E_t + \phi W_{E_t}
\end{eqnarray*}
where $\psi$,$\gamma$, and $\phi$ are $\mathcal{G}_0$-measurable random variables independent of $t$. Define the function $u(t_{1}, t_{2}, x): \mathbb{R}_{+}\times\mathbb{R}_{+}\times H\to \mathbb{R}$ as
\begin{eqnarray*}
	u(t, E_{t}, W_{E_t}) = F(X(0)+\psi t +\gamma E_t + \phi W_{E_t})=F(X(t)).
\end{eqnarray*} 
Now, we prove that the It\^o formula holds for the function $u(t_{1}, t_{2}, x)$.
First, let $0=t_1 < t_2 < ... < t_n =t \leq T$ be a partition of an interval $[0,t]$, then
\begin{eqnarray*}
	\begin{aligned}
		u(t, E_{t}, W_{E_t}) - u(0, 0, 0)
		&= \sum_{j=1}^{n-1} [u(t_{j+1}, E_{t_{j+1}}, W_{E_{t_{j+1}}})-u(t_{j}, E_{t_{j+1}}, W_{E_{t_j+1}})]\\ 
		&\quad+\sum_{j=1}^{n-1} [u(t_{j}, E_{t_{j+1}}, W_{E_{t_{j+1}}})-u(t_{j}, E_{t_{j}}, W_{E_{t_j+1}})]\\
		&\quad +\sum_{j=1}^{n-1} [u(t_{j}, E_{t_{j}}, W_{E_{t_{j+1}}}) - u(t_{j}, E_{t_{j}}, W_{E_{t_j}})].\\
	\end{aligned}
\end{eqnarray*}
Also, let $\Delta t_{j} = t_{j+1} - t_{j}$, $\Delta E_j= E_{t_{j+1}}-E_{t_j}$ and $\Delta W_j= W_{E_{t_{j+1}}}-W_{E_{t_j}}$. Let $\theta_j \in [0,1]$ be a random variable, and $\bar{t}_j = t_{j} + \theta_j (t_{j+1} - t_{j})$, $\bar{E}_j= E_{t_j} + \theta_j (E_{t_{j+1}} - E_{t_j})$ and $\bar{W}_j=W_{E_{t_j}}+\theta_j (W_{E_{t_{j+1}}}-W_{E_{t_j}})$.  Using Taylor's formula,
\begin{eqnarray}\label{taylorexpansion}
   \begin{aligned}
     &u(t, E_{t}, W_{E_t}) - u(0, 0, 0)\\ 
     &= \sum_{j=1}^{n-1}u_{t_{1}}(\bar{t}_{j}, E_{t_{j+1}}, W_{E_{t_{j+1}}})\Delta t_{j} + \sum_{j=1}^{n-1}u_{t_{2}}(t_{j}, \bar{E}_{j}, W_{E_{t_{j+1}}})\Delta E_{j}\\
     &+\sum_{j=1}^{n-1}[\langle u_x(t_{j}, E_{t_{j}}, W_{E_{t_j}}), \Delta W_j\rangle_K + \frac{1}{2}\langle u_{xx}(t_{j}, E_{t_{j}}, \bar{W}_j)(\Delta W_j), \Delta W_j\rangle_K]\\
     &=\sum_{j=1}^{n-1}u_{t_{1}}(t_{j}, E_{t_{j+1}}, W_{E_{t_{j+1}}})\Delta t_{j}\\ 
     &+ \sum_{j=1}^{n-1}u_{t_{2}}(t_{j}, E_{t_{j}}, W_{E_{t_{j+1}}})\Delta E_{j}
     + \sum_{j=1}^{n-1}\langle u_x(t_{j}, E_{t_{j}}, W_{E_{t_j}}), \Delta W_j\rangle_K\\ &+\frac{1}{2} \sum_{j=1}^{n-1}\langle u_{xx}(t_{j}, E_{t_{j}}, W_{E_{t_j}})(\Delta W_j), \Delta W_j \rangle_K\\
     &+\sum_{j=1}^{n-1}[u_{t_{1}}(\bar{t}_{j}, E_{t_{j+1}}, W_{E_{t_{j+1}}}) - u_{t_{1}}(t_{j}, E_{t_{j+1}}, W_{E_{t_{j+1}}})]\Delta t_{j}\\
     &+\sum_{j=1}^{n-1}[u_{t_{2}}(t_{j}, \bar{E}_{j+1}, W_{E_{t_{j+1}}}) - u_{t_{2}}(t_{j}, E_{t_{j}}, W_{E_{t_{j+1}}})]\Delta E_{j}\\
     &+ \frac{1}{2} \sum_{j=1}^{n-1} \langle[u_{xx}(t_{j}, E_{t_{j}}, \bar{W}_j)(\Delta W_j)- u_{xx}(t_{j}, E_{t_{j}}, W_{E_{t_j}})(\Delta W_j)], \Delta W_j\rangle_K\\
     &:= I_{1} + I_{2} + I_{3} + I_{4} + I_{5} + I_{6} + I_{7}.
   \end{aligned}
\end{eqnarray}
By the uniform continuity of the mappings:
\begin{eqnarray*}
	\begin{aligned}
		&[0, T]\times [0, T]\times [0, T]\ni (t, s, r) \to u_{t_{1}}(t, E_{s}, W_{E_{r}})\in\mathbb{R}\\
		&[0, T]\times [0, T]\times [0, T]\ni (t, s, r) \to u_{t_{2}}(t, E_{s}, W_{E_{r}})\in\mathbb{R}
	\end{aligned}
\end{eqnarray*} 
and the continuity of the map $[0,T] \ni t \rightarrow u_x(t, E_{t}, W_{E_t}) \in K^*$, the following holds~$\mathbb{P}$-a.s.
\begin{eqnarray}\label{almostsureconvergenceofirstterminItoformula}
\begin{aligned}
&I_{1} = \sum_{j=1}^{n-1}u_{t_{1}}(t_{j}, E_{j+1}, W_{E_{t_{j+1}}})\Delta t_{j} \to \int_{0}^{t}u_{t_{1}}(s, E_{s}, W_{E_{s}})\mathrm{d}s,\\
&I_{2} = \sum_{j=1}^{n-1}u_{t_{2}}(t_{j}, E_{j}, W_{E_{t_{j+1}}})\Delta E_{t_{j}} \to \int_{0}^{t}u_{t_{2}}(s, E_{s}, W_{E_{s}})\mathrm{d}E_{s},\\
&I_{3} = \sum_{j=1}^{n-1}\langle u_x(s, E_{s}, W_{E_{t_j}}), \Delta W_j\rangle_K \rightarrow \int_0^t\langle u_x(s, E_{s}, W_{E_{s}}), \mathrm{d}W_{E_s}\rangle_K.
\end{aligned}
\end{eqnarray}
Also since the time change $E_{t}$ has bounded variation, 
\begin{eqnarray}\label{convergenceofintegralswithrespecttotimetandtimechangeE_{t}}
\begin{aligned}
&|I_{5}| = \bigg|\sum_{j=1}^{n-1}[u_{t_{1}}(\bar{t}_{j}, E_{t_{j+1}}, W_{E_{t_{j+1}}}) - u_{t_{1}}(t_{j}, E_{t_{j+1}}, W_{E_{t_{j+1}}})]\Delta t_{j}\bigg|\\
&\leq T\sup_{j\leq n-1}\bigg|u_{t_{1}}(\bar{t}_{j}, E_{t_{j+1}}, W_{E_{t_{j+1}}}) - u_{t_{1}}(t_{j}, E_{t_{j+1}}, W_{E_{t_{j+1}}})\bigg|\\
&\to 0,\\
&|I_{6}| = \bigg|\sum_{j=1}^{n-1}[u_{t_{2}}(t_{j}, \bar{E}_{j+1}, W_{E_{t_{j+1}}}) - u_{t_{2}}(t_{j}, E_{t_{j}}, W_{E_{t_{j+1}}})]\Delta E_{j}\bigg|\\
&\leq\sum_{j=1}^{n-1}|\Delta E_{t_{j}}|\sup_{j\leq n-1}\bigg|u_{t_{2}}(t_{j}, \bar{E}_{j+1}, W_{E_{t_{j+1}}}) - u_{t_{2}}(t_{j}, E_{t_{j}}, W_{E_{t_{j+1}}})\bigg|\\
&\to 0.
\end{aligned}
\end{eqnarray}
Similarly, by the continuity of the map $K\ni x\rightarrow u_{xx}(t_{1}, t_{2}, x)\in\mathcal{L}(K, K)$, 
\begin{eqnarray}\label{almostsureconvergenceofthirdterminItoformula}
	\begin{aligned}
		|I_{7}| &= \frac{1}{2}\bigg|\sum_{j=1}^{n-1}\langle [u_{xx}(t_{j}, E_{t_{j}}, \bar{W}_j)(\Delta W_j)\\ 
		&\qquad\qquad\qquad\qquad- u_{xx}(t_{j}, E_{t_{j}}, W_{E_{t_j}})(\Delta W_j)], \Delta W_j \rangle_K\bigg|\\
		&\leq \sup_{j\leq n-1} ||u_{xx}(t_{j}, E_{t_{j}}, \bar{W}_j)(\Delta W_j)- u_{xx}(t_{j}, E_{t_{j}}, W_{E_{t_j}})(\Delta W_j)||_{\mathcal{L}(K,K)}\\
		&\quad\times \sum_{j=1}^{n-1}||\Delta W_j||_K^2 \rightarrow 0
	\end{aligned}
\end{eqnarray}
with probability one as $n \rightarrow \infty$ since the function $u$ has the same smoothness as $F$ and $\bar{W}_j \rightarrow W_{E_{t_j}}$ as the increments $t_{j+1}-t_j$ get smaller. It remains to deal with the fourth term, $I_{4}$. Let $1_j^N=1_{\{\max \{\|W_{E_{t_i}}\|_{K} \leq N,~i \leq j\} \} }$ which is $\mathcal{G}_{t_j}$-measurable. To handle $I_{4}$, the following computations are helpful. First, 
\begin{eqnarray}\label{secondderivativeconditionalexpectation}
\begin{aligned}
&\mathbb{E}\bigg(\langle 1_j^N u_{xx}(t_{j}, E_{t_{j}}, W_{E_{t_j}})(\Delta W_j), \Delta W_j \rangle_K \bigg| \mathcal{G}_{t_j}\bigg)\\
&=\mathbb{E}\bigg(\bigg\langle 1_j^N u_{xx}(t_{j}, E_{t_{j}}, W_{E_{t_j}}) \sum_{k=1}^{\infty} \lambda_k^{1/2} (w_k(E_{t_{j+1}})-w_k(E_{t_j}))f_k,\\
&\qquad\qquad\qquad \sum_{\l=1}^{\infty} \lambda_{\l}^{1/2} (w_{\l}(E_{t_{j+1}})-w_{\l}(E_{t_j}))f_{\l} \bigg\rangle_K \bigg| \mathcal{G}_{t_j}\bigg)\\
&=\sum_{k=1}^{\infty} \mathbb{E}\bigg( \lambda_k\langle 1_j^N u_{xx}(t_{j}, E_{t_{j}}, W_{E_{t_j}}) f_k, f_k \rangle_K  (w_k(E_{t_{j+1}})-w_k(E_{t_j}))^2 \bigg| \mathcal{G}_{t_j}\bigg)\\
&= tr(1_j^N u_{xx}(t_{j}, E_{t_{j}}, W_{E_{t_j}})Q) \Delta E_j.
\end{aligned}
\end{eqnarray}
Second, for the cross term arising in the computation below, without loss of generality, assume $i < j$. Then,
\begin{eqnarray*}
	\begin{aligned}
		I^{N}&:=\mathbb{E}\bigg\{\bigg(\big\langle 1_i^N u_{xx}(t_{j}, E_{t_{j}}, W_{E_{t_i}})(\Delta W_i), \Delta W_i\big\rangle_K\\ &\qquad\qquad\qquad\qquad- tr\big(1_i^N u_{xx}(t_{j}, E_{t_{j}}, W_{E_{t_i}})Q\big)\Delta E_i\bigg)\\
		&\qquad\times\bigg(\big\langle 1_j^N u_{xx}(t_{j}, E_{t_{j}}, W_{E_{t_j}})(\Delta W_j), \Delta W_j\big\rangle_K\\ &\qquad\qquad\qquad\qquad- tr\big(1_j^N u_{xx}(t_{j}, E_{t_j}, W_{E_{t_j}})Q\big)\Delta E_j\bigg) \bigg\}\\
	    &= \mathbb{E}\bigg\{\bigg(\big\langle 1_i^N u_{xx}(t_{j}, E_{t_{j}}, W_{E_{t_i}})(\Delta W_i), \Delta W_i\big\rangle_K\\ &\qquad\qquad\qquad\qquad- tr\big(1_i^N u_{xx}(t_{j}, E_{t_{j}}, W_{E_{t_i}})Q\big)\Delta E_i\bigg)\\
	    &\qquad\times\mathbb{E}\bigg(\big\langle 1_j^N u_{xx}(t_{j}, E_{t_{j}}, W_{E_{t_j}})(\Delta W_j), \Delta W_j\big\rangle_K\\ 
	    &\qquad\qquad\qquad\qquad - tr\big(1_j^N u_{xx}(t_{j}, E_{t_{j}}, W_{E_{t_j}})Q\big)\Delta E_j\bigg|\mathcal{G}_{t_{(i+1)}}\bigg) \bigg\}.
	\end{aligned}
\end{eqnarray*}
Applying the tower property of conditional expectation and~\eqref{secondderivativeconditionalexpectation} yields
\begin{eqnarray}\label{expectationofcorsstermwithdifferentindexes}
\begin{aligned}
	    I^{N}&= \mathbb{E}\bigg\{\bigg(\big\langle 1_i^N u_{xx}(t_{j}, E_{t_{j}}, W_{E_{t_i}})(\Delta W_i), \Delta W_i\big\rangle_K\\ &\qquad\qquad\qquad\qquad- tr\big(1_i^N u_{xx}(t_{j}, E_{t_{j}}, W_{E_{t_i}})Q\big)\Delta E_i\bigg)\\
	    &\qquad\times\mathbb{E}\bigg(\mathbb{E}\big(\big\langle 1_j^N u_{xx}(t_{j}, E_{t_{j}}, W_{E_{t_j}})(\Delta W_j), \Delta W_j\big\rangle_K|\mathcal{G}_{t_{j}}\big)\\ &\qquad\qquad\qquad\qquad- tr\big(1_j^N u_{xx}(t_{j}, E_{t_{j}}, W_{E_{t_j}})Q\big)\Delta E_j\bigg|\mathcal{G}_{t_{(i+1)}}\bigg) \bigg\}\\
	    &= \mathbb{E}\bigg\{\bigg(\big\langle 1_i^N u_{xx}(t_{j}, E_{t_{j}}, W_{E_{t_i}})(\Delta W_i), \Delta W_i\big\rangle_K\\ &\qquad\qquad\qquad\qquad- tr\big(1_i^N u_{xx}(t_{j}, E_{t_{j}}, W_{E_{t_i}})Q\big)\Delta E_i\bigg)\\
	    &\qquad\times\mathbb{E}\bigg(\mathbb{E}\big(tr\big(1_j^N u_{xx}(t_{j}, E_{t_{j}}, W_{E_{t_j}})Q\big)\Delta E_j\\ &\qquad\qquad\qquad\qquad- tr\big(1_j^N u_{xx}(t_{j}, E_{t_{j}}, W_{E_{t_j}})Q\big)\Delta E_j\bigg|\mathcal{G}_{t_{(i+1)}}\bigg) \bigg\}\\	   
		&= 0.
	\end{aligned}
\end{eqnarray}
Third, let $f_{E_{t_{j+1}}, E_{t_{j}}}(\tau_{1}, \tau_{2})$ be the joint density function of random variables, $E_{t_{j+1}}$ and $E_{t_{j}}$. Then, for $t_{j+1} > t_{j}$, letting $\mathbb{D} =\{(\tau_{1}, \tau_{2})\in \mathbb{R}_{+}\times\mathbb{R}_{+}: 0\leq\tau_{2}\leq\tau_{1}\}$,
\begin{eqnarray}\label{expectationofcrossterminproduct}
\begin{aligned}
&\quad\mathbb{E}\bigg(\big\langle 1_j^N u_{xx}(t_{j}, E_{t_{j}}, W_{E_{t_j}})(\Delta W_j), \Delta W_j\big\rangle_K\\ 
&\qquad\qquad\qquad\qquad\qquad\qquad\times tr\big(1_j^N u_{xx}(t_{j}, E_{t_{j}}, W_{E_{t_j}})Q\big)\Delta E_j
\bigg)\\
&= \iint_{\mathbb{D}}\mathbb{E}\bigg(\big\langle 1_j^N u_{xx}(t_{j}, E_{t_{j}}, W_{E_{t_j}})(\Delta W_j), \Delta W_j\big\rangle_K\\ 
&\qquad\qquad\qquad\qquad\qquad\qquad\times tr\big(1_j^N u_{xx}(t_{j}, E_{t_{j}}, W_{E_{t_j}})Q\big)\Delta E_j\\
&\qquad\qquad\qquad\qquad\qquad\bigg|E_{t_{j+1}}=\tau_{1}, E_{t_{j}}=\tau_{2}\bigg)f_{E_{t_{j+1}}, E_{t_{j}}}(\tau_{1}, \tau_{2})\mathrm{d}(\tau_{1}, \tau_{2})\\
&= \iint_{\mathbb{D}}\mathbb{E}\bigg(tr\big(1_j^N u_{xx}(t_{j}, \tau_{2}, W_{\tau_{2}})Q\big)^{2}(\tau_{1} - \tau_{2})^{2}\bigg)f_{E_{t_{j+1}}, E_{t_{j}}}(\tau_{1}, \tau_{2})\mathrm{d}(\tau_{1}, \tau_{2})\\
&= \mathbb{E}\bigg(tr\big(1_j^N u_{xx}(t_{j}, E_{t_{j}}, W_{E_{t_{j}}})Q\big)^{2}(\Delta E_{t_{j}})^{2}\bigg),
\end{aligned}
\end{eqnarray}
Finally, still for $t_{j+1} > t_{j}:$
\begin{eqnarray}\label{expectationof4thmoment}
\begin{aligned}
\mathbb{E}&\|\Delta W_{j}\|_{K}^{4} = \mathbb{E}\|W_{E_{t_{j+1}}} - W_{E_{t_{j}}}\|_{K}^{4}\\ 
&= \iint_{\mathbb{D}}\mathbb{E}\bigg(\|W_{\tau_{1}} - W_{\tau_{2}}\|_{K}^{4}\bigg|E_{t_{j+1}}=\tau_{1}, E_{t_{j}}=\tau_{2}\bigg)\\ 
&\qquad\qquad\qquad\qquad\qquad\qquad\qquad\times f_{E_{t_{j+1}}, E_{t_{j}}}(\tau_{1}, \tau_{2})\mathrm{d}(\tau_{1}, \tau_{2})\\
&= \iint_{\mathbb{D}}3(trQ)^{2}(\tau_{1} - \tau_{2})^{2}f_{E_{t_{j+1}}, E_{t_{j}}}(\tau_{1}, \tau_{2})\mathrm{d}(\tau_{1}, \tau_{2})\\
&= 3(trQ)^{2}\mathbb{E}(E_{t_{j+1}} - E_{t_{j}})^{2} = 3(trQ)^{2}\mathbb{E}(\Delta E_{j})^{2}.
\end{aligned}
\end{eqnarray}
Noting that $u_{xx}$ is bounded on bounded subsets of $H$, apply~\eqref{secondderivativeconditionalexpectation},~\eqref{expectationofcorsstermwithdifferentindexes}, \eqref{expectationofcrossterminproduct} and~\eqref{expectationof4thmoment} to yield
\begin{eqnarray*}
	\begin{aligned}
		&\mathbb{E}\bigg(\sum_{j=1}^{n-1}\big( \big\langle 1_j^N u_{xx}(t_{j}, E_{t_{j}}, W_{E_{t_j}})(\Delta W_j), \Delta W_j\big\rangle_K\qquad\qquad\qquad\qquad\qquad\qquad\qquad\qquad\\ &\qquad\qquad\qquad\qquad\qquad\qquad- tr\big(1_j^N u_{xx}(t_{j}, E_{t_{j}}, W_{E_{t_j}})Q\big) \Delta E_j\big)\bigg)^2\\
		&= \sum_{j=1}^{n-1}\mathbb{E}\bigg( \big\langle 1_j^N u_{xx}(t_{j}, E_{t_{j}}, W_{E_{t_j}})(\Delta W_j), \Delta W_j\big\rangle_K\\ &\qquad\qquad\qquad\qquad\qquad\qquad- tr\big(1_j^N u_{xx}(t_{j}, E_{t_{j}}, W_{E_{t_j}})Q\big) \Delta E_j\bigg)^2\\
	\end{aligned}
\end{eqnarray*}
\begin{eqnarray*}
	\begin{aligned}
		&\qquad\qquad + \sum_{i\neq j=1}^{n-1}\mathbb{E}\bigg\{\bigg( \big\langle 1_i^N u_{xx}(t_{i}, E_{t_{i}}, W_{E_{t_i}})(\Delta W_i), \Delta W_i\big\rangle_K\\ &\qquad\qquad\qquad\qquad\qquad\qquad\quad- tr\big(1_i^N u_{xx}(t_{i}, E_{t_{i}}, W_{E_{t_i}})Q\big) \Delta E_i\bigg)\\
		&\qquad\qquad\times \bigg( \big\langle 1_j^N u_{xx}(t_{j}, E_{t_{j}}, W_{E_{t_j}})(\Delta W_j), \Delta W_j\big\rangle_K\\ &\qquad\qquad\qquad\qquad\qquad\qquad\quad- tr\big(1_j^N u_{xx}(t_{j}, E_{t_{j}}, W_{E_{t_j}})Q\big) \Delta E_j\bigg)\bigg\}\\
		&\quad=\sum_{j=1}^{n-1}\bigg\{\mathbb{E}\big\langle 1_j^N u_{xx}(t_{j}, E_{t_{j}}, W_{E_{t_j}})(\Delta W_j), \Delta W_j\big\rangle_K^2\\ &\qquad\qquad\qquad\qquad\qquad\qquad\quad  - \mathbb{E}\bigg(tr(1_j^N u_{xx}(t_{j}, E_{t_{j}}, W_{E_{t_j}})Q)^2 (\Delta E_j)^2\bigg)\bigg\}\\
		&\quad\leq \sup_{s\leq t, \|h\|_H \leq N} |u_{xx}(s, E_{s}, h)|^2_{\mathcal{L}(H)} \sum_{j=1}^{n-1}\big(\mathbb{E}||\Delta W_j ||_K^4 - (tr Q)^2 \mathbb{E}(\Delta E_j)^2\big)\\
		&\quad= 2\sup_{s\leq t, \|h\|_H \leq N} |u_{xx}(s, E_{s}, h)|^2_{\mathcal{L}(H)}(trQ)^{2} \mathbb{E}\sum_{j=1}^{n-1}(\Delta E_j)^2 \rightarrow 0,
	\end{aligned}
\end{eqnarray*}
where the convergence follows from the Fubini-Tonelli Theorem and the fact that the time-change $E_{t}$ has finite bounded variation. Additionally,
\begin{eqnarray*}
	\begin{aligned}
		\mathbb{P}\bigg(&\sum_{j=1}^{n-1}(1-1_j^N)\big\{\big\langle u_{xx}(t_{j}, E_{t_{j}}, W_{E_{t_j}})(\Delta W_j), \Delta W_j\big\rangle_K\\ &\qquad\qquad\qquad\qquad\qquad- tr\big(1_j^N u_{xx}(t_{j}, E_{t_{j}}, W_{E_{t_j}})Q\big)\Delta E_j\big\} \neq 0\bigg)\\
		&\leq\mathbb{P}\bigg(\sup_{s \leq t} \{||W_{E_s}|| > N\}\bigg) \rightarrow 0~\textrm{as}~N \rightarrow \infty.
	\end{aligned}
\end{eqnarray*}
Thus, it follows that
\begin{eqnarray}\label{almostsureconvergenceofsecondterminItoformula}
\begin{aligned}
	I_{4} &= \frac{1}{2}\sum_{j=1}^{n-1}\langle u_{xx}(t_{j}, E_{t_{j}}, W_{E_{t_j}})(\Delta W_j), \Delta W_j \rangle_K\\ &\rightarrow \frac{1}{2}\int_0^t tr[u_{xx}(s, E_{s}, W_{E_s})Q]\mathrm{d}E_s
\end{aligned}
\end{eqnarray}
in probability. Combining~\eqref{almostsureconvergenceofirstterminItoformula},~\eqref{almostsureconvergenceofthirdterminItoformula} and~\eqref{almostsureconvergenceofsecondterminItoformula}, and taking the limit on the right hand side of~\eqref{taylorexpansion} yields the It\^o formula for the function $u(t,E_t,W_{E_t})$:
\begin{equation} \label{itoformulaoffunctionu(x)}
\begin{aligned}	   
u(t, E_{t}, W_{E_t})  =  u(0, 0, 0) + \int_{0}^{t}u_{t_{1}}(s, E_{s}, W_{E_{s}})\mathrm{d}s\\ 
 +  \int_{0}^{t}u_{t_{2}}(s, E_{s}, W_{E_{s}})\mathrm{d}E_{s}\\
 +  \frac{1}{2}\int_0^t
	   tr[u_{xx}(s, E_{s}, W_{E_s})Q]\mathrm{d}E_s\\ 
 +  \int_0^t \langle u_x(s, E_{s}, W_{E_s}), dW_{E_s} \rangle_K,\\
\end{aligned}
\end{equation}
in probability. Also note that
\begin{equation}\label{derivativesoffunctionu(x)}
	\begin{aligned}   
	u_{t_{1}}(t, E_{t}, k) = \langle F_{x}(X(0)+\psi t+ \gamma E_t + \phi k), \psi\rangle_{H}\\
	   u_{t_{2}}(t, E_{t}, k) = \langle F_{x}(X(0)+\psi t+ \gamma E_t + \phi k), \gamma\rangle_{H}\\
	   u_x(t, E_{t}, k)=\phi^*F_x(X(0)+\psi t+\gamma E_t + \phi k), \\
	   u_{xx}(t, E_{t}, k)=\phi^*F_{xx}(X(0)+\psi t+ \gamma E_t + \phi k)\phi,\\
\end{aligned}
\end{equation}
and
\begin{eqnarray}\label{reorderoftrace}
	tr[F_{xx}(X(s))(\phi Q^{1/2})(\phi Q^{1/2})^*]=tr[\phi^*F_{xx}(X(s))\phi Q].
\end{eqnarray}
Substituting ~\eqref{derivativesoffunctionu(x)} and~\eqref{reorderoftrace} into~\eqref{itoformulaoffunctionu(x)} yields 
\begin{equation}\label{itoformulainprobability}
	\begin{aligned}	
F(X(t))  =  F(X(0)) + \int_0^t \langle F_x(X(s)), \psi(s)\rangle_H\mathrm{d}s\\
		 +  \int_0^t\langle F_x(X(s)), \gamma(s)\rangle_H\mathrm{d}E_s\\ 
 +  \int_0^t\langle F_x(X(s)), \phi(s)\mathrm{d}W_{E_s}\rangle_H\\
	 +  \frac{1}{2} \int_0^t tr(F_{xx}(X(s))(\phi(s)Q^{1/2})(\phi(s)Q^{1/2})^*) \mathrm{d}E_s\\	
\end{aligned}
\end{equation}
in probability. Consequently, there is a subsequence such that the equality~\eqref{itoformulainprobability} holds almost surely. Therefore applying the second change of variable formula yields the desired result.
\end{proof}

\subsection{Stochastic differential equations (SDEs) driven by the time-changed $Q$-Wiener process}
Let $K$ and $H$ be real separable Hilbert spaces, and let $M(t) := W_{E_t}$ be a time-changed $K$-valued $Q$-Wiener process on a complete filtered probability space $(\Omega, \mathcal{G}, \{\mathcal{G}\}_{t\leq T}, \mathbb{P})$ with the filtration $\mathcal{G}_{t} = \mathcal{\tilde{F}}_{E_{t}}$ satisfying the usual conditions. Consider the following type of semilinear SDE driven by the time-changed $Q$-Wiener process on $[0, T]$ in $H$: 
\begin{eqnarray}\label{SDEdrivenbytimechangedQwiener}
 	du(t)=Au(t)\mathrm{d}t+B(u,t)\mathrm{d}M(t)
\end{eqnarray}
with initial condition $u(0)=u_0$, where $A$ is the generator of a $C_{0}$-semigroup of operators $\{S(t), t\geq 0\}$. 
\begin{thm}\label{existenceofmildsolution}
Assume that the following hypotheses are satisfied:
\begin{enumerate}
	\item[1)] $S(t)$ is a contraction $C_0$-semigroup generated by $A$.
	\item[2)] $B(t,u): \mathcal{D}(\mathbb{R}^+, H) \rightarrow \tilde{\Lambda}_2(K_Q, H)$ is non-anticipating where $\mathcal{D}(\mathbb{R}^+, H)$ denotes the $H$-valued cadlag adapted processes with $\mathbb{R}^+$ as the time interval.
	\item[3)] (Local Lipschitz property) For every $r>0$, there exists a constant $K_r>0$ such that for every $x,y \in\mathcal{D}(\mathbb{R}^+, H)$ and $t \geq 0$ 
	\begin{eqnarray*}
		\|B(t,x)-B(t,y)||_{\mathcal{L}_2(K_Q, H)}^2 \leq K_r \sup_{s<t} \|x(s)-y(s)||_H^2
	\end{eqnarray*}
 on the set $\{\omega: \sup_{s<t} \max (\|x(s,\omega)||_H,||y(s,\omega)\|_H)\leq r\}.$
	\item[4)] (Growth condition) There exists a constant $K_{0}>0$ such that for every $x \in \mathcal{D}(\mathbb{R}^+, H)$ and $t \geq 0$,
	\begin{eqnarray*}
		 \|B(t,x)\|_{\mathcal{L}_2(K_Q, H)}^2 \leq K_{0}(1+\sup_{s<t} \|x(s)\|_H^2).
	\end{eqnarray*}
	\item[5)] $\mathbb{E}(\|u_0\|_H^2) < \infty$.
\end{enumerate}
	Then, the SDE~\eqref{SDEdrivenbytimechangedQwiener} has a pathwise continuous solution of the form
	\begin{eqnarray}
		u(t)=S(t)u_0+\int_{0}^{t}S(t-s)B(s,u)dW_{E_{s}},
	\end{eqnarray}
which is called the mild solution to the SDE.
\end{thm}
\begin{rem}
	Since $W_{E_{t}}$ is a Hilbert space-valued square integrable martingale, Theorem~\ref{existenceofmildsolution} follows from~\cite{Grecksch1995} where a similar result is provided for a SDE driven by a general Hilbert space-valued martingale. 
\end{rem}
Consider two types of SDEs in  Hilbert space with the same initial condition $x_{0}$:
\begin{eqnarray}\label{dual2}
	\mathrm{d}Y(t) = (AY(t) + F(t, Y(t)))\mathrm{d}t + B(t, Y(t))\mathrm{d}W_{t},
\end{eqnarray}
and 
\begin{eqnarray}\label{dual1}
	\mathrm{d}X(t) = (AX(t) + F(E_{t}, X(t)))\mathrm{d}E_{t} + B(E_{t}, X(t))\mathrm{d}W_{E_{t}}.
\end{eqnarray}
The following theorem establishes a deep connection between the classic SDE (\ref{dual2}) and the time-changed SDE (\ref{dual1}).
\begin{thm}(Duality of SDEs in Hilbert Space)\label{Duality}
	Let $U_{t}$ be a $\beta$-stable subordinator and $E_{t}$ be the inverse of $U_{t}$, which is a finite $\tilde{\mathcal{F}}_{E_t}$-measurable time-change.
	\begin{itemize}
		\item If an H-valued process $Y(t)$ satisfies SDE~\eqref{dual2}, then the H-valued process $X(t) := Y(E_{t})$ satisfies SDE~\eqref{dual1}.
	\end{itemize}
	\begin{itemize}
		\item If an H-valued process $X(t)$ satisfies SDE~\eqref{dual1}, then the H-valued process $Y(t) := X(U_{t-})$ satisfies SDE~\eqref{dual2}.
	\end{itemize}
\end{thm}
\begin{proof}
	First, consider the integral form of SDE~\eqref{dual2}:
	\begin{eqnarray}\label{sdenotimechangeinteduality}
\begin{aligned}
		Y(t) = x_{0} + \int_{0}^{t}\bigg(AY(s) + F(s, Y(s))\bigg)\mathrm{d}s\\ 
+ \int_{0}^{t}B(s, Y(s))\mathrm{d}W_{s}.
\end{aligned}	
\end{eqnarray}
	Suppose the $H$-valued process $Y(t)$ satisfies~\eqref{sdenotimechangeinteduality}, and let $X(t) := Y(E_{t})$. Applying the first change of variable formula gives 
	\begin{eqnarray*}
		\begin{aligned}
			X(t) &= x_{0} + \int_{0}^{E_{t}}\bigg(AY(s) + F(s, Y(s))\bigg)\mathrm{d}s + \int_{0}^{E_{t}}B(s, Y(s))\mathrm{d}W_{s}\\
			&= x_{0} + \int_{0}^{t}\bigg(AY(E_{s}) + F(E_{s}, Y(E_{s}))\bigg)\mathrm{d}E_{s} + \int_{0}^{t}B(E_{s}, Y(E_{s}))\mathrm{d}W_{E_{s}}\\
			&= x_{0} + \int_{0}^{t}\bigg(AX(s) + F(E_{s}, X(s))\bigg)\mathrm{d}E_{s} + \int_{0}^{t}B(E_{s}, X(s))\mathrm{d}W_{E_{s}},\\
		\end{aligned}
	\end{eqnarray*}
	which is the corresponding integral form of SDE~\eqref{dual1}.
	
	Similarly, suppose the $H$-valued process $X(t)$ satisfies the SDE~\eqref{dual1}. Using the second change of variable formula yields 
	\begin{eqnarray*}
		\begin{aligned}
			X(t) &= x_{0} + \int_{0}^{t}\bigg(AX(s) + F(E_{s}, X(s))\bigg)\mathrm{d}E_{s} + \int_{0}^{t}B(E_{s}, X(s))\mathrm{d}W_{E_{s}}\\
			&= x_{0} + \int_{0}^{E_{t}}\bigg(AX(U_{s-}) + F(E_{U_{s-}}, X(U_{s-}))\bigg)\mathrm{d}s\\ 
			&\quad\qquad+ \int_{0}^{E_{t}}B(E_{U_{s-}}, X(U_{s-}))\mathrm{d}W_{s}\\
			&= x_{0} + \int_{0}^{E_{t}}\bigg(AX(U_{s-}) + F(s, X(U_{s-}))\bigg)\mathrm{d}s + \int_{0}^{E_{t}}B(s, X(U_{s-}))\mathrm{d}W_{s}.\\
		\end{aligned}
	\end{eqnarray*}
	Let $Y(t) := X(U_{t-})$, then
	\begin{eqnarray*}
		\begin{aligned}
			Y(t) 
			& = x_{0} + \int_{0}^{t}\bigg(AY(s) + F(s, Y(s)))\bigg)\mathrm{d}s + \int_{0}^{t}B(s, Y(s)))\mathrm{d}W_{s},
		\end{aligned}
	\end{eqnarray*}
which is the integral form of the SDE~\eqref{dual2}. 
\end{proof}

It is known from~\cite{Mandrecar2010} that under appropriate conditions, the strong solution, $Y(t)$, of the SDE~\eqref{dual2} exists and is unique in the form
\begin{eqnarray*}
	Y(t) = x_{0} + \int_{0}^{t}(AY(s) + F(s, Y(s)))\mathrm{d}s + \int_{0}^{t}B(s, Y(s))\mathrm{d}W_{s}
\end{eqnarray*}
for all $t \leq T$, $\mathbb{P}$-a.s. The existence and uniqueness of a strong solution to the time-changed SDE~\eqref{dual1} is then established based on the duality theorem.
\begin{thm}\label{sthm1}
Assume that the following hypotheses are satisfied:
\begin{enumerate}
	\item[1)] $W_{t}$ is a $K$-valued $Q$-Wiener process on a complete filtered probability space $(\Omega, \mathcal{F}, \{\mathcal{F}_{t}\}_{t\leq T}, \mathbb{P})$ with the filtration $\mathcal{F}_{t}$ satisfying the usual conditions and $E_{t}$ is the inverse $\beta$-stable subordinator which is independent of $W_{t}$.
	\item[2)] A is a linear bounded operator.
	\item[3)] The coefficients $F: \Omega \times [0,T] \times C([0,T],H) \rightarrow H$ and $B: \Omega \times [0,T] \times C([0,T],H) \rightarrow \mathcal{L}_2(K_Q,H)$, where $C([0,T],H)$ is the Banach space of $H$-valued continuous functions on $[0,T]$, satisfy the following conditions
	\begin{enumerate}
		\item $F$ and $B$ are jointly measurable, and for every $0 \leq t \leq T$, they are measurable with respect to the product $\sigma$-field $\mathcal{F}_t \times \mathcal{C}_t$ on $\Omega \times C([0,T], H)$, where $\mathcal{C}_t$ is a $\sigma$-field generated by cylinders with bases over $[0,t]$.
		\item There exists a constant $L$ such that for all $x \in C([0,T],H)$,
		\begin{eqnarray*}
			\qquad\qquad\|F(\omega,t,x)\|_H + \|B(\omega,t,x)\|_{\mathcal{L}_2(K_Q,H)} \leq L(1+ \sup_{0\leq s\leq T} \|x(s)\|_H)
		\end{eqnarray*}
        for $\omega \in \Omega$ and $0 \leq t\leq T$.
		\item For all $x,y \in C([0,T],H)$, $\omega \in \Omega$, $0 \leq t\leq T$, there exists $K_{0}>0$ such that
		\begin{eqnarray*}
			\begin{aligned}
			&\qquad\qquad\quad\|F(\omega,t,x)-F(\omega,t,y)\|_H + \|B(\omega,t,x)- B(\omega,t,y)\|_{\mathcal{L}_2(K_Q,H)}\\
			&\qquad\qquad\leq K_{0}\sup_{0\leq s\leq T} ||x(s)-y(s)||_H.
			\end{aligned}
		\end{eqnarray*}
     \end{enumerate}
	\item[4)] $\mathbb{E}\int_0^T \|B(t,Y(t))\|^2_{\mathcal{L}_2(K_Q,H)}dt < \infty$.
	\item[5)] $Y(t)$ is in the domain of $A$ $\mathrm{d}\mathbb{P} \times\mathrm{d}t$-almost everywhere.
	\item[6)] $x_0$ is an $\mathcal{F}_0$-measurable $H$-valued random variable.
	\end{enumerate}
Then, the time-changed SDE~\eqref{dual1} has a unique strong solution, $X(t)$, satisfying
\begin{eqnarray}\label{intstrongsolnoftimechanged}
	X(t) = x_{0} + \int_{0}^{t}(AX(s) + F(E_{s}, X(s)))\mathrm{d}E_{s} + \int_{0}^{t}B(E_{s}, X(s))\mathrm{d}W_{E_{s}}.
\end{eqnarray}
\end{thm}
\begin{proof}
	From~\cite{Mandrecar2010}, based on conditions in Theorem~\ref{sthm1}, there is a unique solution $Y(t)$ that satisfies the SDE~\eqref{dual2}.  Moreover, it follows from Theorem~\ref{Duality} that $X(t):=Y(E_t)$ satisfies the SDE~\eqref{dual1}.  Therefore, there exists a solution to the time-changed SDE~\eqref{dual1}.
	
	Now suppose there exists another solution to the SDE~\eqref{dual1}.  Call this solution $\hat{X}(t)$.  Then, by Theorem ~\ref{Duality}, the process $\hat{Y}(t):=\hat{X}(U_{t-})$ is a solution to the SDE~\eqref{dual2}.  Since the solution to the SDE~\eqref{dual2} is unique from~\cite{Mandrecar2010}, it must be that $\hat{Y}(t)=Y(t)$.  Thus $\hat{X}(U_{t-})=Y(t)$ which implies that $\hat{X}(t)=Y(E_t)=X(t)$.  Therefore, the solution $X(t)$ of the SDE~\eqref{dual1} is unique and satisfies the desired integral equation~\eqref{intstrongsolnoftimechanged}.
\end{proof}
%
%
\section{Connections between Hilbert space-valued integrals driven by  time-changed $Q$-Wiener processes and time-changed cylindrical Wiener processes, respectively, and Walsh-type integrals }\label{connectionbetweenWalshandHilbert}

The objective of this section is to establish the equality of the three integrals given below for appropriate integrands:   
\begin{eqnarray*}
	\int_{0}^{T}\int_{\mathbb R^{N}}g(t,x) M_E(\mathrm{d}t,\mathrm{d}x) = \int_{0}^{T} g(t) \mathrm{d} \tilde{W}_{E_t}
	=  \int_{0}^{T}\Phi^{g}_t \circ J^{-1} \mathrm{d}W_{E_t},
\end{eqnarray*}
where $M_E$ is a time-changed version of a worthy martingale measure, $\tilde {W}_{E_t}$ is a time-changed version of a cylindrical Wiener process and $W_{E_t}$ is a time-changed Q-Wiener process. Section~\ref{connectionsonnontimechangcase} focuses on the case of no time change. Section~\ref{connectionsontimechangecase} provides the extension to the time-changed case.

\subsection{The case of no time change}\label{connectionsonnontimechangcase}
The equality of these integrals in the case where there is no time change are made in~\cite{Karczewska2005, Dalang2011}. The general idea is to first define a specific random field $F$ and an associated Hilbert space $K$. After identifying these objects, the integral with respect to the resulting cylindrical Wiener process and the integral with respect to the developed martingale measure are shown to be the same for a particular class of integrands. Also a connection between the cylindrical Wiener process and a Q-Wiener process is made that leads to an equality of their respective integrals. Thus, a combination of those results proves that all three integrals are equal. 

Let $\{F(\phi)\ |\ \phi \in C_{0}^{\infty}(\mathbb{R}^{+}\times\mathbb{R}^{N})\}$ be a family of mean zero Gaussian random variables, called a Gaussian random field. The covariance of $F$ is given by
\begin{eqnarray}\label{covarianceofrandomfield}
	\mathbb{E} (F(\phi)F(\psi))=\int_{\mathbb R^+}\int_{\mathbb R^N}\int_{\mathbb R^N}\phi(s, x)f(x-y)\psi(s, y)\mathrm{d}y\mathrm{d}x\mathrm{d}s,
\end{eqnarray}
where $f$ satisfies the following two conditions:
\begin{enumerate}
	\item[1)] $f$ is a non-negative, non-negative definite continuous function on $\mathbb R^{N}\backslash \{0\}$,
	which is integrable in a neighborhood of $0$;
	\item[2)] for all $\gamma \in \mathcal{S}(\mathbb{R}^N)$, the space of $C^\infty$ functions which are rapidly decreasing along with all their derivatives, the following conditions hold:
	\begin{enumerate}
		\item[a)] there exists a tempered measure $\mu$ on $\mathbb{R}^{N}$ such that for any $m\in \mathbb{N}^+$\\ $\displaystyle\int_{\mathbb{R}^N} (1+|\xi|^2)^{-m} \mu(\mathrm{d}\xi) < \infty$, and 
		\item[b)] $\displaystyle\int_{\mathbb{R}^N} f(x) \gamma(x) \mathrm{d}x =\int_{\mathbb{R}^N} \mathcal{F}\gamma(\xi) \mu (\mathrm{d}\xi)$, where $\mathcal{F}\gamma$ is the Fourier transform of $\gamma$.
	\end{enumerate}
\end{enumerate}
For the remainder of Section~\ref{connectionbetweenWalshandHilbert}, fix the specific random field $F$ chosen in~\eqref{covarianceofrandomfield}.
We now define a Hilbert space associated with the fixed random field $F$.  Let $K$ be the completion of the Schwartz space $\mathcal{S}(\mathbb{R}^N)$ with the semi-inner product
\begin{eqnarray}\label{definitionofHilbertspaceKinconnectionsection}
\begin{aligned}
	\langle \gamma, \psi \rangle_K := \int_{\mathbb R^N}\int_{\mathbb R^N}\gamma(x)f(x-y)\psi(y)\mathrm{d}y\mathrm{d}x\\
= \int_{\mathbb{R}^N} \mu(\mathrm{d}\xi)\mathcal{F}\gamma(\xi)\overline{\mathcal{F}\psi(\xi)},
\end{aligned}
\end{eqnarray}
where $\gamma, \psi \in K$, and associated semi-norm $||\cdot||_K$. Then $K$ is a Hilbert space, see~\cite{Dalang1999}. In the remaining of this section, the Hilbert space $K$ always refers to~\eqref{definitionofHilbertspaceKinconnectionsection}.

Moreover, after fixing a time interval $[0, T]$, it is possible to consider the set $K_{T} := L^{2}([0, T],; K)$ with the norm 
\begin{eqnarray*}
	\|g\|^{2}_{K_{T}} = \int_{0}^{T}\|g(s)\|^{2}_{K}\mathrm{d}s.
\end{eqnarray*}
Note that $C^{\infty}_{0}([0, T]\times\mathbb{R}^{N})$ is dense in $K_{T}$. It should also be noted that although $F$ was originally defined on smooth, compactly supported functions, it is possible to extend $F$ to functions of the form $1_{[0, T]}(\cdot)\varphi(*)$ where $\varphi\in\mathcal{S}(\mathbb{R}^{N})$. This follows since such an $F$ is a random linear functional such that $\gamma\mapsto F(\gamma)$ is an isometry from $(C^{\infty}_{0}([0, T]\times\mathbb{R}^{N}), \|\cdot\|_{K_{T}})$ into $L^{2}(\Omega, \mathcal{F}, \mathbb{P})$, i.e., a family of mean zero Gaussian random variables characterized by~\eqref{covarianceofrandomfield}. For further details, see~\cite{Dalang2011}.

\begin{defn}
	A cylindrical Wiener process on a Hilbert space $K$ as defined in~\eqref{definitionofHilbertspaceKinconnectionsection} is a family of random variables \{$\tilde{W}_t, t\geq 0\}$ such that:
	\begin{enumerate}
		\item for each $h \in K$, $\{\tilde W_t(h), t\geq 0 \}$ defines a Brownian motion with mean 0 and variance $t \langle h, h \rangle_K$;
		\item for all $s, t \in \mathbb{R}^+$ and $h,g \in K$, $\mathbb{E}(\tilde W_s(h) \tilde W_t(g))= (s \wedge t) \langle h, g \rangle_K$ where $s \wedge t:= \min(s,t)$.
	\end{enumerate}
\end{defn}
From~\cite{Dalang2011}, the stochastic process $\{\tilde{W}_{t}, t\geq 0\}$ defined in terms of the fixed random field $F$ with covariance chosen in~\eqref{covarianceofrandomfield} is given by
\begin{eqnarray}\label{definitionofcylindricalprocess}
\tilde W_t(\varphi):=F(1_{[0,t]}(\cdot)\varphi(*)),~\text{for}~\varphi\in K,
\end{eqnarray}
is a cylindrical Wiener process on the Hilbert space $K$. A complete orthonormal basis $\{f_j\}$ can be chosen such that $\{f_j\} \subset \mathcal{S}(\mathbb{R}^N)$ since $\mathcal{S}(\mathbb{R}^N)$ is a dense subspace of $K$. Consider the space $L^2(\Omega \times [0,T]; K)$ of predictable processes $g$ such that
\begin{eqnarray*}
	\mathbb{E}\left(\int_0^T ||g(s)||_K^2 \mathrm{d}s\right) < \infty.
\end{eqnarray*}
For $g\in L^2(\Omega \times [0,T]; K)$, the stochastic integral in Hilbert space $H$ with respect to the cylindrical Wiener process $\tilde{W}_{t}$ is defined as
\begin{eqnarray*}
	\int_0^Tg(s)\mathrm{d}\tilde W_s := \sum_{j=1}^{\infty} \int_0^T \langle g(s), f_j \rangle_K \mathrm{d}\tilde W_s(f_j)
\end{eqnarray*}
where particularly $H = \mathbb{R}$ and $\{f_j\}$ is an orthonormal basis of $K$ in~\eqref{definitionofHilbertspaceKinconnectionsection}.  The series is convergent in $L^2(\Omega, \mathcal{F}, \mathbb{P})$, and the sum does not depend on the choice of orthonormal basis. Additionally, the following isometry holds:
\begin{eqnarray*}
	\mathbb{E}\left(\left(\int_0^Tg(s)\mathrm{d}\tilde W_s\right)^2\right)= \mathbb{E}\left(\int_0^T ||g(s)||_K^2 \mathrm{d}s\right).
\end{eqnarray*}

On the other hand, consider $M_t$ defined by
\begin{eqnarray}\label{expressionofmartingalemeasure}
M_t(A):=F(1_{[0,t]}(\cdot)1_{A}(*)), ~~~ t \in [0,T], A \in \mathcal B_b(\mathbb R^N),
\end{eqnarray}
where $F$ is the specific random field chosen in~\eqref{covarianceofrandomfield} and $ \mathcal B_b(\mathbb R^N)$ denotes the set of bounded Borel sets of $\mathbb R^N$. The covariance of $\{M_t(A)\}$ is given by: 
\begin{eqnarray*}
	\begin{aligned}
		\mathbb E (M_t(A)M_t(B)) &=  \int_{\mathbb R^+}\int_{\mathbb R^N}\int_{\mathbb R^N}1_{[0,t]}(s)1_{A}(x)f(x-y)1_{B}(y)\mathrm{d}y\mathrm{d}x\mathrm{d}s\\
		&=  t\int_{\mathbb R^N}\int_{\mathbb R^N}1_{A}(x)f(x-y)1_{B}(y)\mathrm{d}y\mathrm{d}x.
	\end{aligned}
\end{eqnarray*}
Therefore, $M_{t}$ is a martingale measure in the following sense.
\begin{defn}(\cite{Dalang2011} )\label{definitionofmartingalemeasure} 
	A process $\{M_t(A)\}_{t\geq 0,A\in\mathcal B(\mathbb R^N)}$ is a martingale measure with respect to $\{\mathcal F_t\}_{t\geq 0}$ if:
	\begin{enumerate}
		
		\item for all $A\in\mathcal B(\mathbb R^{N})$, $M_{0}(A)=0$ a.s.;
		
		\item for $t>0$, $M_t$ is a sigma-finite $L^2(\mathbb{P})$-valued signed measure; and
		
		\item for all $A\in\mathcal B(\mathbb R^{N})$, $\{M_t(A)\}_{t\geq 0}$ is a mean-zero martingale with respect to the filtration $\{\mathcal F_t\}_{t\geq 0}$.
		
	\end{enumerate}
\end{defn}
Further, in order to define stochastic integrals with respect to a martingale measure, the martingale measure needs to be worthy.

\begin{defn}(\cite{Dalang2011})\label{definitionofworthymartingalemeasure} A martingale measure $M$ is worthy if there exists a random sigma-finite measure $K(A\times B\times C, \omega)$, where $A,B\in\mathcal B(\mathbb R^N)$, $C\in \mathcal B(\mathbb R_+)$, and $\omega\in\Omega$, such that:
\begin{enumerate}
\item $A\times B\times C\mapsto K(A\times B\times C,\omega)$ is nonnegative definite and symmetric;
		
\item $\{K(A\times B\times (0,t])\}_{t\geq 0}$ is a predictable process for all $A,B\in\mathcal B(\mathbb R^N)$;
		
\item for all compact sets $A,B\in\mathcal B(\mathbb R^N)$ and $t>0$,
	\begin{eqnarray*}
		\mathbb E(K(A\times B\times (0,t]))<\infty;
	\end{eqnarray*}

\item for all $A,B\in\mathcal B(\mathbb R^N)$ and $t>0$,
	\begin{eqnarray*}
		|\mathbb E(M_t(A)M_t(B))|\leq \mathbb K(A\times B\times (0,t])\text{ a.s. }
	\end{eqnarray*}
\end{enumerate}
	
\end{defn}
In particular, the stochastic integral with respect to a worthy martingale measure is defined in such a way that it is itself a martingale measure.
First, consider elementary processes $g$ of the form
\begin{eqnarray}\label{definitionofelementprocess}
g(s,x,\omega)=1_{(a,b]}(s)1_{A}(x)X(\omega)
\end{eqnarray}
where $0\leq a<b\leq T$, $X$ is bounded and $\mathcal F_a$-measurable, and $A\in\mathcal B(\mathbb R^N)$.
If $g$ is an elementary process as in~\eqref{definitionofelementprocess}, then define $g\cdot M$ by
\begin{eqnarray*}
	\begin{aligned}
		g\cdot M_t(B) := \int_0^t \int_B g(s,x) M(\mathrm{d}s, \mathrm{d}x)\\
		=X(\omega)\left(M_{t\wedge b}(A\cap B)-M_{t\wedge a}(A\cap B)\right).
	\end{aligned}
\end{eqnarray*}
The definition of $g\cdot M$ can be extended by linearity to simple processes, which are finite sums of elementary processes.
Let $\mathcal P_+$ denote the set of predictable processes $(\omega,t,x)\mapsto g(t,x;\omega)$ such that 
\begin{eqnarray*}
	\|g\|_+^2:=\mathbb E\left(\int_{0}^{T}\int_{\mathbb R^N}\int_{\mathbb R^N}|g(t,x)|~f(x-y)~|g(t,y)|\mathrm{d}y\mathrm{d}x\mathrm{d}t\right)<\infty.
\end{eqnarray*}
Taking limits of simple processes, the definition of $g\cdot M$ extends to all $g\in \mathcal P_+$.
From~\cite{Karczewska2005}, $g\cdot M$ is a worthy martingale measure if $g\in\mathcal P_+$. Therefore, it makes sense to define the stochastic integral with respect to $M$ as a martingale measure in the following way:
\begin{eqnarray*}
	\int_{0}^{t}\int_{A}g(s,x)M(\mathrm{d}s, \mathrm{d}x)=:g\cdot M_t(A).
\end{eqnarray*}

The next result is Proposition 2.6 in~\cite{Dalang2011}; the detailed proof is given there.

\begin{prop}(\cite{Dalang2011})\label{theoremonconnectionofmartingaleandcylindrical}
Suppose $g \in \mathcal{P}_{+}$.  Then $g \in L^2(\Omega \times [0,T]; K)$ and
\begin{eqnarray*}
	\int_{0}^{T}\int_{\mathbb{R}^N}^{}g(t,x)M(\mathrm{d}t, \mathrm{d}x) = \int_{0}^{T}g(t)\mathrm{d}\tilde W_t,
\end{eqnarray*}
where $M$ is the worthy martingale measure defined in~\eqref{expressionofmartingalemeasure} and $\tilde W_{t}$ is the cylindrical Wiener process defined in~\eqref{definitionofcylindricalprocess}.
\end{prop}

Finally, Proposition~\ref{theoremonconnectioofcylindricalprocessandQWiener}  below provides conditions under which these integrals coincide with the integral with respect to a Q-Wiener process. Define the operator $J:K \rightarrow K$ by
\begin{eqnarray}\label{definitionofoperatorJ}
	J(h):=\sum_{j=1}^{\infty} \lambda_j^{1/2}\langle h, f_j\rangle_K f_j,~~~ h \in K,
\end{eqnarray}
where $\{f_j\}$ is an orthonormal basis in $K$ and $\lambda_{j} \geq 0$ satisfies 
\[
\sum_{j=1}^{\infty} \lambda_j < \infty.
\]
\begin{lem}\label{definitionofQfromoperatorJ}
	Let $Q=JJ^*: K\rightarrow K$. Then $Q$ has eigenvalues $\lambda_j$ corresponding to the eigenvectors $f_j$, i.e. $Qf_j=\lambda_jf_j$.
\end{lem}

\begin{proof}
Let $f_{j}$ be the orthonormal basis as in~\eqref{definitionofoperatorJ}, then
	\begin{eqnarray*}
		\begin{aligned}
			Qf_{j}=(JJ^*)(f_{j})&=\sum_{l=1}^{\infty} \lambda_l^{1/2}\langle J^*f_{j}, f_l \rangle_K f_l = \sum_{l=1}^{\infty}\lambda_l^{1/2}\langle f_{\l}, Jf_j \rangle_K f_l\\
			&= \sum_{l=1}^{\infty}\lambda_l^{1/2}\langle f_{j}, \sum_{i=1}^{\infty} \lambda_i^{1/2} \langle f_l, f_i \rangle_K f_i \rangle_K f_l\\
			&= \sum_{l=1}^{\infty}\lambda_l\langle f_{j}, f_l \rangle_K f_l=\lambda_{j}f_{j}.
		\end{aligned}
	\end{eqnarray*}
\end{proof}

Additionally, $Q$ is symmetric (self-adjoint), non-negative definite, and $trQ=\sum_{j=1}^{\infty} \lambda_j < \infty$. The operator $J: K \rightarrow K_Q$ is an isometry since
\begin{eqnarray*}
	\|h\|_K = \|Q^{-1/2}J(h)\|_K = \|J(h)\|_{K_Q},~~~ h \in K,
\end{eqnarray*}
where $Q^{-1/2}$ denotes the pseudo-inverse of $Q^{1/2}$. Further, the inverse operator $J^{-1}: K_Q=Q^{1/2}K \rightarrow K$ is also an isometry. Therefore, taking $\{w_j(t)\}_{j=1}^{\infty}$ to be the family of independent Brownian motion processes defined by \ref{definitionofcylindricalprocess}, a Q-Wiener process in $K$ as defined in Section~\ref{timechangedQwienerprocess} can be constructed as
\begin{eqnarray}\label{specificdefinitionofQWienerprocess}
	W_t:=\sum_{j=1}^{\infty} w_j(t)J(f_j)=\sum_{j=1}^{\infty} w_j(t)\lambda_j^{1/2}f_j.
\end{eqnarray}

As seen in Section~\ref{stointegralwithtimechangedQwiener}, a predictable process $\phi$ will be integrable with respect to the $Q$-Wiener process $W_t$ if
\begin{eqnarray}\label{integralconditionofQWienerprocess}
\mathbb{E}\left(\int_0^T ||\phi(t)||^2_{\mathcal{L}_2(K_Q, H)}\mathrm{d}t\right) < \infty,
\end{eqnarray}
Consider $g \in L^2(\Omega \times [0,T]; K)$, and define the operator,~$\Phi_s^g$~:$K \rightarrow \mathbb{R}$, by
\begin{eqnarray}\label{defintionofoperatorPhi}
	\Phi_s^g(\eta)=\langle g(s), \eta \rangle_K,~~~ \eta \in K.
\end{eqnarray}
The following proposition is from Dalang and Quer-Sardanyons~\cite{Dalang2011}; the proof is included here since it contains information that will be used later.
\begin{prop}\label{theoremonconnectioofcylindricalprocessandQWiener}
	If  $g \in \mathcal{P_+}$ and $\Phi_t^g$ as defined in~\eqref{defintionofoperatorPhi}, then $\Phi_t^g \circ J^{-1}$ satisfies the condition~\eqref{integralconditionofQWienerprocess} and
	$$\int_0^T \Phi_t^g \circ J^{-1}\mathrm{d}W_t = \int_0^T g(t)\mathrm{d}\tilde W_t,$$
	where $W_{t}$ is the $Q$-Wiener process defined in~\eqref{specificdefinitionofQWienerprocess}, $J$ is the operator defined in~\eqref{definitionofoperatorJ} and $\tilde{W}_{t}$ is the cylindrical Wiener process defined in~\eqref{definitionofcylindricalprocess}.
\end{prop}
\begin{proof}
First, to show that $\Phi_t^g \circ J^{-1} \in \mathcal{L}_2(K_Q, \mathbb{R})$, note that
	\begin{eqnarray*}
		\begin{aligned}
			||\Phi_t^g \circ J^{-1}||^2_{\mathcal{L}^2(K_Q, \mathbb{R})} &= \sum_{j=1}^{\infty} [(\Phi_t^g \circ J^{-1})(\lambda_j^{1/2} f_j)]^2 = \sum_{j=1}^{\infty} [\Phi_t^g (J^{-1}\lambda_j^{1/2} f_j)]^2\\
			&= \sum_{j=1}^{\infty} \langle g(t), J^{-1}Q^{1/2} f_j \rangle_K^2 =\sum_{j=1}^{\infty} \langle g(t), f_j \rangle_K^2\\
			&=||g(t)||_K^2.
		\end{aligned}
	\end{eqnarray*}
	Therefore,
\begin{eqnarray*}
	\mathbb{E}\left(\int_0^T ||\Phi_t^g \circ J^{-1}||^2_{\mathcal{L}_2(K_Q, \mathbb{R})}\mathrm{d}t\right)= \mathbb{E}\left(\int_0^T ||g(t)||_K^2\mathrm{d}t\right) < \infty,
\end{eqnarray*}
	since $g \in \mathcal{P}_{+}$ implies that $g \in L^2(\Omega \times [0,T]; K)$.  Thus, $\Phi_t^g \circ J^{-1}$ satisfies condition~\eqref{integralconditionofQWienerprocess}. Also, from Lemma~\ref{stochintwithqwiener},
	\begin{align*}
	\int_0^T \Phi_t^g \circ J^{-1} dW_t &= \sum_{j=1}^{\infty} \int_0^T (\Phi_t^g \circ J^{-1})(\lambda_j^{1/2} f_j)\mathrm{d}\langle W_t, \lambda_j^{1/2} f_j \rangle_{K_Q}\\
	&= \sum_{j=1}^{\infty} \int_0^T \langle g(t), f_j \rangle_K \mathrm{d}w_j(t)\\
	&= \sum_{j=1}^{\infty} \int_0^T \langle g(t), f_j \rangle_K\mathrm{d}\tilde W_t(f_j)\\
	&=\int_0^T g(t)\mathrm{d}\tilde W_t.
	\end{align*}
\end{proof}
Therefore, combining Propositions~\ref{theoremonconnectionofmartingaleandcylindrical} and~\ref{theoremonconnectioofcylindricalprocessandQWiener} yields the desired integral connections. 

\begin{cor}
	For $g \in \mathcal{P_+}$ and $\Phi_t^g$ as defined in~\eqref{defintionofoperatorPhi}, 
	\begin{eqnarray*}
		\int_{0}^{T}\int_{\mathbb{R}^N}^{}g(t,x)M(\mathrm{d}t, \mathrm{d}x) = \int_{0}^{T}g(t)\mathrm{d}\tilde W_t = \int_0^T \Phi_t^g \circ J^{-1}\mathrm{d}W_t,
	\end{eqnarray*}
where $M$ is the martingale measure defined in~\eqref{expressionofmartingalemeasure}, $\tilde{W}_{t}$ is the cylindrical Wiener process defined in~\eqref{definitionofcylindricalprocess}, $J$ is the operator defined in ~\eqref{definitionofoperatorJ} and $W_{t}$ is the $Q$-Wiener process defined in~\eqref{specificdefinitionofQWienerprocess}.
\end{cor}

\subsection{The case of a time change}\label{connectionsontimechangecase}
The previous section summarized results from~\cite{Karczewska2005,Dalang2011, Dalang2009} establishing the equality of integrals with respect to a martingale measure, a cylindrical Wiener process, and a Q-Wiener process.  This section will extend those results to the time-changed case.  The procedure for showing the equivalence of the integrals will be very similar to that used in the previous section.  The same random field $F$ and Hilbert space $K$ are used to define time-changed versions of a cylindrical Wiener process and a martingale measure.  Their associated integrals are then shown to be equal.  Finally, a connection between the given time-changed cylindrical Wiener process and a time-changed Q-Wiener process leads to the equality of all three integrals.

First, recall that the time change,~$E_{t}$, is the inverse of a $\beta$-stable subordinator. Then, the time-changed cylindrical Wiener process is defined as follows.

\begin{defn}\label{definitionoftimechangedcylindricalprocess}
 Let $K$ be a separable Hilbert space. A family of random variables $\{\tilde W_{E_t},t\geq 0\}$
	is a time-changed cylindrical Wiener process on $K$ if the following conditions hold:
	\begin{enumerate}
		\item for any $k\in K$, $\{\tilde W_{E_t}(k),t\geq 0\}$ defines a time-changed Brownian motion with mean 0 and variance $\mathbb E(E_t)\langle k,k \rangle_{K}$; and
		
		\item for all $s,t\in R_{+}$ and $k,h\in K$,
		\begin{eqnarray*}
			\mathbb E (\tilde W_{E_s}(k)\tilde W_{E_t}(h))=\mathbb E(E_{s\wedge t})\langle k,h \rangle_{K}.
		\end{eqnarray*}
	\end{enumerate}
\end{defn}
Let $g: R^{+}\times \Omega\rightarrow K$ be any predictable process such that
\begin{eqnarray}\label{spaceofsquareintegrablefunctionwithrespecttotimchange}
\mathbb E\left(\int_{0}^{T}\|g(s)\|_{K}^2\mathrm{d}E_s\right)<\infty.
\end{eqnarray}
Consider the series 
\begin{eqnarray*}
	\sum_{j=1}^{\infty}\int_{0}^{T}\langle g_s, f_j\rangle_{K}\mathrm{d}\tilde W_{E_s}(f_j).
\end{eqnarray*}
Convergence of this series in $L^2(\Omega, \mathcal G, \mathbb{P})$ is established in Proposition \ref{convergenceofseriesinthedefinitionoftimechangedcylindricalintegral}. This justifies defining 
the stochastic integral of $g$ with respect to a time-changed cylindrical Wiener process as follows:
\begin{eqnarray*}
	\int_0^T g(s)\mathrm{d}\tilde W_{E_s}:=\sum_{j=1}^{\infty}\int_{0}^{T}\langle g_s, f_j\rangle_{K}\mathrm{d}\tilde W_{E_s}(f_j).
\end{eqnarray*}

\begin{prop}\label{convergenceofseriesinthedefinitionoftimechangedcylindricalintegral}
	 Let $g:\mathbb R^{+}\times \Omega\rightarrow K$ in~\eqref{definitionofHilbertspaceKinconnectionsection} be any predictable process satisfying condition~\eqref{spaceofsquareintegrablefunctionwithrespecttotimchange}. Then, the series
\begin{eqnarray}\label{definitionoftimechangedcylindricalintegral}
     \sum_{j=1}^{\infty}\int_{0}^{T}\langle g_s, f_j\rangle_{K}\mathrm{d}\tilde W_{E_s}(f_j)
\end{eqnarray}
converges in $L^2(\Omega, \mathcal G, \mathbb{P})$. 	
\end{prop}

\begin{proof}
Let 
\begin{eqnarray*}
	Y_n:=\sum_{j=1}^{n} \int_0^T \langle g_s, f_j \rangle_K\mathrm{d}\tilde W_{E_s}(f_j).
\end{eqnarray*}
In order to show the convergence of the series defined in~\eqref{definitionoftimechangedcylindricalintegral}, it is sufficient to show $\{Y_n\}$ is a Cauchy sequence in $L^2(\Omega, \mathcal{G}, \mathbb{P})$.  For $n>m$,
\begin{eqnarray*}
	\begin{aligned}
		||Y_n &- Y_m||_2^2= \mathbb{E} \bigg( \sum_{j=m+1}^{n} \int_0^T \langle g_s, f_j \rangle_K \mathrm{d}\tilde W_{E_s}(f_j) \bigg)^2\\
		&= \mathbb{E} \bigg[\bigg( \sum_{j=m+1}^{n} \int_0^T \langle g_s, f_j \rangle_K \mathrm{d}\tilde W_{E_s}(f_j) \bigg)\bigg( \sum_{i=m+1}^{n} \int_0^T \langle g_s, f_i \rangle_K \mathrm{d}\tilde W_{E_s}(f_i) \bigg)\bigg]\\
		&= \mathbb{E} \bigg( \sum_{j=i=m+1}^{n} \bigg[ \int_0^T \langle g_s, f_j \rangle_K\mathrm{d}\tilde W_{E_s}(f_j) \bigg]^2 \bigg)\\
		&\qquad + \mathbb{E} \bigg( \sum_{j\neq i=m+1}^{n} \bigg[ \int_0^T \langle g_s, f_j \rangle_K \mathrm{d}\tilde W_{E_s}(f_j) \bigg]\bigg[ \int_0^T \langle g_s, f_i \rangle_K\mathrm{d}\tilde W_{E_s}(f_i) \bigg] \bigg)\\
		&=: I + II.
	\end{aligned}
\end{eqnarray*}

Since the time-changed $Q$-Wiener process $W_{E_{t}}$ defined in~\eqref{defoftimechangedQwiener} is a square integrable martingale with respect to the filtration~$\mathcal{G}_{t} = \tilde{\mathcal{F}}_{E_t}$ defined in~\eqref{generalizationoffiltration}, for $h = \lambda^{-1/2}_{j}f_{j} \in K, j = 1, 2, \cdots$, $\langle W_{E_{t}}, h\rangle_{K}$ is also a square integrable martingale with respect to the filtration~$\mathcal{G}_{t}$, i.e., for $0 < s < t$,
\begin{eqnarray*}
    \mathbb{E}(\langle W_{E_{t}}, h\rangle_{K}|\mathcal{G}_{s}) = \langle W_{E_{s}}, h\rangle_{k} .
\end{eqnarray*}
 This implies that each projection, which is a time-changed Brownian motion, $w_{j}(E_{t}),~j = 1, 2, \cdots$, is also a square integrable martingale with respect to the same filtration $\mathcal{G}_{t}$, i.e., for $0<s<t$, $\displaystyle\mathbb{E}(w_{j}(E_{t})|\mathcal{G}_{s}) = w_{j}(E_{s})$. Therefore, the integral,
\begin{eqnarray*}
	\int_0^T \langle g_s, f_j \rangle_K\mathrm{d}w_j(E_s),
\end{eqnarray*}
is also a square integral martingale with respect to the filtration~$\mathcal{G}_{t}$, and it follows from the It\^o isometry that
\begin{eqnarray*}
\mathbb{E} \bigg[ \int_0^T \langle g_s, f_j \rangle_K\mathrm{d}w_j(E_s)\bigg]^2 = \mathbb{E} \bigg[ \int_0^T \langle g_s, f_j \rangle_K^2\mathrm{d}E_s \bigg].
\end{eqnarray*}
Further, by a proof similar to that of Theorem~\ref{martingaleoftimechagedQwiener}, the product $w_{i}(E_{t})w_{j}(E_{t})$ is a square integrable martingale for $i\neq j$. This means the quadratic covariation process of martingales $w_{i}(E_{t})$ and $w_{j}(E_{t})$ is zero, i.e., $[w_{i}(E_{t}), w_{j}(E_{t})] = 0$. Thus, using the martingale property of $w_j(E_s)$ and its associated integral, along with the Cauchy-Schwartz inequality
	\begin{eqnarray}\label{boundofsquareintegralofcauchy}
		\begin{aligned}
				I&:= \mathbb{E} \bigg( \sum_{j=i=m+1}^{n} \bigg[ \int_0^T \langle g_s, f_j \rangle_K\mathrm{d}\tilde W_{E_s}(f_j) \bigg]^2 \bigg)\\
				&=  \sum_{j=i=m+1}^{n} \mathbb{E} \bigg[ \int_0^T \langle g_s, f_j \rangle_K\mathrm{d}w_j(E_s)\bigg]^2\\
				&=  \sum_{j=i=m+1}^{n} \mathbb{E} \bigg[ \int_0^T \langle g_s, f_j \rangle_K^2\mathrm{d}E_s \bigg] \leq \sum_{j=i=m+1}^{\infty}\mathbb{E} \bigg[ \int_0^T \langle g_s, f_j \rangle_K^2\mathrm{d}E_s \bigg].
		\end{aligned}
	\end{eqnarray}
	By assumption~\eqref{spaceofsquareintegrablefunctionwithrespecttotimchange},
	\begin{eqnarray*}
		\sum_{j=1}^{\infty}\mathbb{E} \bigg[ \int_0^T \langle g_s, f_j \rangle_K^2\mathrm{d}E_s \bigg] = \mathbb{E} \bigg[ \int_0^T \| g_s \|_K^2\mathrm{d}E_s \bigg] < \infty,
	\end{eqnarray*} 
	so the above tail of the partial sum in~\eqref{boundofsquareintegralofcauchy} converges to $0$ as $m~ (\text{hence}~n)\to\infty$.
	Meanwhile, 
	\begin{eqnarray}\label{boundcrossintegralsofcauchy}
		\begin{aligned}
				II&:= \mathbb{E} \bigg( \sum_{j\neq i=m+1}^{n} \bigg[ \int_0^T \langle g_s, f_j \rangle_K\mathrm{d}\tilde W_{E_s}(f_j) \bigg]\bigg[ \int_0^T \langle g_s, f_i \rangle_K \mathrm{d}\tilde W_{E_s}(f_i) \bigg] \bigg)\\
				&=\sum_{j\neq i=m+1}^{n} \mathbb{E} \bigg[\bigg( \int_0^T \langle g_s, f_j \rangle_K\mathrm{d}\tilde W_{E_s}(f_j) \bigg)\bigg( \int_0^T \langle g_s, f_i \rangle_K\mathrm{d}\tilde W_{E_s}(f_i) \bigg) \bigg]\\
				&=\sum_{j\neq i=m+1}^{n} \mathbb{E} \bigg[\bigg( \int_0^T \langle g_s, f_j \rangle_K\mathrm{d}w_j(E_s) \bigg)\bigg( \int_0^T \langle g_s, f_i \rangle_K\mathrm{d}w_i(E_s) \bigg) \bigg]\\
				&= \sum_{j\neq i=m+1}^{n} \mathbb{E} \bigg( \int_0^T \langle g_s, f_j \rangle_K \langle g_s, f_i \rangle_K\mathrm{d}[w_j(E_s),w_i(E_s)] \bigg)\\
				&= 0.
		\end{aligned}
	\end{eqnarray}
Combining~\eqref{boundofsquareintegralofcauchy} and~\eqref{boundcrossintegralsofcauchy} yields $||Y_n-Y_m||_2^2 = I + II\to 0$ as $n, m\to\infty$. Thus, $\{Y_n\}$ is a Cauchy sequence in $L^2(\Omega, \mathcal{G}, \mathbb{P})$, completing the proof.
\end{proof}

The next proposition shows how to define a cylindrical process from the fixed random field $F$ chosen in~\eqref{covarianceofrandomfield} .

\begin{prop}\label{definitionoftimechangedmartingalemeasureproposition} 
	For $t\geq 0$ and $\phi\in K$, set 
	\begin{eqnarray}\label{definitionoftimechangedcylindricalprocessexpression}
	\tilde W_{E_t}(\phi):=F(1_{[0,E_t]}(\cdot)\phi(*)).
	\end{eqnarray}
	Then, the process $\tilde W_{E_t}$ is a time-changed cylindrical Wiener process. 
\end{prop}

\begin{proof}
Consider a fixed $\phi \in K$.  $\tilde W_{E_t}(\phi)$ is a time-changed Brownian motion by construction. Additionally,
	\begin{eqnarray*}
		\begin{aligned}
			\mathbb{E}[\tilde W_{E_t}(\phi)] &= \mathbb{E}[F(1_{[0,E_t]}(\cdot)\phi(*))] = \int_0^{\infty} \mathbb{E} [F(1_{[0,\tau]}(\cdot)\phi(*))]f_{E_t}(\tau) \mathrm{d}\tau\\
			&= \int_0^{\infty} 0 \cdot f_{E_t}(\tau) \mathrm{d}\tau=0,
		\end{aligned}
	\end{eqnarray*}
and 
	\begin{eqnarray*}
		\begin{aligned}
				\mathbb{E}[\tilde W_{E_t}(\phi)&\tilde W_{E_t}(\phi)] = \mathbb{E}[F(1_{[0,E_t]}(\cdot)\phi(*))F(1_{[0,E_t]}(\cdot)\phi(*))]\\
				~~~&= \mathbb{E} \left[ \int_{\mathbb{R}^{+}} 1_{[0,E_t]}(s) 1_{[0,E_t]}(s) \int_{\mathbb{R}^{N}}\int_{\mathbb{R}^{N}} \phi(x) f(x-y) \phi(y) \mathrm{d}y \mathrm{d}x \mathrm{d}s \right] \\
				~~~&= \mathbb{E} \left[ \langle \phi, \phi \rangle_K \int_{\mathbb{R}^{+}} 1_{[0,E_t]}(s) \mathrm{d}s\right]
				= \mathbb{E}(E_t)\langle \phi, \phi \rangle_K.
		\end{aligned}
	\end{eqnarray*}
Further, for fixed $s,t \in \mathbb{R}^+$ and $\phi, \psi \in K$,
\begin{eqnarray*}
	\begin{aligned}
		\mathbb{E}[\tilde W_{E_t}(\phi) &\tilde W_{E_s}(\psi)] = \mathbb{E}[F(1_{[0,E_t]}(\cdot)\phi(*))F(1_{[0,E_s]}(\cdot)\psi(*))]\\
		~~~&=\mathbb{E} \int_{\mathbb{R}^+} 1_{[0,E_t]}(r)1_{[0,E_s]}(r) \int_{\mathbb{R}^N} \int_{\mathbb{R}^N} \phi(x)f(x-y)\psi(y) dy dx dr\\
		~~~&=\mathbb{E} \int_{\mathbb{R}^+} \langle \phi, \psi \rangle_K 1_{[0, E_t \wedge E_s]}(r) dr 
		=\mathbb{E}(E_{t\wedge s})\langle \phi, \psi \rangle_K.
	\end{aligned}
\end{eqnarray*}
Therefore, according to the Definition~\ref{definitionoftimechangedcylindricalprocess}, $\tilde W_{E_t}$ is a time-changed cylindrical Wiener process. 
\end{proof}
\noindent Moreover, it follows from Proposition \ref{convergenceofseriesinthedefinitionoftimechangedcylindricalintegral} that the integral with respect to the process $\tilde W_{E_t}$ defined in~\eqref{definitionoftimechangedcylindricalprocessexpression} is well-defined. 

On the other hand, $E_t$ is independent of the martingale measure $M_t(A)$. So, define a time-changed version of $\{M_t(A)\}$ by
\begin{eqnarray}\label{definitionoftimechangedmartingalemeasure}
M_{E_t}(A)= M(A\times [0,E_t]):=F(1_{[0,E_t]}(\cdot)1_{A}(*)).
\end{eqnarray}
By conditioning on the time change, the covariance for $ M_{E_t}(A)$ is
\begin{eqnarray*}
	\begin{aligned}
		\mathbb E (M_{E_t}(A)  &M_{E_t}(B)) = \int_{0}^{\infty}\mathbb E( M_{\tau}(A)M_{\tau}(B))f_{E_t}(\tau)\mathrm{d}\tau\\
		& =  \int_{0}^{\infty}\tau\left(\int_{\mathbb R^N}\int_{\mathbb R^N}1_{A}(x)f(x-y)1_{B}(y)\mathrm{d}y\mathrm{d}x\right)f_{E_t}(\tau)\mathrm{d}\tau\\
		& =  \mathbb E(E_t)\int_{\mathbb R^N}\int_{\mathbb R^N}1_{A}(x)f(x-y)1_{B}(y)\mathrm{d}y\mathrm{d}x,
	\end{aligned}
\end{eqnarray*}
where $f_{E_t}$ is the density function of $E_t$. Also
\begin{eqnarray*}
	\begin{aligned}
		\mathbb E[ M_{E_t}(A)]^2 = \mathbb E(E_t)\int_{\mathbb R^N}\int_{\mathbb R^N}1_{A}(x)f(x-y)1_{A}(y)\mathrm{d}y\mathrm{d}x < \infty
	\end{aligned}
\end{eqnarray*}
for all $A\in\mathcal{B}_b(\mathbb R^N)$. Thus, $M_{E_t}(A)$ has a finite second moment for all $A\in\mathcal{B}_b(\mathbb R^N)$.
Then the following theorem shows that  $\{M_{E_t}(A)\}$ is also a martingale measure.

\begin{thm}\label{timechangedmartingalemeasure}
	$\{M_{E_t}(A)\}_{t\geq 0, A\in\mathcal B(\mathbb R^N)}$ is a martingale measure with respect to the filtration $\{\tilde{\mathcal F}_{E_t}\}_{t\geq 0}$, where $\tilde{\mathcal F}_{t}$ in~\eqref{generalizationoffiltration} is generated by the time change $E_{t}$ and independent Brownian motions $\tilde{W}_{t}(f_{j}), j = 1, 2, \cdots$ defined by \ref{definitionofcylindricalprocess}.
\end{thm}

\begin{proof} 
It suffices to check the conditions in Definition~\ref{definitionofmartingalemeasure}. First, since $E_0=0$ a.s., $M_{E_0}(A)=M_{0}(A)=0$ a.s. because $M_{t}(A)$ is a martingale measure. 

Second, let $A$, $B\in \mathcal B(\mathbb R^N)$ be disjoint. Then, for fixed $\tau$, $M_\tau(A\cup B)$ and $M_\tau (A)+M_\tau (B)$ are mean zero Gaussian random variables and
	\begin{eqnarray*}
		\begin{aligned}
				\text{Var}(&M_{\tau}(A\cup B))  =  \underset{(A\cup B)\times (A\cup B)}{\int\int}f(x-y)\mathrm{d}y\mathrm{d}x\\
				& =  \underset{A\times A}{\int\int}f(x-y)dydx+\underset{B\times B}{\int\int}f(x-y)dydx+2\underset{A\times B}{\int\int}f(x-y)\mathrm{d}y\mathrm{d}x\\
				& =  \text{Var}(M_{\tau}(A))+\text{Var}(M_{\tau}(B))+2\mathbb E(M_{\tau}(A)M_{\tau}(B))\\
				& = \text{Var}\bigg(M_{\tau}(A)+M_{\tau}(B)\bigg).\\
		\end{aligned}
	\end{eqnarray*}
Also note that
\begin{eqnarray*}
		M_{\tau}(A\cup B)=M_{\tau}(A)+M_{\tau}(B)~\text{a.s.}
\end{eqnarray*}
Thus, conditioning on the time change yields
\begin{eqnarray*}
	\begin{aligned}
			\mathbb{P}&(M_{E_t}(A\cup B) =  M_{E_t}(A)+M_{E_t}(B))\\
			&= \int_{0}^{\infty}\mathbb{P}(M_\tau(A\cup B)=M_{\tau}(A)+M_{\tau}(B))f_{E_t}(\tau)\mathrm{d}\tau\\
			&=  \int_{0}^{\infty}f_{E_t}(\tau)\mathrm{d}\tau=1,
	\end{aligned}
\end{eqnarray*}
which means $M_{E_{t}}(\cdot)$ is additive a.s. Furthermore, assume $A_1\supset A_2\supset...$ such that ${\cap}_n A_n=\emptyset$, 
\begin{eqnarray*}
	\mathbb E[M_{E_t}(A_{n})]^2=\mathbb E(E_{t})\underset{A_n\times A_n}{\int\int}f(x-y)\mathrm{d}y\mathrm{d}x\rightarrow 0
\end{eqnarray*}
as $n\to \infty$, and so $M_{E_t}(A_n)\to 0$ in $L^2(\mathbb{P}).$  This proves the countable additivity of $M_{E_{t}}(\cdot)$.
	
Finally, since $\{M_{E_t}(A)\}$ has a finite second moment, a similar argument to the proof of Theorem~\ref{martingaleoftimechagedQwiener} shows that $\{M_{E_t}(A)\}$ is a martingale for all $A\in\mathcal B(\mathbb R^N)$. Thus, $\{M_{E_t}(A)\}$ is a martingale measure with respect to the filtration $\mathcal{G}_{t} = \tilde{\mathcal{F}}_{E_t}$.	
\end{proof}

Define the dominating measure $K$ by
\begin{eqnarray}\label{dominatingmeasureexpresion}
\begin{aligned}
	&\quad K(A\times B\times C)\\ &:=\mathbb E(\lambda(\{E_s(\omega):s\in C\}))\int_{\mathbb R^N}\int_{\mathbb R^N}1_{A}(x)f(x-y)1_{B}(y)\mathrm{d}y\mathrm{d}x,
\end{aligned}
\end{eqnarray}
where $A,B\in \mathcal B(\mathbb R^N)$ and $\lambda$ is the Lebesgue measure on $C\in \mathcal B(\mathbb R_+)$ . Then, the following theorem shows that the martingale measure $\{M_{E_t}(A)\}$ is worthy.
\begin{thm} 
	The martingale measure $\{M_{E_t}(A)\}_{t\geq 0, A\in\mathcal B(\mathbb R^N)}$ is worthy with respect to the filtration $\{\tilde{\mathcal{F}}_{E_t}\}_{t\geq 0}$, i.e. the same one as defined in Theorem~\ref{timechangedmartingalemeasure}.
\end{thm}

\begin{proof} To show that the martingale measure $\{M_{E_t}(A)\}$ is worthy, it suffices to show that the dominating measure $K$ defined in~\eqref{dominatingmeasureexpresion} satisfies the conditions of Definition~\ref{definitionofworthymartingalemeasure}.
\begin{enumerate}
	\item For all $C\in\mathcal B(\mathbb R^{+})$, 
	\begin{eqnarray*}
		\begin{aligned}
			K(A\times B\times C)& =  \mathbb E(\lambda(\{E_s(\omega):s\in C\}))\int_{\mathbb R^N}\int_{\mathbb R^N}1_{A}(x)f(x-y)1_{B}(y)\mathrm{d}y\mathrm{d}x\\
			& =  \mathbb E(\lambda(\{E_s(\omega):s\in C\}))\int_{\mathbb R^N}\int_{\mathbb R^N}1_{B}(x)f(x-y)1_{A}(y)\mathrm{d}y\mathrm{d}x\\
			& =  K(B\times A\times C).
		\end{aligned}
	\end{eqnarray*}
Additionally, for all $A\in\mathcal B(\mathbb R^N)$, $C\in \mathcal B(\mathbb R^{+})$, since $f$ is non-negative definite,
\begin{eqnarray*}
	\begin{aligned}
		K(A\times A\times C) = \mathbb E(\lambda(\{E_s(\omega):s\in C\}))\int_{\mathbb R^N}\int_{\mathbb R^N}1_{A}(x)f(x-y)1_{A}(y)\mathrm{d}y\mathrm{d}x\geq 0.
	\end{aligned}
\end{eqnarray*}
\item For all $A,B\in\mathcal B(\mathbb R^N)$, $t>0$, 
\begin{eqnarray*}
	K(A\times B\times (0,t])=\mathbb E(E_t)\int_{\mathbb R^N}\int_{\mathbb R^N}1_{A}(x)f(x-y)1_{B}(y)\mathrm{d}y\mathrm{d}x
\end{eqnarray*}
is $\tilde{\mathcal F}_{E_t}$-measurable.
\item For all compact sets $A,B\in\mathcal B(\mathbb R^N)$ and $t>0$,
\begin{eqnarray*}
	\mathbb E |K(A\times B\times (0,t])|=\mathbb E(E_t)\int_{\mathbb R^N}\int_{\mathbb R^N}1_{A}(x)f(x-y)1_{B}(y)\mathrm{d}y\mathrm{d}x < \infty.
\end{eqnarray*}
\item For all $A,B\in\mathcal B(\mathbb R^N)$ and $t>0$,
\begin{eqnarray*}
	\begin{aligned}
		|\mathbb E(M_{E_t}(A)M_{E_t}(B))| & =  \mathbb E(E_t)\int_{\mathbb R^N}\int_{\mathbb R^N}1_{A}(x)f(x-y)1_{B}(y)\mathrm{d}y\mathrm{d}x\\
		& =  K(A\times B\times (0,t]).
	\end{aligned}	
\end{eqnarray*}
\end{enumerate}
Thus, the martingale measure $\{M_{E_t}(A)\}$ is worthy. 
\end{proof}

As shown in the non-time-changed case, the stochastic integral with respect to a worthy time-changed martingale measure is also a worthy martingale measure. First, consider elementary processes $g$ of the form
\begin{eqnarray}\label{elementprocessoftimechangecase}
g(s,x,\omega)=1_{(a,b]}(s)1_{A}(x)X(\omega)
\end{eqnarray}
where $0\leq a < b \leq T$, $A \in \mathcal{B}(\mathbb{R}^N)$ and $X$ is both bounded and $\mathcal{G}_a:=\tilde{\mathcal{F}}_{E_a}$-measurable.  
For $t,r\in [0,T]$, let 
\begin{eqnarray*}
M_{E_t\wedge E_r}:=M_{E_{t\wedge r}},
\end{eqnarray*}
and let $M_{E}(\mathrm{d}s,\mathrm{d}x)$ denote integration with respect to the martingale measure in the both $s$ and $x$.
Then, define $g \cdot M_E$ by
\begin{eqnarray*}
	\begin{aligned}
		g\cdot M_{E_t}(B) &:= X(\omega)(M_{E_t \wedge E_b}(A \cap B)- M_{E_t \wedge E_a}(A \cap B))\\
		&~= \int_{0}^{t}\int_{\mathbb R^{N}}g(s,x) M_E(\mathrm{d}s,\mathrm{d}x).
	\end{aligned}
\end{eqnarray*}

As usual, this definition of $g\cdot M_E$ can be extended to finite sums of elementary processes and finally to predictable processes $g$ such that
\begin{eqnarray}\label{predictableprocessfortimechange}
\begin{aligned}
||g||_{\dagger}^2
& :=\mathbb E\left(\int_{0}^{T}\int_{\mathbb R^N}\int_{\mathbb R^N}|g(t,x)|f(x-y)|g(t,y)|\mathrm{d}y\mathrm{d}x\mathrm{d}E_t\right)\\
& < \infty.
\end{aligned}
\end{eqnarray}
Let ${\mathcal P}_{\dagger}$ denote the set of predictable processes $(\omega,t,x) \mapsto g(t,x;\omega)$ such that~\eqref{predictableprocessfortimechange} holds. For $g\in {\mathcal P}_{\dagger}$, $g\cdot M_E$ is a worthy martingale measure and the stochastic integral with respect to $M_E$ is defined by
\begin{eqnarray*}
	\int_0^t \int_A g(s,x)M_E(ds, dx)=: g \cdot M_{E_t}(A).
\end{eqnarray*}
The following theorem connects an integral with respect to a time-changed martingale measure with an integral with respect to a time-changed cylindrical Wiener process.

\begin{thm}\label{connectionbetweentimechangedmartingalemeasureandcylindricalprocesses} Let $M_{E_{t}}$ be the time-changed martingale measure defined in~\eqref{definitionoftimechangedmartingalemeasure} and $\tilde W_{E_t}$ be the time-changed cylindrical process defined in~\eqref{definitionoftimechangedcylindricalprocessexpression}. For $g\in {\mathcal P}_{\dagger}$, then,
\begin{eqnarray*}
   \int_{0}^{T}\int_{\mathbb R^{N}}g(t,x)M_E(\mathrm{d}t, \mathrm{d}x) = \int_0^Tg(t)\mathrm{d}\tilde W_{E_t}. 
\end{eqnarray*}
\end{thm}

\begin{proof} 
	First notice that since $g \in \mathcal P_{\dagger}$,
	\begin{eqnarray*}
		\begin{aligned}
			\mathbb{E}\left(\int_0^T ||g(s)||_K^2 dE_s\right) &= \mathbb{E}\left( \int_0^T \int_{\mathbb{R}^N}\int_{\mathbb{R}^N} g(s,x) f(x-y) g(s,y) \mathrm{d}y\mathrm{d}x\mathrm{d}E_s \right) \\
			&\leq ||g||^2_{\dagger} < \infty,	
		\end{aligned}
	\end{eqnarray*}
which means $g$ satisfies the condition~\eqref{spaceofsquareintegrablefunctionwithrespecttotimchange}. Further, since the set of elementary processes is dense in ${ \mathcal P}_{\dagger}$, it suffices to check that the integrals coincide for elementary processes of the form
\begin{eqnarray*}
	g(s,x,\omega)=1_{(a,b]}(s)1_{A}(x)X(\omega)
\end{eqnarray*}
where $0\leq a<b\leq T$, $A\in\mathcal B(\mathbb R^N)$ and $X(\omega)$ is both bounded and $\tilde{\mathcal{F}}_{E_{a}}$-measurable. For this, note that
\begin{eqnarray*}
	\begin{aligned}
		\int_{0}^{T}\int_{\mathbb R^{N}}g(t,x)M_E(\mathrm{d}t, \mathrm{d}x) & = X(M_{E_T \wedge E_b}(A)-M_{E_T \wedge E_a}(A))\\
		& = X(M_{E_b}(A)-M_{E_a}(A))\\
		& =X(F(1_{[0,E_b]}(\cdot)1_{A}(*))-F(1_{[0,E_a]}(\cdot)1_{A}(*)))\\
		& =X(F(1_{(E_a,E_b]}(\cdot)1_{A}(*))).
	\end{aligned}
\end{eqnarray*}
On the other hand, using the linearity of $F$
\begin{eqnarray*}
	\begin{aligned}
		\int_{0}^{T}g(t)d\tilde W_{E_t} & =  \sum_{j=1}^{\infty} \int_a^b X\langle 1_A, f_j\rangle_{K}d\tilde W_{E_t}(f_j)\\
		& =  X\sum_{j=1}^{\infty}\langle 1_{A},f_j\rangle_{K}(\tilde W_{E_b}(f_j)-\tilde W_{E_a}(f_j))\\
		& =  X\sum_{j=1}^{\infty}\langle 1_A,f_j\rangle_{K}[F(1_{[0,E_b]}(\cdot)f_j)-F(1_{[0,E_a]}(\cdot)f_j)]\\
		& =  X\sum_{j=1}^{\infty}\langle 1_A,f_j\rangle_{K}[F(1_{(E_a,E_b]}(\cdot)f_j)]\\
		& =  X[F(1_{(E_a,E_b]}(\cdot)\sum_{j=1}^{\infty}\langle 1_A,f_j\rangle_{K}f_j]\\
		& =   X[F(1_{(E_a,E_b]}(\cdot)1_{A}(*))].
	\end{aligned}
\end{eqnarray*}
\end{proof}
Next, a connection between the time-changed cylindrical Wiener process and the time-changed $Q$-Wiener process will be established. Define
\begin{eqnarray}\label{definitionoftimechangedQwienerprocessfromoperatorJ}
W_{E_t}:=\sum_{j=1}^{\infty}w_j(E_t)J({f_j})
\end{eqnarray}
where $w_{j}(E_t) := \tilde W_{E_t}(f_j)$ are the time-changed Brownian motions defined  in Proposition~\ref{definitionofQfromoperatorJ}. Also for $g\in L^{2}(\Omega\times [0,T], K)$, define the operator $\Phi^{g}_s$ by
\begin{eqnarray}\label{definitionofPhioftimechanged}
\Phi^{g}_s(\eta):=\langle g(s), \eta \rangle_{K}\text{ }\eta\in K. 
\end{eqnarray}
A predictable process $\phi$ will be integrable with respect to $W_{E_t}$ if
\begin{eqnarray}\label{conditionofintegrabilityoftimechangedQwienerprocess}
\mathbb{E}\left(\int_0^T ||\phi(t)||^2_{\mathcal{L}_2(K_Q, H)}dE_t\right) < \infty.
\end{eqnarray}
The next result provides the connection between the integral with respect to the time-changed cylindrical Wiener process and the time-changed $Q$-Wiener process.

\begin{thm}\label{connectionbetweentimechangedcylindricalintegralandQwienerintegral}
	Let $\tilde W_{E_t}$ be the time-changed cylindrical Wiener process as in ~\eqref{definitionoftimechangedcylindricalprocessexpression} and let $W_{E_t}$ be the time-changed $Q$-Wiener process as in ~\eqref{definitionoftimechangedQwienerprocessfromoperatorJ}. Let $g\in \mathcal P_{\dagger}$. Then, $\Phi_{s}^{g}\circ J^{-1}$ satisfies condition~\eqref{conditionofintegrabilityoftimechangedQwienerprocess} and
\begin{eqnarray*}
	\int_{0}^{T}\Phi_{s}^{g}\circ J^{-1}dW_{E_s} = \int_{0}^{T}g(s)d\tilde W_{E_s}.
\end{eqnarray*}
\end{thm}

\begin{proof}
First, from the proof of Proposition~\ref{theoremonconnectioofcylindricalprocessandQWiener}, $\|\Phi_{s}^{g}\circ J^{-1}\|^2_{\mathcal{L}_2(K_Q,\mathbb{R})} = \|g(s)\|_K^2$. Thus,
\begin{eqnarray*}
	\mathbb{E}\left(\int_0^T ||\Phi_t^g \circ J^{-1}||^2_{\mathcal{L}_2(K_Q, \mathbb{R})}\mathrm{d}E_t\right) = \mathbb{E}\left(\int_0^T ||g(t)||_K^2 \mathrm{d}E_t\right) < \infty,
\end{eqnarray*}
since it was shown in Theorem~\ref{connectionbetweentimechangedmartingalemeasureandcylindricalprocesses} that $g \in \mathcal{P}_{\dagger}$ implies that condition~\eqref{spaceofsquareintegrablefunctionwithrespecttotimchange} holds.  So, $\Phi_t^g \circ J^{-1}$ satisfies the condition~\eqref{conditionofintegrabilityoftimechangedQwienerprocess}. Also from Definition~\ref{timechangstochinte}, 
\begin{eqnarray*}
	\begin{aligned}
	   \int_{0}^{T}\Phi_{s}^{g}\circ J^{-1}dW_{E_s} & =\sum_{j=1}^{\infty}\int_{0}^{T}\Phi_{s}^{g}\circ J^{-1}(\lambda_j^{1/2}f_j)\mathrm{d}\langle W_s, \lambda_j^{1/2}f_j \rangle_{K_Q}\\
	   & = \sum_{j=1}^{\infty}\int_{0}^{T}\langle g(s), f_j \rangle_{K} \mathrm{d}w_{j}(E_s)\\
	  & = \sum_{j=1}^{\infty}\int_{0}^{T}\langle g(s), f_j \rangle_{K} \mathrm{d}\tilde W_{E_s}(f_j)\\
	  & = \int_{0}^{T}g(s)\mathrm{d}\tilde W_{E_s}.
	\end{aligned}
\end{eqnarray*}
\end{proof}
Finally, combining the results of Theorems~\ref{connectionbetweentimechangedmartingalemeasureandcylindricalprocesses} and~\ref{connectionbetweentimechangedcylindricalintegralandQwienerintegral} yields the desired correspondence of integrals with respect to time-changed processes.
\begin{cor}
	For $g \in \mathcal{P}_{\dagger}$ and $\Phi_t^g$ as defined in~\eqref{definitionofPhioftimechanged}, 
\begin{eqnarray*}
	\int_{0}^{T}\int_{\mathbb{R}^N}^{}g(t,x)M_E(dt, dx) = \int_0^Tg(t)\mathrm{d}\tilde W_{E_t} = \int_0^T \Phi_t^g \circ J^{-1} dW_{E_t},
\end{eqnarray*}
where $M_{E_{t}}$ is the time-changed martingale measure as in~\eqref{definitionoftimechangedmartingalemeasure}, $\tilde{W}_{E_{t}}$ is the time-changed cylindrical process as in~\eqref{definitionoftimechangedcylindricalprocessexpression}, $J$ is the operator defined in ~\eqref{definitionofoperatorJ} and $W_{E_t}$ is the time-changed $Q$-Wiener process as in~\eqref{definitionoftimechangedQwienerprocessfromoperatorJ}.
\end{cor}

\section{Fokker-Planck-Kolmogorov equations associated with the time-changed stochastic differential equations }\label{fpeinhilbertspace}
  
In this section, the Fokker-Plank-Kolmogorov (FPK) equations corresponding to sub-diffusion processes in Hilbert space are introduced. A FPK equation is a deterministic differential equation whose solution is the probability density function for a stochastic process. 
A fundamental example of such an equation in finite dimensions is the heat equation, whose solution is the density for Brownian motion.
These equations are important for several reasons. As in the heat equation, these equations are often helpful in understanding a scientific phenomenon. The connection of these equations to a stochastic process allows one to use information about the stochastic process to study these phenomena. Conversely, knowing FPK equations corresponding to a particular stochastic process is helpful in the simulation of this stochastic process, see~\cite{Hahn2012,Kei2011,Hahn2011139}. More recently, the FPK equations corresponding to diffusion processes in infinite dimensional Hilbert spaces have been analyzed, see~\cite{Bogachev2009,Bogachev2010,Bogachev2011}. 

An introduction to the sub-diffusion processes on Hilbert spaces considered here and their corresponding FPK equations requires a preliminary discussion of diffusion processes on Hilbert space and their FPK equations.
Consider the following classic SDE driven by the $Q$-Wiener process  
\begin{eqnarray}\label{SDEqprocess}
\left\{
\begin{array}{ll}
dY(t) = [AY(t) + F(t, Y(t))]\mathrm{d}t + C\mathrm{d}W_{t} \\
\\
Y(0) = x\in{H}\\
\end{array} 
\right.
\end{eqnarray}
where $A: D(A)\subset H \to H$ is the infinitesimal generator of a $C_{0}$-semigroup $S(t) = e^{tA}$, $t\geq 0$, in $H$, and $W_{t}$ is a $K$-valued $Q$-Wiener process on a complete filtered probability space $(\Omega, \mathcal{F}, \{\mathcal{F}_{t}\}_{t\leq T}, \mathbb{P})$ with the filtration $\mathcal{F}_{t}$ satisfying the usual conditions.  Suppose that $F:\Omega\times[0, T]\times H\rightarrow H$, $C :\Omega\times K\rightarrow H$
and $C\in\Lambda_{2}(K_{Q}, H)$. Also assume the initial value $x$ is an $\mathcal F_{0}$-measurable, $H$-valued random variable. Let $Y(t)$ be a strong solution of the SDE~\eqref{SDEqprocess} so that $Y(t)$ will satisfy the following integral equation: 
\begin{eqnarray*}
	Y(t) = x + \int_{0}^{t}[AY(s) + F(s, Y(s))]ds + CW_{t}.
\end{eqnarray*}
The Kolmogorov operator $L_{0}$ corresponding to the classic SDE~\eqref{SDEqprocess} is
\begin{eqnarray}\label{kolmogorovoperator}
\begin{aligned}
L_{0}\phi(x) &= \langle x, A^{*}D_{x}\phi(x)\rangle_{H} + \langle F(t, x), D_{x}\phi(x)\rangle_{H}\\ 
&+ \frac{1}{2}tr[(CQ^{1/2})(CQ^{1/2})^{*}D^{2}_{x}\phi(x)],
\end{aligned}
\end{eqnarray}
where $x\in H$, $t\in[0, T]$, and $D_{x}$, $D^{2}_{x}$ denote the first- and second-order Fr\'echet derivatives in space, respectively. $D(L_{0})$ denotes the domain of the operator $L_{0}$ and $A^{*}$ denotes the adjoint of the operator $A$. More details on the domain $D(L_{0})$ are given in \cite{Bogachev2009,Bogachev2010,Bogachev2011}. 

Let $\mu(\mathrm{d}t, \mathrm{d}x)$ be a product measure on $[0, T]\times H$ of the type
\begin{eqnarray*}
	\mu(\mathrm{d}t, \mathrm{d}x) = \mu_{t}(\mathrm{d}x)\mathrm{d}t,
\end{eqnarray*}
where $\mu_{t} \in \mathbb{P}(H)$ is a Borel probability measure on the Hilbert space $H$ for all $t\in[0, T]$. 
Let $P^{Y}_{t}$ be the transition evolution operator on $\mathcal{B}_{b}(H)$, the space of bounded, Borel-measurable functions on $H$, defined by
\begin{eqnarray}\label{evoloperatorofY}
P_{t}^{Y}\phi(x) = \mathbb{E}(\phi(Y(t))|Y(0) = x),\text{ }0\leq{t}\leq{T},\text{ }\phi\in{\mathcal{B}_{b}(H)},
\end{eqnarray}
and let $(P^{Y}_{t})^{*}$ be its adjoint operator. Note that $(P^Y_t)$ is a semigroup generated by the Markov process $Y(t)$.
According to \cite{Bogachev2011}, it is possible to define the measure $\mu^{Y}_{t}$ induced by the solution $Y(t)$ as
\begin{eqnarray}\label{probmeasurebyY}
\mu_{t}^{Y}(dy) := (P^{Y}_{t})^{*}\xi(dy),
\end{eqnarray}
where $\xi\in\mathbb{P}(H)$ is the measure associated with the initial value $x$. 
The induced  measure, $\mu_{t}^{Y}(dy)$, is defined as 
\begin{eqnarray}\label{meaningofevoloperatorofY}
\int_{H}\phi(y)\mu_{t}^{Y}(dy) = \int_{H}P^{Y}_{t}\phi(y)\xi(dy),~\textrm{for all}~\phi\in\mathcal{B}_{b}(H).
\end{eqnarray} 
Under the assumption
\begin{eqnarray*}
	\int_{[0, T]\times H}\bigg(|\langle y, A^{*}h\rangle_{H}| + \|F(t, y)\|_{H}\bigg)\mu(\mathrm{d}t, \mathrm{d}y) < \infty,
\end{eqnarray*}
where $h\in D(A^{*})$, the induced measure $\mu_{t}^{Y}$ satisfies  the following FPK  equation 
\begin{eqnarray}\label{intfpeofsde}
\begin{aligned}
\frac{\mathrm{d}}{\mathrm{d}t}\int_{H}\phi(y)\mu^{Y}_{t}(dy) &=\int_{H}L_{0}\phi(y)\mu^{Y}_{t}(dy),~\textrm{for}~dt\textrm{-a.e.,}~t\in[0, T],
\end{aligned}
\end{eqnarray}
where the initial condition is
\begin{eqnarray}\label{initialconditionofclassicfpe}
	\lim_{t\to 0}\int_{H}\phi(y)\mu^{Y}_{t}(\mathrm{d}y) = \int_{H}\phi(y)\xi(dy).
\end{eqnarray}
Further, if the domain of the Kolmogorov operator $L_{0}$ is comprised of test functions, integration by parts yields
\begin{eqnarray}\label{derfpeofsde}
\frac{\partial}{\partial t}\mu^{Y}_{t} = L_{0}^{*}\mu^{Y}_{t}, ~\mu^{Y}_{0} = \xi.
\end{eqnarray}
Further details are given in ~\cite{Bogachev2011}.

The following lemma is needed to extend the FPK equations~\eqref{intfpeofsde} and~\eqref{initialconditionofclassicfpe} or~\eqref{derfpeofsde} associated to the solution of the classic SDE~\eqref{SDEqprocess} to the case of the solution to an SDE driven by a time-changed $Q$-Wiener process.
\begin{lem}\label{laplaceoftimechange}
	Let $U_{\beta}(t)$ be a $\beta$-stable subordinator with the cumulative distribution function $F_{\tau}(t) = \mathbb{P}(U_{\beta}(\tau) \leq t)$ and density function $f_{\tau}(t)$. Suppose the inverse of $U_\beta(t)$ is $E_{t}$ with the density function $f_{E_{t}}(\tau)$. Then, for any integrable function $h(\tau)$ on $(0, \infty)$, the function $q(t)$ defined by the following integral
	\begin{eqnarray*}
		q(t) := \int_0^\infty f_{E_t}(\tau) h(\tau) d\tau
	\end{eqnarray*}
    has Laplace transform
	\begin{eqnarray*}
		\mathcal{L}_{t\to s}\{q(t)\} = s^{\beta-1}[\widetilde{h(\tau)}](s^\beta),
	\end{eqnarray*}
	where $\widetilde{h(\tau)}(s) = \mathcal{L}_{\tau\to s}\{h(\tau)\}$.
\end{lem}
\begin{proof}
	Using the self-similarity property of the $\beta$-stable subordinator $U_{\beta}(t)$, the distribution function~$F_{E_{t}}(\tau)$ associated with the time change $E_{t}$ is
	\begin{eqnarray}\label{distofinvsubordinator}
	\begin{aligned}
	F_{E_t}(\tau)
	&=\mathbb{P}(E_{t} \leq \tau)
	= \mathbb{P}(U_\beta(\tau) > t)
	=1-\mathbb{P}(\tau^{1/\beta}U_\beta(1) \leq t)\\
	&= 1- \mathbb{P}\bigg(U_\beta(1) \leq \frac{t}{\tau^{1/\beta}}\bigg) = 1- F_1\bigg(\frac{t}{\tau^{1/\beta}}\bigg).
	\end{aligned}
	\end{eqnarray}
	Differentiating both sides of~\eqref{distofinvsubordinator},
	\begin{eqnarray*}
	f_{E_t}(\tau)= -\frac{\partial}{\partial\tau} \left\{ \frac{1}{\tau^{1/\beta}}(Jf_1)\bigg(\frac{t}{\tau^{1/\beta}}\bigg)\right\},
	\end{eqnarray*}
	where $J$ is an integral operator defined by 
	\begin{eqnarray}\label{integrateoperatorJ}
		(Jf)\bigg(\frac{t}{a}\bigg) = \int_{0}^{t}f\bigg(\frac{s}{a}\bigg)\mathrm{d}s~\text{for all}~a > 0.
	\end{eqnarray}
	Therefore, the Laplace transform 
	\begin{eqnarray}\label{laplacetransformofintegrealrelatedtotimechange}
		\begin{aligned}
			\mathcal{L}_{t \rightarrow s}\bigg[\frac{1}{a} (Jf_1)\bigg(\frac{t}{a}\bigg)\bigg](s)
			&= \frac{1}{a} \int_0^\infty (Jf_1)\bigg(\frac{t}{a}\bigg) e^{-st} dt = \frac{1}{as} \int_0^\infty f_1\bigg(\frac{t}{a}\bigg) e^{-st} dt\\
			&= \frac{1}{s} \int_0^\infty  f_1(\hat{t}) e^{-as\hat{t}} d\hat{t}= \frac{1}{s} \tilde{f_1}(as),
		\end{aligned}
	\end{eqnarray}
	where $\widetilde{f_{1}}(s)$ is the Laplace transform of the function $f(t)$. Now consider
	\begin{align*}
	q(t) := \int_0^\infty f_{E_t}(\tau) h(\tau) d\tau
	= -\int_0^\infty \frac{\partial}{\partial\tau} \left\{ \frac{1}{\tau^{1/\beta}}(Jf_1)\bigg(\frac{t}{\tau^{1/\beta}}\bigg)\right\} h(\tau) d\tau.
	\end{align*}
   Applying the Laplace transform of the $\beta$-stable subordinator, $U_{\beta}(t)$, in~\eqref{laplacetransformofsubordinator} and using~\eqref{laplacetransformofintegrealrelatedtotimechange} yields
	\begin{align*}
	\mathcal{L}_{t\to s}\{q(t)\} &= -\int_0^\infty \frac{\partial}{\partial\tau}  \left\{ \frac{1}{s} e^{-\tau s^\beta} \right\} h(\tau) d\tau = s^{\beta-1}[\widetilde{h(\tau)}](s^\beta).
	\end{align*}
\end{proof}

Let $W_{E_{t}}$ be a time-changed $Q$-Wiener process on a complete filtered probability space $(\Omega, \mathcal{G}, \{\mathcal{G}_{t}\}_{t\leq T}, \mathbb{P})$ with the filtration $\mathcal{G}_{t} := \tilde{\mathcal{F}}_{E_{t}}$ satisfying the usual conditions. Suppose that $x$ is an $\mathcal F_{0}$-measurable and $H$-valued random variable. Let the coefficients $A, F,$ and $C$ be the same as in the classic SDE~\eqref{SDEqprocess}. Consider the following autonomous SDE on $H$ driven by $W_{E_t}$ and on the time interval $[0, T]$:
\begin{equation}\label{timechangedSDE}
\left\{ \begin{array}{rl}
&\mathrm{d}X(t) = (AX(t) + F(X(t)))\mathrm{d}E_{t} + C\mathrm{d}W_{E_{t}},\\

&\\

& X(0) = x \in H, ~~t\geq 0.
\end{array} \right.
\end{equation}
We derive the time-fractional FPK equation associated with the time-changed SDE~\eqref{timechangedSDE} via two different methods: first by applying the time-changed It\^o formula and second by using the duality in Theorem~\ref{Duality}. The advantage of the first approach is that it directly reveals the connection between the time-fractional FPK equations in Theorems~\ref{intfpeoftimeSDE} and~\ref{difffracfpe}, and the time-changed SDE~\eqref{timechangedSDE}. The advantage of the second approach is that it reveals the connection between the time-fractional FPK equations in Theorems~\ref{intfpeoftimeSDE} and~\ref{difffracfpe}, and the classic FPK equations in~\eqref{intfpeofsde} and~\eqref{initialconditionofclassicfpe} or~\eqref{derfpeofsde}.

\begin{thm}\label{intfpeoftimeSDE}
	Suppose the coefficients $A,F,$ and $C$ of the time-changed SDE \eqref{timechangedSDE} satisfy the conditions in Theorem~\ref{sthm1}. Let $X(t)$ be the solution to~\eqref{timechangedSDE}. Also suppose that $X(U_{\beta}(t))$ is independent of $E_{t}$. Then the probability kernel $\mu_{t}^{X}(\mathrm{d}x)$ induced by the solution $X(t)$ satisfies the following fractional integral equation 
	\begin{eqnarray}\label{fracintfpeoftimeSDE}
	D_{t}^{\beta}\int_{H}\phi(x)\mu^{X}_{t}(\mathrm{d}x) = \int_{H}L_{0}\phi(x)\mu^{X}_{t}(\mathrm{d}x),
	\end{eqnarray}
	with initial condition $\mu_{0}^{X}(\mathrm{d}x) = \xi(\mathrm{d}x)$, where $\phi\in D(L_{0})$, $L_{0}$ is the Kolmogorov operator defined in~\eqref{kolmogorovoperator} and $D_{t}^{\beta}$ denotes the Caputo fractional derivative operator as defined in Theorem~\ref{sub-diffusionoftimechangedQwienerprocess}.
\end{thm}
\begin{proof}(\textbf{Method via the time-changed It\^o formula})
Let $Y(t) := X(U_{t})$.
Since $E_{t}$ and $W_{E_{t}}$ are both constant on $[U_{t-}, U_{t}]$, the integrals 
\begin{eqnarray*}
	\int_{0}^{t}(AX(s) + F(X(s)))\mathrm{d}E_{s}~~\text{and}~~\int_{0}^{t}B(X(s))\mathrm{d}W_{E_{s}}
\end{eqnarray*}
are also constant on $[U_{t-}, U_{t}]$. Therefore, since $X(t)$ is the solution to the SDE ~\eqref{timechangedSDE}, $X(t)$ is also 
constant on $[U_{t-}, U_{t}]$ and satisfies $Y(E_{t}) = X(U_{E_{t}}) = X(t)$. So, for a fixed time $t\in [0, T]$, let $\mu^{X}_{t}$ and $\mu^{Y}_{t}$ denote the probability measures induced on $H$ by the stochastic processes $X(t)$ and $Y(t)$, respectively. Thus, for $\phi\in \mathrm{D}(L_{0})$,
	\begin{eqnarray}\label{expectionoffuncionofsoluon}
	\mathbb{E}(\phi(X(t))) = \int_{H}\phi(x)\mu^{X}_{t}(\mathrm{d}x).
	\end{eqnarray}
	Since $X(t) = Y(E_{t})$, taking the expectation of $X(t)$ conditioned on $E_{t}$,  
	\begin{eqnarray*}
	\begin{aligned}
	\mathbb{E}(\phi(X(t))) &= \mathbb{E}(\phi(Y(E_{t})))  = \int_{0}^{\infty}\mathbb{E}(\phi(Y_{\tau})|E_{t}=\tau)f_{E_{t}}(\tau)\mathrm{d}\tau,
	\end{aligned}
	\end{eqnarray*}
	where $f_{E_{t}}(\tau)$ is the density function of $E_{t}$. By the assumption that $Y(t) = X(U_{t})$ is independent of $E_{t}$, 
	\begin{equation}\label{expressionofexpectation}
\begin{aligned}	
\mathbb{E}(\phi(X(t))) & = \int_{0}^{\infty}\int_{H}\phi(x)\mu_{\tau}^{Y}(\mathrm{d}x)f_{E_{t}}(\tau)\mathrm{d}\tau\\
	& = \int_{H}\phi(x)\int_{0}^{\infty}\mu_{\tau}^{Y}(\mathrm{d}x)f_{E_{t}}(\tau)\mathrm{d}\tau.
\end{aligned}	
\end{equation}
	Since $\phi\in \mathrm{D}(L_{0})$ is arbitrary, combining~\eqref{expectionoffuncionofsoluon} and~\eqref{expressionofexpectation} yields
	\begin{eqnarray}\label{connectionofXmeasYmeas}
	\mu^{X}_{t}(\mathrm{d}x) = \int_{0}^{\infty}\mu_{\tau}^{Y}(\mathrm{d}x)f_{E_{t}}(\tau)\mathrm{d}\tau.
	\end{eqnarray}
Since $X(t)$ is constant on every interval $[U_{r-}, U_{r}]$, $X(U_{r-}) = X(U_{r}) = Y(r)$. Thus, by the time-changed It\^o formula,   
	\begin{eqnarray}\label{usingtimeitoformula}
	\begin{aligned}
	\phi(X(t)) - \phi(x_{0}) &= \int_{0}^{E_{t}}L_{0}\phi(X(U_{r-}))\mathrm{d}r + \int_{0}^{E_{t}}\langle\phi_{x}(X(U_{r-})), C\mathrm{d}W_{r}\rangle\\
	&= \int_{0}^{E_{t}}L_{0}\phi(Y(r))\mathrm{d}r + \int_{0}^{E_{t}}\langle\phi_{x}(Y(r)), C\mathrm{d}W_{r}\rangle.
	\end{aligned}
	\end{eqnarray}
	Since $\phi\in \mathrm{D}(L_{0})$, the integral 
	\begin{eqnarray*}
		M(\tau) = \int_{0}^{\tau}\langle \phi_{x}(Y(r)), C\mathrm{d}W_{r}\rangle_{H}
	\end{eqnarray*}
	is a square integrable $\mathcal{F}_{\tau}$-martingale. Taking expectations on both sides of~\eqref{usingtimeitoformula} and conditioning on $E_{t}$ gives 
	\begin{eqnarray}\label{expressionofcondtionexpectation}
	\begin{aligned}
	\mathbb{E}[\phi(X(t))&|X(0) = x_{0}] - \phi(x_{0})\\
	&= \int_{0}^{\infty}\mathbb{E}[\int_{0}^{\tau}L_{0}\phi(Y(r))\mathrm{d}r + M(\tau)|E_{t} = \tau, X(0) = x_{0}]f_{E_{t}}(\tau)\mathrm{d}\tau\\
	&=\int_{0}^{\infty}\int_{0}^{\tau}\mathbb{E}[L_{0}\phi(Y(r))| X(0) = x_{0}]\mathrm{d}r f_{E_{t}}(\tau)\mathrm{d}\tau\\
	&=\int_{0}^{\infty}\int_{0}^{\tau}\int_{H}L_{0}\phi(x)\mu^{Y}_{r}(\mathrm{d}x)\mathrm{d}r f_{E_{t}}(\tau)\mathrm{d}\tau\\
	&=\int_{0}^{\infty}\bigg(JP^{Y}(\tau)\bigg)f_{E_{t}}(\tau)\mathrm{d}\tau,
	\end{aligned}
	\end{eqnarray}
	where $J$ is the integral operator as defined in~\eqref{integrateoperatorJ} and $P^{Y}(r)$ is defined by
	\begin{eqnarray*}
		P^{Y}(r) = \int_{H}L_{0}\phi(x)\mu_{r}^{Y}(\mathrm{d}x).
	\end{eqnarray*}
 On the other hand, 
	\begin{eqnarray}\label{lefhandsideexpectation}
	\begin{aligned}
	\mathbb{E}[\phi(X(t))|X(0) = x_{0}] &- \mathbb{E}[\phi(x_{0})|X(0) = x_{0}]\\ &=\int_{H}\phi(x)\mu^{X}_{t}(\mathrm{d}x) - \int_{H}\phi(x)\xi(\mathrm{d}x),\\
	\end{aligned}
	\end{eqnarray}
	which also implies the initial condition $\mu_{0}^{X}(\mathrm{d}x) = \xi(\mathrm{d}x)$. 
Combining ~\eqref{expressionofcondtionexpectation}, ~\eqref{lefhandsideexpectation}, Lemma~\ref{laplaceoftimechange} yields
	\begin{eqnarray*}
		\begin{aligned}
			\mathcal{L}_{t\to s}\bigg\{\int_{H}\phi(x)\mu^{X}_{t}(\mathrm{d}x)\bigg\} &- \frac{1}{s}\int_{H}\phi(x)\xi(\mathrm{d}x)\\ &= s^{\beta - 1}[\widetilde{J_{t}P^{Y}(\tau)}](s^{\beta}) = \frac{s^{\beta - 1}}{s^{\beta}}[\widetilde{P^{Y}(\tau)}](s^{\beta}),\\
		\end{aligned}
	\end{eqnarray*}
	which, in turn, implies 
	\begin{eqnarray}\label{laplaceofexpectwithmeasure}
	s^{\beta}\mathcal{L}_{t\to s}\bigg\{\int_{H}\phi(x)\mu^{X}_{t}(\mathrm{d}x)\bigg\} - s^{\beta - 1}\int_{H}\phi(x)\xi(\mathrm{d}x) = s^{\beta - 1}[\widetilde{P^{Y}(\tau)}](s^{\beta}).
	\end{eqnarray}
	Combining ~\eqref{connectionofXmeasYmeas} and Fubini's Theorem yields 
	\begin{eqnarray}\label{integralwithrespecttoXYmeasures}
	\begin{aligned}
	    \int_{H}L_{0}\phi(x)\mu^{X}_{t}(\mathrm{d}x) &= \int_{H}L_{0}\phi(x)\int_{0}^{\infty}\mu_{\tau}^{Y}(\mathrm{d}x)f_{E_{t}}(\tau)\mathrm{d}\tau\\
	    &=\int_{0}^{\infty}\int_{H}L_{0}\phi(x)\mu_{\tau}^{Y}(\mathrm{d}x)f_{E_{t}}(\tau)\mathrm{d}\tau\\
	    &=\int_{0}^{\infty}P^{Y}(\tau)f_{E_{t}}(\tau)\mathrm{d}\tau.
	\end{aligned}
	\end{eqnarray}
	Taking Laplace transforms of both sides of~\eqref{integralwithrespecttoXYmeasures} gives
	\begin{eqnarray}\label{laplaceofintegralwithmeasure}
	\begin{aligned}
	\mathcal{L}_{t\to s}\bigg\{\int_{H}L_{0}\phi(x)\mu^{X}_{t}(\mathrm{d}x)\bigg\} 
	&= \mathcal{L}_{t\to s}\bigg\{\int_{0}^{\infty}P^{Y}(\tau)f_{E_{t}}(\tau)\mathrm{d}\tau\bigg\}\\ 
	&= s^{\beta - 1}[\widetilde{P^{Y}(\tau)}](s^{\beta}).
	\end{aligned}
	\end{eqnarray}
	Recall that the following equality holds:
	\begin{eqnarray}\label{laplacetransformofcaputoderivative}
		\mathcal{L}_{t\to s}\big\{D_{t}^{\beta}f(t)\big\} = s^{\beta}\mathcal{L}_{t\to s}\big\{f(t)\big\} - s^{\beta - 1}f(0),
	\end{eqnarray}
where $f(t)$ is a real-valued function on $t \geq 0$. Therefore, combining~\eqref{laplaceofexpectwithmeasure} and~\eqref{laplaceofintegralwithmeasure} yields
	\begin{eqnarray*}
		s^{\beta}\mathcal{L}_{t\to s}\bigg\{\int_{H}\phi(x)\mu^{X}_{t}(\mathrm{d}x)\bigg\} - s^{\beta - 1}\int_{H}\phi(x)\xi(\mathrm{d}x) = \mathcal{L}_{t\to s}\bigg\{\int_{H}L_{0}\phi(x)\mu^{X}_{t}(\mathrm{d}x)\bigg\},
	\end{eqnarray*}
	which, together with~\eqref{laplacetransformofcaputoderivative}, implies that
	\begin{eqnarray*}
		D_{t}^{\beta}\int_{H}\phi(x)\mu^{X}_{t}(\mathrm{d}x) = \int_{H}L_{0}\phi(x)\mu^{X}_{t}(\mathrm{d}x).
	\end{eqnarray*}
\end{proof}
The next theorem gives the familiar differential form of the FPK equation for the solution to the time-changed SDE  ~\eqref{timechangedSDE}.
\begin{thm}\label{difffracfpe}
	Suppose the conditions in Theorem~\ref{intfpeoftimeSDE} hold. If the domain of the operator $L_{0}$ defined in~\eqref{kolmogorovoperator} is a set of test functions, then the  probability measure $\mu_{t}^{X}$ induced by the solution $X(t)$ satisfies the following time-fractional PDE 
	\begin{eqnarray}\label{difffracfpeoftimesde}
		D_{t}^{\beta}\mu_{t}^{X} = L_{0}^{*}\mu_{t}^{X},
	\end{eqnarray}
	with initial condition $\mu_{0}^{X}(\mathrm{d}x) = \xi(\mathrm{d}x)$, where $L^{*}_{0}$ is the adjoint of the operator $L_{0}$ and $D_{t}^{\beta}$ denotes the Caputo fractional derivative operator.
\end{thm}
\begin{proof}
	The proof of Theorem~\ref{intfpeoftimeSDE} gives 
	\begin{eqnarray*}
		\int_{H}\phi(x)\mu^{X}_{t}(\mathrm{d}x) - \int_{H}\phi(x)\xi(\mathrm{d}x) = \int_{0}^{\infty}\int_{0}^{\tau}\int_{H}L_{0}\phi(x)\mu^{Y}_{r}(\mathrm{d}x)\mathrm{d}r f_{E_{t}}(\tau)\mathrm{d}\tau. 
	\end{eqnarray*}
	Since $\phi\in D(L_{0})$ is a test function, applying the integration by parts operator yields
	\begin{eqnarray*}
		\int_{H}\phi(x)\mu^{X}_{t}(\mathrm{d}x) - \int_{H}\phi(x)\xi(\mathrm{d}x) = \int_{H}\phi(x)\int_{0}^{\infty}\int_{0}^{\tau}L^{*}_{0}\mu^{Y}_{r}(\mathrm{d}x)\mathrm{d}r f_{E_{t}}(\tau)\mathrm{d}\tau,
	\end{eqnarray*}
	which means 
	\begin{eqnarray}\label{expressionofXmeasureininveolutionofY}
	\mu^{X}_{t}(\mathrm{d}x) - \xi(\mathrm{d}x) = \int_{0}^{\infty}\int_{0}^{\tau}L^{*}_{0}\mu^{Y}_{r}(\mathrm{d}x)\mathrm{d}r f_{E_{t}}(\tau)\mathrm{d}\tau.
	\end{eqnarray}	
Additional information on the adjoint operator, $L_{0}^{*}$, is given in ~\cite{Elworthy1974}. 

The proof of Theorem~\ref{intfpeoftimeSDE} also gives
	\begin{eqnarray}\label{expressionofXmeasureinY}
	\mu^{X}_{t}(\mathrm{d}x) = \int_{0}^{\infty}\mu_{\tau}^{Y}(\mathrm{d}x)f_{E_{t}}(\tau)\mathrm{d}\tau.
	\end{eqnarray}
	Taking the Laplace transforms in ~\eqref{expressionofXmeasureininveolutionofY} and~\eqref{expressionofXmeasureinY}           	gives 
	\begin{eqnarray*}
		s^{\beta}\mathcal{L}_{t\to s}\bigg\{\mu^{X}_{t}(\mathrm{d}x)\bigg\} - s^{\beta - 1}\xi(\mathrm{d}x) = \mathcal{L}_{t\to s}\bigg\{L_{0}^{*}\mu_{t}^{X}(\mathrm{d}x)\bigg\},
	\end{eqnarray*}
	which implies
	\begin{eqnarray*}
		D_{t}^{\beta}\mu_{t}^{X} = L_{0}^{*}\mu_{t}^{X},
	\end{eqnarray*}
	with initial condition $\mu_{0}^{X}(\mathrm{d}x) = \xi(\mathrm{d}x)$. 
\end{proof}

We now use the second approach based on duality to derive the FPK equation for the solution to the time-changed SDEs~\eqref{timechangedSDE}.

\begin{thm}\label{fpeofduality}
	Under the assumptions of Theorem ~\ref{intfpeoftimeSDE}, the time-fractional FPK equation associated with the time-changed SDE~\eqref{timechangedSDE} follows from ~\eqref{fracintfpeoftimeSDE}. Also if the domain of the operator $L_{0}$ defined in~\eqref{kolmogorovoperator} is a set of test functions, then the time-fractional FPK equation has the form~\eqref{difffracfpeoftimesde}. 
\end{thm}
\begin{proof}(\textbf{Via duality}) 
	From the duality theorem, Theorem~\ref{Duality}, the solution of the time-changed SDE~\eqref{timechangedSDE} is 
	\begin{eqnarray}\label{daualityofXandY}
	X(t) = Y(E_{t}),
	\end{eqnarray}
	where $Y(t)$ is the solution to the classic SDE~\eqref{SDEqprocess}.
	As in ~\eqref{evoloperatorofY} and \eqref{probmeasurebyY}, define the transition evolution operator, $P^{X}_{t}$, induced by the solution, $X(t)$, as follows:
	\begin{eqnarray}\label{evoloperatorofX}
	\begin{aligned}
	P_{t}^{X}\phi(x) &= E(\phi(X(t))|X(0) = x_{0}) = E(\phi(Y(E_{t}))|X(0) = x_{0})\\ &= \int_{0}^{\infty}P_{\tau}^{Y}\phi(x)f_{E_{t}}(\tau)\mathrm{d}\tau,~~~~~~\text{ }0\leq{t}\leq{T},\text{ }\phi\in{\mathcal{B}_{b}(H)}.
	\end{aligned}
	\end{eqnarray}
	The probability measure, $\mu_{t}^{X}(\mathrm{d}x)$, induced by $X(t)$ is
	\begin{eqnarray}\label{probmeasurebyX}
	\mu_{t}^{X}(\mathrm{d}x) := (P^{X}_{t})^{*}\xi(\mathrm{d}x),
	\end{eqnarray}
	which means for all~$\phi\in\mathcal{B}_{b}(H)$,
	\begin{eqnarray}\label{meaningofevoloperatorofX}
	\begin{aligned}
	\int_{H}\phi(x)\mu_{t}^{X}(dx) := \int_{H}P^{X}_{t}\phi(x)\xi(dx) &= \int_{H}\int_{0}^{\infty}P^{Y}_{\tau}\phi(x)\xi(\mathrm{d}x)f_{E_{t}}(\tau)d\tau.
	\end{aligned}
	\end{eqnarray} 
	Therefore, the connection between the probability measures, $\mu_{t}^{X}(\mathrm{d}x)$ and $\mu_{t}^{Y}(\mathrm{d}y)$, is obtained from~\eqref{meaningofevoloperatorofX} by applying Fubini's theorem
	\begin{eqnarray}\label{intrelationshipofXandY}
	\begin{aligned}
	\int_{H}\phi(x)\mu_{t}^{X}(dx) = \int_{0}^{\infty}\int_{H}\phi(x)\mu_{\tau}^{Y}(\mathrm{d}x)f_{E_{t}}(\tau)d\tau, ~~\textrm{for all}~\phi\in\mathcal{B}_{b}(H),
	\end{aligned}
	\end{eqnarray}
i.e., 
\begin{eqnarray}\label{relationshipofXandY}
   \begin{aligned}
	 \mu_{t}^{X}(dx) = \int_{0}^{\infty}\mu_{\tau}^{Y}(\mathrm{d}x)f_{E_{t}}(\tau)d\tau, ~~\textrm{for all}~\phi\in\mathcal{B}_{b}(H).
   \end{aligned}
\end{eqnarray}
Appealing to Lemma~\ref{laplaceoftimechange}, and taking Laplace transforms of both sides of~\eqref{intrelationshipofXandY} leads to
	\begin{eqnarray}\label{laplaceofinteofmeasures}
	\mathcal{L}_{t\to s}\bigg\{\int_{H}\phi(x)\mu_{t}^{X}(dx)\bigg\} = s^{\beta - 1}\mathcal{L}_{\tau\to s}\bigg\{\int_{H}\phi(x)\mu_{\tau}^{Y}(dx)\bigg\}(s^{\beta}).
	\end{eqnarray}
	On the other hand, taking Laplace transforms on both sides of~\eqref{intfpeofsde} leads to
	\begin{eqnarray}\label{laplaceofclassicfpe}
	s\mathcal{L}_{t\to s}\bigg\{\int_{H}\phi(y)\mu_{t}^{Y}(dy)\bigg\} - \int_{H}\phi(y)\xi(\mathrm{d}y) = \mathcal{L}_{t\to s}\bigg\{\int_{H}L_{0}\phi(y)\mu^{Y}_{t}(dy)\bigg\}. 
	\end{eqnarray}
	Replacing $s$ by $s^{\beta}$ in~\eqref{laplaceofclassicfpe} yields
	\begin{eqnarray}\label{laplaceofclassicfpe1}
	\begin{aligned}
	s^{\beta}\mathcal{L}_{t\to s}\bigg\{\int_{H}\phi(y)\mu_{t}^{Y}(dy)\bigg\}(s^{\beta}) &- \int_{H}\phi(y)\xi(\mathrm{d}y)\\ &= \mathcal{L}_{t\to s}\bigg\{\int_{H}L_{0}\phi(y)\mu^{Y}_{t}(dy)\bigg\}(s^{\beta}). 
	\end{aligned}
	\end{eqnarray}
	Thus, combining~\eqref{laplaceofinteofmeasures} and~\eqref{laplaceofclassicfpe1} gives
	\begin{eqnarray*}
		\begin{aligned}
			s^{\beta}\mathcal{L}_{t\to s}\bigg\{\int_{H}\phi(x)\mu_{t}^{X}(dx)\bigg\} &- s^{\beta - 1}\int_{H}\phi(y)\xi(\mathrm{d}y)\\ 
			&= s^{\beta - 1}\mathcal{L}_{t\to s}\bigg\{\int_{H}L_{0}\phi(y)\mu^{Y}_{t}(dy)\bigg\}(s^{\beta}), 
		\end{aligned}
	\end{eqnarray*}
	which implies 
	\begin{eqnarray}\label{intfpeforduality}
	\begin{aligned}
	D^{\beta}_{t}\int_{H}\phi(x)\mu_{t}^{X}(dx) &= \int_{0}^{\infty}\int_{H}L_{0}\phi(y)\mu^{Y}_{\tau}(dy)f_{E_{t}}(\tau)\mathrm{d}\tau\\
	&=\int_{H}L_{0}\phi(y)\int_{0}^{\infty}\mu^{Y}_{\tau}(dy)f_{E_{t}}(\tau)\mathrm{d}\tau\\
	&=\int_{H}L_{0}\phi(y)\mu^{X}_{t}(dy),
	\end{aligned}
	\end{eqnarray}
	as required. Further, if the domain of the Kolmogorov operator, $L_{0}$, is comprised of test functions, taking Laplace transforms on both sides of~\eqref{derfpeofsde} yields 
	\begin{eqnarray}\label{laplacefpeoftestfunction}
		s\mathcal{L}_{t\to s}\bigg\{\mu_{t}^{Y}\bigg\} - \mu_{0}^{Y} = \mathcal{L}_{t\to s}\bigg\{L^{*}_{0}\mu_{t}^{Y}\bigg\},
	\end{eqnarray}
	while taking Laplace transforms of both sides of~\eqref{relationshipofXandY} and applying Lemma~\ref{laplaceoftimechange} yields
	\begin{eqnarray}\label{laplaceofmeasures}
	\mathcal{L}_{t\to s}\bigg\{\mu_{t}^{X}\bigg\} = s^{\beta - 1}\mathcal{L}_{\tau\to s}\bigg\{\mu_{\tau}^{Y}\bigg\}(s^{\beta}).
	\end{eqnarray}
	Finally, combining~\eqref{laplacefpeoftestfunction} and~\eqref{laplaceofmeasures} yields
	\begin{eqnarray*}
		\begin{aligned}
			D^{\beta}_{t}\mu_{t}^{X} &=L^{*}_{0}\mu^{X}_{t}.
		\end{aligned}
	\end{eqnarray*}
\end{proof}

\subsection*{Acknowledgment}
The authors would like to thank our advisor, Dr. Marjorie Hahn, for her advice, fruitful discussions, and patience with this work. Also, the authors would like to thank Dr. Daniel Conus, Dr. Kei Kobyashi, and Dr. Lee Stanley for their helpful discussions.

\end{document}